\pdfoutput=1
\documentclass[paper=a4, english, final ]{scrartcl}

\usepackage{babel}

\usepackage[T1]{fontenc}
\usepackage[utf8]{inputenc}

\usepackage{lmodern}

\usepackage[final, babel ]{microtype}

\usepackage{amsfonts} \usepackage{amssymb}  \usepackage{mathtools} 

\providecommand\given{} \newcommand\SetSymbol[1][]{
   \mathrel{}\mathclose{}#1|\allowbreak\mathopen{}\mathrel{}}
\DeclarePairedDelimiterX\Set[1]{\lbrace}{\rbrace}{ \renewcommand\given{\SetSymbol[\delimsize]} #1 }

\DeclarePairedDelimiter\abs{\lvert}{\rvert}\DeclarePairedDelimiter\norm{\lVert}{\rVert}

\makeatletter
\let\oldabs\abs
\def\abs{\@ifstar{\oldabs}{\oldabs*}}
\let\oldnorm\norm
\def\norm{\@ifstar{\oldnorm}{\oldnorm*}}

\let\oldSet\Set
\def\Set{\@ifstar{\oldSet}{\oldSet*}}
\makeatother

\newcommand{\MoveEqLeftR}{\MoveEqLeft[3.3]}
 \mathtoolsset{showonlyrefs}

\usepackage{chngcntr}

\usepackage{amsthm} \theoremstyle{plain}

\newtheorem{prop}[equation]{Proposition}

\newtheorem*{thm*}{Theorem}
\newtheorem*{prop*}{Proposition}

\newtheorem*{principle*}{Principle}

\theoremstyle{definition}

\newtheorem{lemma}[equation]{Lemma}
\newtheorem{defn}[equation]{Definition}

\newtheorem*{cor*}{Corollary}
\newtheorem*{lemma*}{Lemma}
\newtheorem*{defn*}{Definition}

\theoremstyle{remark}
\newtheorem{rem}[equation]{Remark}
\newtheorem{ex}[equation]{Example}

\newtheorem*{rem*}{Remark}
\newtheorem*{ex*}{Example} 
\theoremstyle{definition}
\newtheorem{axiom}[equation]{Axiom}
\newtheorem{conjecture}[equation]{Conjecture}

\usepackage{thmtools, thm-restate}

\usepackage{tikz-cd} \tikzset{
	pf/.style={commutative diagrams/.cd, every arrow, every label},
	surj/.style=commutative diagrams/two heads,
	inj/.style=commutative diagrams/hook,
	gl/.style=commutative diagrams/equal,
	mat/.style={matrix of math nodes, commutative diagrams/.cd, every cell},
	dr/.style={matrix of math nodes, commutative diagrams/.cd, every cell, column sep=small},
	seq/.style={matrix of math nodes, commutative diagrams/.cd, every cell, column sep=small}
	}
\newenvironment{diag*}{\[\begin{tikzpicture}[commutative diagrams/.cd, every diagram, baseline=(current bounding box.center)]}{\end{tikzpicture}\]\ignorespacesafterend}
\newenvironment{diag}{\begin{equation}\begin{tikzpicture}[commutative diagrams/.cd, every diagram, baseline=(current bounding box.center)]}{\end{tikzpicture}\end{equation}\ignorespacesafterend}
 \newenvironment{Tgraph}{\begin{equation}\begin{tikzpicture}[node distance={20mm}, vertex/.style = {draw, circle}]}{\end{tikzpicture}\end{equation}\ignorespacesafterend}

\usepackage{csquotes} \usepackage[style=authoryear, date=year,
	isbn=false,
	]{biblatex}
\DeclareSourcemap{
	\maps[datatype=bibtex]{
\map[overwrite]{
			\step[fieldsource=author,
				match = {William Fulton|W. Fulton},
				replace = {William Fulton}
			]
			\step[fieldsource=author,
				match = {M. Kontsevich|Maxim Kontsevich},
				replace = {Maxim Kontsevich}
			]
			\step[fieldsource=author,
				match = {Yu. Manin|Yuri Manin|Yu. I. Manin|Y. Manin|Yuri I. Manin},
				replace = {Yuri I. Manin}
			]
			\step[fieldsource=author,
				match = {Alexander Voronov|A. Voronov|A. A. Voronov},
				replace = {Alexander Voronov}
			]
		}
\map[overwrite]{
			\step[fieldset=pagetotal, null]
			\step[fieldset=translator, null]
			\step[fieldset=origlanguage, null]
			\step[fieldset=eventtitle, null]
			\step[fieldset=eventdate, null]
			\step[fieldset=venue, null]
			\step[fieldset=edition, null]
			\step[fieldset=pubstate, null]
		}
	}
}
\author{\texorpdfstring{Enno Keßler \and Artan Sheshmani \and Shing-Tung Yau}{
Enno Keßler, Artan Sheshmani, Shing-Tung Yau
}
}
\title{Super Gromov--Witten Invariants via torus localization}
\date{}

\usepackage[final, pdfusetitle ]{hyperref}
\hypersetup{colorlinks=false}  

\counterwithin{equation}{subsection}

\usepackage{faktor}

\usepackage{slashed}
\newcommand{\Dirac}{\slashed{D}}

\usepackage{enumitem}
\setlist[enumerate,1]{label = (\roman*)}

\DeclareMathOperator{\ACI}{I}

\DeclareMathOperator{\Aut}{Aut}

\DeclareMathOperator{\Hom}{Hom}
\DeclareMathOperator{\im}{im}
\DeclareMathOperator{\id}{id}

\DeclareMathOperator{\Spec}{Spec}

\DeclareMathOperator{\SGL}{SL}

\newcommand{\dual}[1]{{#1}^{\vee}}

\newcommand{\cat}[1]{\mathsf{#1}}

\renewcommand{\d}{\mathop{}\!d}

\newcommand{\pt}{pt}

\newcommand{\VSec}[2][]{\Gamma_{#1}\left(#2\right)}

\newcommand{\tangent}[2][]{T_{#1}#2} \newcommand{\differential}[1]{\d{#1}} \newcommand{\cotangent}[1]{\dual{T}#1}

\newcommand{\Integers}{\mathbb{Z}}
\newcommand{\Z}{\Integers}
\newcommand{\RationalNumbers}{\mathbb{Q}}
\newcommand{\RealNumbers}{\mathbb{R}}
\newcommand{\R}{\RealNumbers}
\newcommand{\ComplexNumbers}{\mathbb{C}}
\newcommand{\C}{\ComplexNumbers}

\newcommand{\ProjectiveSpace}[2][]{\mathbb{P}_{#1}^{#2}}

\newcommand{\cD}{\mathcal{D}}
\newcommand{\cO}{\mathcal{O}}

\newcommand{\targetACI}{J}
\newcommand{\DJBar}{\overline{D}_\targetACI}
\newcommand{\DelJBar}{\overline{\partial}_\targetACI}
\DeclareMathOperator{\rk}{rk}
\newcommand{\ev}{\mathrm{ev}}
\newcommand{\Ev}{\mathrm{Ev}}
\newcommand{\SymG}[1]{\operatorname{Sym}(#1)}

\newcommand{\ModuliStackCurves}[1]{\overline{M}_{#1}}
\newcommand{\ModuliSpaceCurves}[1]{\overline{M}^{c}_{#1}}
\newcommand{\SmModuliSpaceCurves}[1]{M^{sm}_{#1}}

\newcommand{\UniversalCurve}[1]{C_{#1}}

\newcommand{\ModuliStackMaps}[2]{\overline{M}_{#1}\left(#2\right)}
\newcommand{\SmModuliSpaceMaps}[2]{M^{sm}_{#1}\left(#2\right)}
\newcommand{\SmCModuliSpaceMaps}[2]{\overline{M}^{sm}_{#1}\left(#2\right)}
\newcommand{\SmModuliSpaceJMaps}[1]{M^{sm}\left(#1\right)}
\newcommand{\UniversalMap}[2]{C_{#1}\left(#2\right)}

\newcommand{\ModuliStackSpinCurves}[1]{\overline{M}^{spin}_{#1}}
\newcommand{\UniversalSpinCurve}[1]{C^{spin}_{#1}}
\newcommand{\UniversalSpinorBundleC}{\mathcal{S}}

\newcommand{\ModuliStackSpinMaps}[2]{\overline{M}^{spin}_{#1}\left(#2\right)}
\newcommand{\UniversalSpinMap}[2]{C^{spin}_{#1}\left(#2\right)}
\newcommand{\UniversalSpinorBundleM}{\mathfrak{S}}

\newcommand{\ModuliStackSuperCurves}[1]{\overline{\mathcal{M}}_{#1}}
\newcommand{\SmModuliSpaceSuperCurves}[1]{\mathcal{M}^{sm}_{#1}}

\newcommand{\ModuliStackSuperMaps}[2]{\overline{\mathcal{M}}_{#1}\left(#2\right)}
\newcommand{\SmModuliSpaceSuperMaps}[2]{\mathcal{M}^{sm}_{#1}\left(#2\right)}
\newcommand{\SmModuliSpaceSuperJMaps}[1]{\mathcal{M}^{sm}\left(#1\right)}
\newcommand{\SmCModuliSpaceSuperMaps}[2]{\overline{\mathcal{M}}^{sm}_{#1}\left(#2\right)}

\begin{document}
\maketitle

\begin{abstract}
	In this article we propose a definition of super Gromov--Witten invariants by postulating a torus localization property for the odd directions of the moduli spaces of super stable maps and super stable curves of genus zero.
	That is, we define super Gromov--Witten invariants as the integral over the pullback of homology classes along the evaluation maps divided by the equivariant Euler class of the normal bundle of the embedding of the moduli space of stable spin maps into the moduli space of super stable maps.
	This definition sidesteps the difficulties of defining a supergeometric intersection theory and works with classical intersection theory only.
	The properties of the normal bundles, known from the differential geometric construction of the moduli space of super stable maps, imply that super Gromov--Witten invariants satisfy a generalization of Kontsevich--Manin axioms and allow for the construction of a super small quantum cohomology ring.
	We describe a method to calculate super Gromov--Witten invariants of \(\ProjectiveSpace{n}\) of genus zero by a further geometric torus localization and give explicit numbers in degree one when dimension and number of marked points are small.
\end{abstract}

\counterwithin{equation}{section}
\section{Introduction}
In this article we give an algebro-geometric proposal for a supergeometric generalization of Gromov--Witten invariants of genus zero based on our earlier differential geometric work on moduli spaces of super stable maps of genus zero in~\cites{KSY-SJC}{KSY-SQCI}{KSY-TAMSSSMGZ}.

Classically Gromov--Witten invariants can roughly be described as the count of stable maps from surfaces to a fixed almost Kähler manifold satisfying a list of topological constraints.
The idea was introduced as a tool to study symplectic manifolds in~\cite{G-PHCSM} and with motivation in theoretical physics in~\cite{W-TDGISTMC}.
For an introduction to the Gromov--Witten theory from the perspective of symplectic geometry we refer to~\cite{McDS-JHCST}.
In the algebro-geometric interpretation, Gromov--Witten invariants associate to homology classes of the target variety a number via intersection theory on the moduli stacks of stable maps.
Gromov--Witten invariants satisfy a list of axioms spelled out for the first time in~\cite{KM-GWCQCEG}.
Those axioms, which we will call Kontsevich--Manin axioms have turned out to be a powerful tool for the calculation of Gromov--Witten invariants, most importantly the gluing or splitting axiom that allows to reduce to situations of fewer marked points or lower genus.
More generally, the Kontsevich--Manin axioms turn the Gromov--Witten invariants into the prime example of a cohomological field theory and allow for the construction of the quantum product on the homology of the target.
For an introduction to the algebro-geometric perspective of stable maps and  Gromov--Witten invariants we refer to~\cites{M-FMQCMS}{CK-MSAG}.

Here, we propose super Gromov--Witten invariants of genus zero which satisfy generalized Kontsevich--Manin axioms, allow for the construction of super quantum cohomology that extends quantum cohomology and are motivated by super stable maps.
Super stable maps are maps from prestable marked super curves into a target manifold or variety.
That is, the domain curve has one classical even function and one anti-commuting function as local coordinates which satisfy a non-integrability or supersymmetry condition and at most nodal singularities.
Smooth super curves are also known as super Riemann surfaces and have been studied from many different perspectives and motivated from super string theory, see~\cite{EK-SGSRSSCA} for references.
We have studied the moduli spaces of super stable maps of genus zero from the perspective of super differential geometry and for almost Kähler manifolds as target in~\cites{KSY-SQCI}{KSY-TAMSSSMGZ}.

Ideally, one would like to construct algebro-geometric super Gromov--Witten invariants by constructing the super moduli stack of super stable maps as well as a suitable intersection theory on super stacks and apply the classical formulas from Gromov--Witten theory.
We believe, that the recent progress in the understanding of moduli spaces of stable super curves in~\cites{FKP-MSSCCLB}{MZ-EHSTIISFTV}{OV-SMSGZSUSYCRP}{BR-SMSCRP} can, in principle, be generalized to construct super moduli stacks of super stable maps.
However, unfortunately, it is not clear to us what the supergeometric intersection theory or even Chow groups or cohomology theories for super schemes will look like.
We list some attempts for supergeometric intersection theory or cohomology theories in Section~\ref{SSec:GWMotivation}.
Our proposal for super Gromov--Witten invariants sidesteps the question of constructing the super moduli stacks as well as the supergeometric intersection theory and is motivated by the following two assumptions:
\begin{itemize}
	\item
		There is a torus action on the super moduli stacks of super stable maps that leaves the embedded moduli stack of classical stable spin maps invariant.
	\item
		A well defined supergeometric intersection theory would have a torus localization theorem.
\end{itemize}
We use the resulting formula from the assumed supergeometric torus localization theorem as an ad-hoc definition of super Gromov--Witten invariants:
Super Gromov--Witten invariants are integrals over the pullback of homology classes from the target along the evaluation maps divided by the equivariant Euler class of the normal bundle to the embedding of the moduli space of stable spin maps into the super moduli space of super stable maps.
Hereby, integrals are taken over the classical moduli spaces of stable spin maps and the normal bundles are obtained by transferring the differential geometric description of the normal bundles to algebraic geometry.
Even though we have shown that appropriate torus actions exist on moduli spaces of super stable maps of genus zero in certain cases in~\cite{KSY-TAMSSSMGZ}, we are aware that the above assumptions are too strong to be correct in full generality.
In any case, our proposal for super Gromov--Witten invariants yields a generalization of classical Gromov--Witten invariants that satisfies interesting generalizations of the Kontsevich--Manin axioms and take some geometric information of the super moduli space of super stable maps into account.

To be more precise, let us first describe the SUSY normal bundles in more detail:
Let \(c\colon C=\ProjectiveSpace[\C]{1}\times B\to B\) be the trivial family of smooth curves of genus zero with sections \(p_i\colon B\to C\) for \(i=1,\dotsc, k\) and a map \(\phi\colon C\to X\) to a convex target scheme \(X\).
Denote by \(S_C\) the extension of the spinor bundle \(S=\cO(1)\) to \(C\).
Then the SUSY normal bundle~\(N_{(C, \phi)}\) is a locally free sheaf on \(B\) given by the short exact sequence
\begin{diag}
	\matrix[mat](m) {
		0 & c_*S_C & {\displaystyle\bigoplus_{i=1}^k p_i^*S_{C} \oplus c_*\left(\dual{S}_C\otimes \phi_*\tangent{X}\right)} & N_{(C, \phi)} & 0.\\
	} ;
	\path[pf] {
		(m-1-1) edge (m-1-2)
		(m-1-2) edge (m-1-3)
		(m-1-3) edge (m-1-4)
		(m-1-4) edge (m-1-5)
	};
\end{diag}
This short exact sequence is what we obtained in~\cite{KSY-TAMSSSMGZ} for the differential geometric description of the normal bundle of the inclusion \(\SmModuliSpaceMaps{0,k}{X, \beta} \to \SmModuliSpaceSuperMaps{0,k}{X, \beta}\) of superorbifolds from the moduli space of stable maps of genus zero with \(k\) marked points into its super analog.
The construction of \(N_{(C, \phi)}\) is compatible with standard operations from the theory of algebraic curves such as gluing or forgetting marked points.
Consequently, we obtain a description of SUSY normal bundles for the moduli spaces of stable maps of fixed tree type.
Every closed point of the moduli stack of stable spin maps~\(\ModuliStackSpinMaps{0,k}{X,\beta}\) describes a stable spin map of fixed tree type and hence we expect the SUSY normal bundle of the universal curve over \(M^{spin}_{0,k}(X, \beta)\) to extend to a vector bundle \(\overline{N}_{k, \beta}\) on \(\ModuliStackSpinMaps{0,k}{X,\beta}\).

We now define genus zero \(k\)-point super Gromov--Witten invariants of \(X\) as
\begin{equation}\label{eq:defnSGWIntro}
	\left<SGW_{0,k}^{X,\beta}\right>(\alpha_1, \dotsc, \alpha_k)
	= \int_{\ModuliStackSpinMaps{0,k}{X,\beta}}\frac{\ev_1^*\alpha_1\cup\dotsm\cup \ev_k^*\alpha_k}{2e^K(\overline{N}_{k, \beta})}.
\end{equation}
The classes \(\ev_i^*\alpha_i\) are the pullback along the evaluation maps \(\ev_i\colon \ModuliStackMaps{0,k}{X,\beta}\to X\) of cohomology classes \(\alpha_i\in H^*(X)\).
The equivariant Euler class \(e^K(\overline{N}_{k, \beta})\) is an invertible element of \(H^*(\ModuliStackMaps{0,k}{X,\beta})\otimes \C(\kappa)\) where \(\kappa\) is the character of the torus \(K=\C^*\).
Note that the forgetful map \(F\colon \ModuliStackSpinMaps{0,k}{X,\beta}\to \ModuliStackMaps{0,k}{X,\beta}\) induces an isomorphism on rational cohomology.
The term of cohomological degree zero of \({\left(e^K(\overline{N}_{k, \beta})\right)}^{-1}\) is \(\kappa^{-\rk \overline{N}_{k, \beta}}\) and consequently if \(\sum_i \deg \alpha_i = 2\dim \ModuliStackMaps{0,k}{X,\beta}\) the super Gromov--Witten invariants reproduce classical Gromov--Witten invariants up to the polynomial prefactor.
New invariants arise when \(c=2\dim \ModuliStackMaps{0,k}{X,\beta} - \sum_i\deg \alpha_i\) is strictly positive and the component of \({\left(e^K(\overline{N}_{k, \beta})\right)}^{-1}\) of cohomological degree \(c\) contributes characteristic classes of \(\overline{N}_{k,\beta}\).
Some but not all characteristic classes of \(\overline{N}_{k,\beta}\) can be expressed as descendant classes and hence we do expect that super Gromov--Witten invariants are new invariants and cannot be directly reduced to known ones.

We point out that even though Equation~\eqref{eq:defnSGWIntro} is clearly motivated by torus localization the invariants only depend on the existence and properties of the normal bundles \(\overline{N}_{k,\beta}\) on \(\ModuliStackSpinMaps{0,k}{X,\beta}\).
Super Gromov--Witten classes can be defined in the same spirit and are treated in this article.
Further generalizations motivated by the geometry of super stable maps and super stable curves are briefly discussed in Section~\ref{SSec:Generalizations}.

Super Gromov--Witten invariants and super Gromov--Witten classes satisfy generalizations of the Kontsevich--Manin axioms because the vector bundles \(\overline{N}_{k, \beta}\) are compatible with forgetful and gluing maps.
For example, a simple version of the splitting axiom for super Gromov--Witten classes states that the integral of the super Gromov--Witten class \(SGW_{0,4}^{X, \beta_1+\beta_2}(\alpha_1, \alpha_2, \alpha_3, \alpha_4)\) over the boundary divisor \(D=[\ModuliStackMaps{0,3}{X, \beta_1}\times_X \ModuliStackMaps{0,3}{X,\beta_2}]\) can be be calculated by summing over products of three point Gromov--Witten invariants:
\begin{equation}\label{eq:GluingAxiomIntro}
	\begin{split}
		\MoveEqLeft
		2\int_D SGW_{0,4}^{X, \beta_1+\beta_2}(\alpha_1, \alpha_2, \alpha_3, \alpha_4) \\
		&=\sum_{a,b} g^{ab}\left<SGW_{0,3}^{X, \beta_1}\right>(\alpha_1, \alpha_2, T_a)\left<SGW_{0,3}^{X, \beta_2}\right>(T_b, \alpha_3, \alpha_4).
	\end{split}
\end{equation}
Here, \(T_a\) is a homogeneous basis of the cohomology of \(X\) and \(g_{ab}=\int_X T_a\cup T_b\).
The identity~\eqref{eq:GluingAxiomIntro} is central to the proof that three point super Gromov--Witten invariants allow to define a supergeometric generalization of the small quantum product in cohomology.
The proof of the splitting axiom for super Gromov--Witten invariants reduces to the proof of the splitting axiom for the classical Gromov--Witten invariants together with \(gl_X^*\overline{N}_{k_1+k_2,\beta_1+\beta_2} = \overline{N}_{k_1+1,\beta_1}\oplus \overline{N}_{k_2+1, \beta_2}\) for the gluing map
\begin{equation}
	gl_X\colon \ModuliStackSpinMaps{0,k_1+1}{X, \beta_1}\times_X \ModuliStackSpinMaps{0,k_2+1}{X, \beta_2}\to \ModuliStackSpinMaps{0,k_1+k_2}{X, \beta_1+\beta_2}.
\end{equation}

The Point-Mapping-Axiom states that for \(\beta=0\) and \(k\geq 3\)
\begin{equation}
	\left<SGW^{X, 0}_{0,k}\right>(\alpha_1, \dotsc, \alpha_k)
	= \left<SGW^{\pt}_{0,k}\right> \int_X \alpha_1 \cup \dotsm \cup \alpha_k.
\end{equation}
Here, the super Gromov--Witten invariants of a point \(\left<SGW^{\pt}_{0,k}\right>\) are monomials in \(\C(\kappa)\), independent of \(X\) which can be expressed in terms of Chern classes of Hodge bundles on \(\ModuliStackCurves{0,k}\).
The supergeometric divisor and fundamental class axiom acquire a correction term encoding the first Chern class of the kernel of \(\overline{N}_{k+1, \beta}\to \overline{N}_{k, \beta}\).

For the target \(X=\ProjectiveSpace{n}\) we sketch a general approach to evaluate all super Gromov--Witten invariants via an additional localization with respect to the torus action on the moduli of stable maps induced from the \({\left(\C^*\right)}^{n+1}\)-action on \(\ProjectiveSpace{n}\).
The calculation is based on~\cite{K-ERCVTA} where the stable maps invariant under the torus action where classified, Euler classes of normal bundles calculated and classical Gromov--Witten invariants of \(\ProjectiveSpace{n}\) obtained.
We extend Kontsevich's method by computing the torus action on \(\overline{N}_{k,\beta}\) induced from the torus action on \(\ProjectiveSpace{n}\) for particular fixed points.
This leads to formulas for super Gromov--Witten invariants \(\left<SGW_{0,k}^{\ProjectiveSpace{n}, d}\right>(\alpha_1, \dotsc, \alpha_k)\) of degree \(d=1\) which can be evaluated for \(n=1\) and \(k=1,2,3\) directly.
Certain cases for \(n=2,3,4,5\) are obtained with the help of computer algebra.
Cases with higher degree and higher number of marked points can then, in principle be calculated using the axioms and the recursive methods from~\cite{K-ERCVTA}.

The outline of the paper is as follows:
In Section~\ref{Sec:SUSYNormalBundles}, we define and motivate the SUSY normal bundles on the moduli spaces of stable curves and stable maps.
Motivated by torus localization we define super Gromov--Witten invariants dividing by the equivariant Euler class of the SUSY normal bundles in Section~\ref{Sec:DefnSGW}.
We study their basic properties, in particular that they satisfy a generalization of the Kontsevich--Manin axioms which allows to construct a super quantum product in Section~\ref{Sec:SuperSmallQuantumCohomology}.
As an example, we calculate the  super Gromov--Witten invariants of degree one in \(\ProjectiveSpace{n}\) in Section~\ref{Sec:SGWOfPn} when the number of dimensions and marked points is low.

\subsection*{Acknowledgments}
We thank
Nitin Chidambaram,
Kai Cieliebak,
Olivia Dumitrescu,
Sam Grushevsky,
Wei Gu,
Sheldon Katz,
Reinier Kramer,
Motohico Mulase,
Sasha Polishchuk,
Maxim Smirnov,
Tom Wennink, and
Dimitry Zvonkine
for helpful discussions on various aspects of this work.
Enno Keßler thanks the Max-Planck-Institut für Mathematik, Bonn as well as the Max-Planck-Institut für Mathematik in den Naturwissenschaften, Leipzig for good working conditions and financial support during this work.
Enno Keßler and Artan Sheshmani thank the Simons Center for Geometry and Physics at Stony Brook for the invitation to the program on SuperGeometry and SuperModuli where this work was presented.
Enno Keßler also thanks the Galileo Galilei Institute for Theoretical Physics in Arcetri, Florence for hospitality during the workshop on Emergent Geometries from Strings and Quantum Fields.
 \counterwithin{equation}{subsection}

\section{SUSY normal bundles}\label{Sec:SUSYNormalBundles}
In this section we describe the SUSY normal bundles for stable curves and stable maps.
Intuitively, those bundles should be thought of as follows:
Let \(\overline{N}_k\) be the normal bundle of the embedding \(\ModuliStackSpinCurves{0,k}\to \ModuliStackSuperCurves{0,k}\) of the moduli stack of stable curves of genus zero and with \(k\)-marked points into the moduli stack of stable SUSY curves of genus zero with \(k\) Neveu--Schwarz punctures.
Every family \(C\to B\) of such stable spin curves is given by a map \(b\colon B\to \ModuliStackSpinCurves{0,k}\) and we obtain the SUSY normal bundle \(N_C=b^*\overline{N}_k\).
Similarly, let \(\overline{N}_{k,\beta}\) be the normal bundle of the embedding \(\ModuliStackSpinMaps{0,k}{X,\beta}\to \ModuliStackSuperMaps{0,k}{X, \beta}\) of the moduli space of stable spin maps of genus zero, \(k\) marked points and homology class \(\beta\in H_2^+(X)\) into its supergeometric analogue.
Every such family \(\phi\colon C\to X\) of stable spin maps over \(B\) is given by a map \(b\colon B\to \ModuliStackSpinMaps{0,k}{X,\beta}\) and set the SUSY normal bundle to be \(N_{(C, \phi)}=b^*\overline{N}_{k,\beta}\).
As the moduli space of super stable maps is not yet sufficiently developed, we do not construct \(\overline{N}_k\) and \(\overline{N}_{k,\beta}\) directly.
Rather we give a definition of \(N_C\) and \(N_{(C,\phi)}\) for the case of constant tree type using only the language of classical algebraic geometry and the motivation from~\cite{KSY-TAMSSSMGZ}.
We leave the extension of the SUSY normal bundles to the moduli stacks of stable spin curves and stable spin maps as a conjecture.

In Section~\ref{SSec:SUSYNormalBundlesForStableCurves}, we first recall the notions of \(k\)-pointed prestable curves and stable curves of genus zero, the notion of dual tree and the spinor bundles on prestable curves.
Then we define the SUSY normal bundles for stable \(k\)-pointed curves of genus zero of fixed tree type and derive some basic properties.
In Section~\ref{SSec:ConjectureSUSYNormalBundleStableSpinCurves} we formulate the conjecture that the SUSY normal bundles form a vector bundle on the moduli stack of stable spin curves.
SUSY normal bundles for stable maps are then treated analogously in Section~\ref{SSec:SUSYNormalBundlesForStableMaps} and Section~\ref{SSec:ConjectureSUSYNormalBundleStableSpinMaps}.
In Section~\ref{SSec:SNMotivation} we give background information and from the differential geometric treatment of moduli spaces of super stable maps in~\cite{KSY-TAMSSSMGZ}.

\subsection{SUSY normal bundles for stable spin curves}\label{SSec:SUSYNormalBundlesForStableCurves}
Before studying SUSY normal bundles for stable spin curves we recall some basic properties of prestable curves of genus zero:
A prestable curve over the base \(B\) is a flat, proper morphism \(c\colon C\to B\) such that each geometric fiber is a reduced one-dimensional scheme with at most nodal singularities, see~\cite[III Definition~2.1]{M-FMQCMS}.
We assume in this article that all curves are of genus zero, that is all geometric fibers of the family \(C\to B\) have arithmetic genus zero.
Any smooth prestable curve of genus zero (over a point) is isomorphic to \(\ProjectiveSpace{1}\) together with marked points \(p_i\in\ProjectiveSpace{1}\) for \(i=1,\dotsc,k\).
The normalization of any prestable curve over \(B=\Spec\C\) with \(n\) nodes is isomorphic to \(n+1\) copies of \(\ProjectiveSpace{1}\).

Let \(J\) be a finite set of cardinality \(k\).
Often we will use \(J=\Set{1,\dots,k}\).
A \(J\)-marked prestable curve is a prestable curve together with sections \(p_j\colon B\to C\) for \(j\in J\) such that in each geometric fiber the images of the sections are distinct points lying in the smooth locus.
A \(J\)-marked prestable curve of genus zero is called a stable curve of genus zero if every irreducible component of \(C\) contains at least three special, that is marked or singular points.

The dual tree of a prestable curve \(C\to \Spec \C\) of genus zero is a graph \(\tau\) with set of vertices~\(V_\tau\) given by irreducible components of \(C\), the set of edges~\(E_\tau\) representing nodes and the set of tails~\(T_\tau\) representing marked points of the curve, see~\cite[III Definition~2.5]{M-FMQCMS}.
We say that a \(J\)-marked prestable curve \(C\to B\) is of constant tree type \(\tau\) if every geometric fiber has dual tree \(\tau\).
In this case, we also have nodal sections \(p_e\colon B\to C\) for every edge \(e \in E_\tau\) of \(\tau\) which map in every geometric fiber to the nodal point represented by \(e\).

For a prestable curve \(c\colon C\to B\) we denote the relative dualizing sheaf by \(\omega_{C}\).
A spinor bundle \(S_C\) on the prestable curve \(c\colon C\to B\) is a coherent sheaf \(S_C\) on \(C\), generically locally free of rank one such that
\begin{equation}
	S_{C}=\Hom(S_{C}, \dual{\omega}_C).
\end{equation}
Similar bundles have been considered in~\cites{C-MCTC}{J-GMHSC} as well as~\cites{D-LaM}{FKP-MSSCCLB} in the context of supergeometry.
We call a prestable curve together with a choice of spinor bundle a prestable spin curve.
In genus zero, there is a unique spinor bundle on every prestable curve \(C\).
If \(C=\ProjectiveSpace{1}\) is a smooth curve \(S_C=\cO(1)\) and if \(C\) is a prestable curve over \(\Spec\C\) and \(\tilde{c}\colon \tilde{C}\to C\) its normalization then \(S_C = \tilde{c}_*S_{\tilde{C}}\).
Note that \(S_C\) is locally free of rank one on the smooth locus of \(C\) but not locally free and of rank two at nodes of \(C\).
In addition over the smooth locus the spinor bundle satisfies \(S_C\otimes S_C= \tangent{C}\).

A morphism
\begin{equation}
	\left(c\colon C\to B, \Set{p_j}_{j\in J}, S_C\right)
	\to
	\left(c'\colon C'\to B, \Set{p'_{j'}}_{j'\in J'}, S_{C'}\right)
\end{equation}
of prestable spin curves over \(B\) consists of a bijective map \(l\colon J\to J'\), a morphism \(g\colon C\to C'\) over \(B\) and a fiberwise linear map \(s\colon S_C\to g^*S_{C'}\) such that \(g(p_j) = p'_{l(j)}\) and \(s\otimes s = \differential{g}\) on the smooth locus.

The group of spin automorphisms of \(\ProjectiveSpace{1}\) with no marked points is the group \(\SGL(2)\) as we have discussed in~\cites[Section~3.1]{KSY-SQCI}[Section~3]{KSY-TAMSSSMGZ} in detail.
Matrices \(l=\begin{psmallmatrix}a & c\\ b&d\\\end{psmallmatrix}\) in \(\SGL(2)\) act on projective coordinates by left multiplication
\begin{equation}
	\left[ X^1 : X^2\right]
	\begin{pmatrix}
		a & c\\
		b & d\\
	\end{pmatrix}
	=
	\left[a X^1 + b X^2 : c X^1 + d X^2\right]
\end{equation}
and on sections of \(\cO(1)\) as follows
\begin{equation}
	t_1 X^1 + t_2 X^2 \mapsto (-a t_1 + c t_2) X^1 + (b t_1 - d t_2) X^2.
\end{equation}
Note that while \(\pm l\) induce the same action on \(\ProjectiveSpace{1}\) their action on sections of \(\cO(1)\) differs by a sign.
If \(C\) is a smooth stable spin curve, that is it has at least three marked points, the automorphism group of the smooth stable spin curve \(C\) is \(\Z_2\) which acts only on the spinor bundle by flipping a sign.
The automorphisms group of stable nodal spin curves of genus zero and fixed tree type \(\tau\) is \(\Z_2^{\#V_\tau}\) where \(\#V_\tau\) is the number of irreducible components.

\begin{defn}\label{defn:NC}
	Let \(c\colon C\to B\) be a stable spin curve of constant tree type \(\tau\).
	The SUSY normal bundle is the coherent sheaf \(N_{C}\) on \(B\) given by
	\begin{diag}\label{SES:DefnNCB}
		\matrix[mat](m) {
			0 & c_*S_C & {\displaystyle \bigoplus_{t\in T_\tau} p_t^*S_C \bigoplus_{e\in E_\tau} p_e^*S_C} & N_C & 0 \\
				& s & \bigoplus p_t^*s \bigoplus p_e^*s & & \\
		} ;
		\path[pf] {
			(m-1-1) edge (m-1-2)
			(m-1-2) edge (m-1-3)
			(m-1-3) edge (m-1-4)
			(m-1-4) edge (m-1-5)
			(m-2-2) edge[commutative diagrams/mapsto] (m-2-3)
		};
	\end{diag}
\end{defn}
Any isomorphism of stable spin curves yields an isomorphism of SUSY normal bundles because the restriction map behaves equivariantly with respect to the identification of spinor bundles.

\begin{lemma}\label{lemma:NCBLocallyFree}
	The sheaf \(N_C\) is locally free of rank \(r_k=k-2\).
\end{lemma}
\begin{proof}
	By construction the sheaf \(N_C\) is coherent.
	Assume that the tree \(\tau\) has \(\#E_\tau\) edges and \(\#E_\tau+1\) vertices.
	At every closed point \(b\in B\) the geometric fiber of \(c\) has \(\#E_\tau+1\) irreducible components and \(\#E_\tau\) nodes.
	Then \(c_*S_{C}\) has rank \(2\left(\#E_\tau + 1\right)\) at \(b\) and the middle term of~\eqref{SES:DefnNCB} has rank \(2\#E_\tau+k\).
	Consequently, \(N_{C}\) has rank \(k-2\) at all closed points.
	As closed points are dense in \(B\) and the rank is upper semicontinuous, the rank of \(N_{C}\) is \(k-2\) at all points.
	It follows from Nakayama's Lemma that any coherent sheaf with constant rank is locally free, see, for example~\cite[Exercise~14.4.L(a)]{V-RSFAG}.
\end{proof}

\begin{ex}\label{ex:Nk}
	Let \(c\colon C=\ProjectiveSpace{1}\times B\to B\) and \(p_i\colon B\to \ProjectiveSpace{1}\) for \(i=1,\dotsc, k\) marked points.
	In this situation \(S_C = \cO(1)\) and we can assume without loss of generality that \(p_1=0\), \(p_2=\infty\) and \(p_3=1\).
	Hence \(p_1^*S_C\), \(p_2^*S_C\) and \(p_3^*S_C\) are trivial line bundles over \(B\).
	Sections of \(S_C\) are given by \(s=a_0X^0+a_1X^1\) where \(a_0,a_1\) are sections of \(\cO_B\) and \([X^0:X^1]\) are projective coordinates of \(\ProjectiveSpace{1}\).
	Restricting \(s\) to \(p_1=0=[0:1]\) yields \(a_0\) and the restriction of \(s\) to \(p_2=\infty=[1:0]\) yields \(a_1\).
	Consequently,
	\begin{equation}
		N_C
		= p_3^*S_C\oplus p_4^*S_C\dotsb\oplus p_k^*S_C
		= \C \oplus p_4^*S_C\dotsb\oplus p_k^*S_C.
	\end{equation}
\end{ex}

\begin{lemma}
	Let \(C\) be a smooth stable spin curve over a point with \(k\geq 3\) marked points \(p_1, \dotsc, p_k\).
	Furthermore, let
	\begin{equation}
		S_0 = S_C\left(-\sum_{j=1}^k p_j\right)
	\end{equation}
	the sheaf of sections of \(S\) vanishing at the marked points.
	Then \(N_C = H^1(S_0)\).
\end{lemma}
\begin{proof}
	The sheaves \(S_0\) and \(S_C\) fit into the following short exact sequence
	\begin{diag}
		\matrix[mat](m) {
			0 & S_0 & S_C & \bigoplus_{j=1}^k \cO_{p_j} & 0 \\
		} ;
		\path[pf] {
			(m-1-1) edge (m-1-2)
			(m-1-2) edge (m-1-3)
			(m-1-3) edge (m-1-4)
			(m-1-4) edge (m-1-5)
		};
	\end{diag}
	where \(\cO_{p_j}\) is the skyscraper sheaf at \(p_j\).
	Its long exact cohomology sequence yields
	\begin{diag}
		\matrix[mat](m) {
			0 & H^0(S_C) & H^0\left(\bigoplus_{j=1}^k \cO_{p_j}\right) & H^1(S_0) & 0 \\
		} ;
		\path[pf] {
			(m-1-1) edge (m-1-2)
			(m-1-2) edge (m-1-3)
			(m-1-3) edge (m-1-4)
			(m-1-4) edge (m-1-5)
		};
	\end{diag}
	because \(H^0(S_0)=0\) and \(H^1(S_C)=0\) for degree reasons.
	The map on the left hand side is the restriction of sections of \(S_C\) to the marked points, that is, coincides with the map on the left of the short exact sequence~\eqref{SES:DefnNCB}.
	The result \(N_C = H^1(S_0)\) follows.
\end{proof}

\begin{lemma}\label{lemma:PullbackNC}
	Let \(b\colon \tilde{B}\to B\) be a flat morphism.
	Then \(N_{C\times_B \tilde{B}}=b^*N_{C}\).
\end{lemma}
\begin{proof}
	We write \(\tilde{C}=C\times_B \tilde{B}\) for the total space after base change, \(\tilde{c}\colon \tilde{C}\to \tilde{B}\) for the family of curves and \(\tilde{p}_t=(p_t\circ b, \id_{\tilde{B}})\colon \tilde{B}\to \tilde{C}\) as well as \(\tilde{p}_e=(p_e\circ b, \id_{\tilde{B}})\colon \tilde{B}\to \tilde{C}\) for its sections induced by \(p_t\) for \(t\in T_\tau\) and \(p_e\) for \(e\in E_\tau\).
	The map \(f\colon \tilde{C}\to C\) satisfies \(c\circ f = b\circ \tilde{c}\) and \(f^*S_C=S_{\tilde{C}}\).

	As \(b\) is flat the pullback of~\eqref{SES:DefnNCB} yields
	\begin{diag}\label{SES:PullbackOfDefnNCB}
		\matrix[mat](m) {
			0 & b^*c_*S_C & {\displaystyle \bigoplus_{t\in T_\tau} b^*p_t^*S_C \bigoplus_{e\in E_\tau} b^*p_e^*S_C} & b^*N_C & 0 \\
		} ;
		\path[pf] {
			(m-1-1) edge (m-1-2)
			(m-1-2) edge (m-1-3)
			(m-1-3) edge (m-1-4)
			(m-1-4) edge (m-1-5)
		};
	\end{diag}
	For the first entry of the short exact sequence we have by flat base change for cohomology, see~\cite[Theorem~24.2.9]{V-RSFAG}, that \(b^*c_*S_C = \tilde{c}_*f^*S_C = \tilde{c}_*S_{\tilde{C}}\).
	The summands of the middle term of the short exact sequence~\eqref{SES:PullbackOfDefnNCB} are given by \(b^*p_t^*S_C=\tilde{p}_t^*f^*S_C=\tilde{p}_t^*S_{\tilde{C}}\) and \(b^*p_e^*S_C=\tilde{p}_e^*S_{\tilde{C}}\).
	Hence, the short exact sequence~\eqref{SES:PullbackOfDefnNCB} is isomorphic to the one defining \(N_{\tilde{C}}\) and \(b^*N_{C}=N_{\tilde{C}}\).
\end{proof}

\begin{prop}\label{prop:GluingNC}
	Let \(c_1\colon C_1\to B\) be a family of stable spin curves of constant tree type \(\tau_1\) and \(c_2\colon C_2\to B\) be a family of stable spin curves of constant tree type \(\tau_2\).
	For two tails \(t_1\in T_{\tau_1}\) and \(t_2\in T_{\tau_2}\) denote by \(\sigma\) the tree obtained from grafting the tails \(t_1\) and~\(t_2\).
	Furthermore, let \(c\colon C\to B\) be the stable curve of constant tree type \(\sigma\) obtained from gluing \(C_1\) and \(C_2\) along \(p_{t_1}\) and \(p_{t_2}\).
	Then \(N_C=N_{C_1}\oplus N_{C_2}\).
\end{prop}
\begin{proof}
	Note that \(S_{C}=i^1_*S_{C_1}\oplus i^2_*S_{C_2}\), where \(i^j\colon C_j\to C\) for \(j=1,2\) are the inclusions.
	This implies \(c_*S_{C}={\left(c_1\right)}_* S_{C_1}\oplus {\left(c_2\right)}_* S_{C_2}\).

	By construction the edges of \(\sigma\) are given by \(E_\sigma=E_{\tau_1}\cup E_{\tau_2}\cup \Set{(t_1, t_2)}\) and the tails of \(\sigma\) are given by \(T_\sigma = \left(T_{\tau_1}\cup T_{\tau_2}\right)\setminus \Set{t_1, t_2}\).
	Furthermore, for \(e\in E_{\tau_i}\) we have \(p_e^*S_C=p_e^*S_{C_i}\), for \(t\in T_{\tau_i}\setminus\Set{t_i}\) we have \(p_t^*S_C = p_t^*S_{C_i}\) and \(p_{(t_1,t_2)}^*S_C = p_{t_1}^*S_{C_1}\oplus p_{t_2}^*S_{C_2}\).
	Consequently,
	\begin{equation}
		\begin{split}
			\MoveEqLeft
			\bigoplus_{t\in T_\sigma} p_t^*S_{C} \bigoplus_{e\in E_\sigma} p_j^*S_C \\
			&= \bigoplus_{t\in T_{\tau_1}\setminus\Set{t_1}}p_t^*S_C \bigoplus_{t\in T_{\tau_2}\setminus\Set{t_2}} p_t^*S_C \oplus p_{(t_1, t_2)}^*S_C \bigoplus_{e\in E_{\tau_1}} p_e^*S_C\bigoplus_{e\in E_{\tau_2}}p_e^*S_C \\
			&= \bigoplus_{t\in T_{\tau_1}}p_t^*S_{C_1} \bigoplus_{e\in E_{\tau_1}} p_e^*S_{C_1}\bigoplus_{t\in T_{\tau_2}} p_t^*S_{C_2} \bigoplus_{e\in E_{\tau_2}}p_e^*S_{C_2}
		\end{split}
	\end{equation}
	This implies the claim.
\end{proof}

\begin{prop}\label{prop:StabilizationNC}
	Let \(c\colon C\to B\) be a stable spin curve of constant tree type \(\tau\).
	For \(t_0\in T_\tau\), denote by \(\tilde{\tau}\) the stable tree that arises from \(\tau\) by removing the tail \(t_0\) and stabilization.
	We denote by \(st\colon C\to \tilde{C}\) the map that stabilizes the curve \(C\) after removing the marked point \(p_{t_0}\).
	Then there is a short exact sequence:
	\begin{diag}\matrix[mat](m) {
			0 & p_{t_0}^*S_C & N_C & N_{\tilde{C}} & 0 \\
		} ;
		\path[pf] {
			(m-1-1) edge (m-1-2)
			(m-1-2) edge (m-1-3)
			(m-1-3) edge (m-1-4)
			(m-1-4) edge (m-1-5)
		};
	\end{diag}
\end{prop}
\begin{proof}
	Let \(v\) be the vertex of \(\tau\) that is bounding the removed tail \(t_0\).
	There are three different cases:
	\begin{enumerate}
		\item
			The vertex \(v\) has at least four adjacent flags.
			In this case no vertices of \(\tau\) are contracted, \(C=\tilde{C}\), \(S_C=S_{\tilde{C}}\) and \(st\colon C\to \tilde{C}\) is the identical map.
			We obtain the diagram where rows and columns are exact
			\begin{diag}\matrix[mat](m) {
					& & 0 & 0 & \\
					& & c_*S_C & \tilde{c}_* S_{\tilde{C}} & \\
					0 & p_{t_0}^*S_C & {\displaystyle\bigoplus_{t\in T_\tau} p_{t}^*S_C \bigoplus_{e\in E_\tau} p_e^*S_C} & {\displaystyle\bigoplus_{\tilde{t}\in T_{\tilde{\tau}}} p_{\tilde{t}}^*S_{\tilde{C}} \bigoplus_{\tilde{e}\in E_{\tilde{\tau}}} p_{\tilde{e}}^*S_{\tilde{C}}} & 0 \\
					0 & p_{t_0}^*S_C & N_C & N_{\tilde{C}} & 0 \\
					& & 0 & 0 & \\
				} ;
				\path[pf] {
					(m-1-3) edge (m-2-3)
					(m-1-4) edge (m-2-4)
					(m-2-3) edge (m-3-3)
						edge[gl] (m-2-4)
					(m-2-4) edge (m-3-4)
					(m-3-1) edge (m-3-2)
					(m-3-2) edge (m-3-3)
						edge[gl] (m-4-2)
					(m-3-3) edge (m-3-4)
						edge (m-4-3)
					(m-3-4) edge (m-3-5)
						edge (m-4-4)
					(m-4-1) edge (m-4-2)
					(m-4-2) edge (m-4-3)
					(m-4-3) edge (m-4-4)
						edge (m-5-3)
					(m-4-4) edge (m-4-5)
						edge (m-5-4)
				};
			\end{diag}
		\item\label{item:ProofPropStabilizationCurves}
			The vertex \(v\) has precisely three adjacent flags, the tail \(t_0\), a tail \(t'\) and an edge \(e'\) which consists of the two flags \(f_v\) and \(f_{v'}\).
			In this case the map \(st\colon C\to \tilde{C}\) is the identity on all irreducible components of \(C\) except for the irreducible component~\(C_v\) represented by \(v\) which is contracted to \(p_{f_{v'}}\).
			It follows that \(st\circ p_{e'} = \tilde{p}_{t'}\) and \(S_C = st^*S_{\tilde{C}}\otimes S_{C_v}\) where \(S_{C_v}\) is the spinor bundle on the irreducible component \(C_v\).
			Then there is the following commutative diagram with exact rows and columns:
			\begin{diag}\matrix[mat, column sep=tiny](m) {
					& 0 & 0 & 0 &\\
					0 & c_*S_{C_v} & c_*S_C & \tilde{c}_* S_{\tilde{C}} & 0\\
					0 & p_{t_0}^*S_C\oplus p_{t'}^*S_C \oplus p_{f_v}^*S_C & {\displaystyle\bigoplus_{t\in T_\tau} p_{t}^*S_C \bigoplus_{e\in E_\tau} p_e^*S_C} & {\displaystyle\bigoplus_{\tilde{t}\in T_{\tilde{\tau}}} p_{\tilde{t}}^*S_{\tilde{C}} \bigoplus_{\tilde{e}\in E_{\tilde{\tau}}} p_{\tilde{e}}^*S_{\tilde{C}}} & 0\\
					0 & p_{t_0}^*S_C & N_C & N_{\tilde{C}} & 0\\
					& 0 & 0 & 0 &\\
				} ;
				\path[pf] {
					(m-1-2) edge (m-2-2)
					(m-1-3) edge (m-2-3)
					(m-1-4) edge (m-2-4)
					(m-2-1) edge (m-2-2)
					(m-2-2) edge (m-2-3)
						edge (m-3-2)
					(m-2-3) edge (m-3-3)
						edge (m-2-4)
					(m-2-4) edge (m-2-5)
						edge (m-3-4)
					(m-3-1) edge (m-3-2)
					(m-3-2) edge (m-3-3)
						edge (m-4-2)
					(m-3-3) edge (m-3-4)
						edge (m-4-3)
					(m-3-4) edge (m-3-5)
						edge (m-4-4)
					(m-4-1) edge (m-4-2)
					(m-4-2) edge (m-4-3)
						edge (m-5-2)
					(m-4-3) edge (m-4-4)
						edge (m-5-3)
					(m-4-4) edge (m-4-5)
						edge (m-5-4)
				};
			\end{diag}
		\item
			The vertex \(v\) has precisely three adjacent flags, the tail \(t_0\) and two edges \(e\) and \(e'\).
			In this case the map \(st\colon C\to \tilde{C}\) is the identity on all irreducible components of \(C\) except for the irreducible components \(C_v\) which is contracted to \(p_e=p_e'\).
			Again the tree \(\tilde{\tau}\) has three flags less than \(\tau\) and the claim can be shown using a diagram similar to the one in case~\ref{item:ProofPropStabilizationCurves}.
			\qedhere
	\end{enumerate}
\end{proof}

\subsection{Conjecture on the SUSY normal bundle on the moduli stack of stable spin curves}\label{SSec:ConjectureSUSYNormalBundleStableSpinCurves}
In this section we discuss how SUSY normal bundles extend to the moduli stack of stable spin curves.

Recall the following facts about the moduli space of stable curves, see, for example,~\cite[Chapter~III, §3]{M-FMQCMS}
There is a universal curve \(\UniversalCurve{0,k}\to\ModuliStackCurves{0,k}\) fibered over the moduli stack \(\ModuliStackCurves{0,k}\) of stable curves of genus zero with \(k\) marked points together with sections \(p_1, \dotsc, p_k\colon \ModuliStackCurves{0,k}\to \UniversalCurve{0,k}\).
That is, for any family \((c\colon C\to B, q_1, \dotsc, q_k)\) of stable maps of genus zero with \(k\) marked points there is a unique map \(b\colon B\to \ModuliStackCurves{0,k}\) sucht that \(C=B\otimes_{\ModuliStackCurves{0,k}} \UniversalCurve{0,k}\) and \(q_i= b^*p_i\).
Stable curves of genus zero do not have non-trivial automorphisms.
Hence, the moduli stack of stable curves of genus zero is a scheme isomorphic to the coarse moduli space \(\ModuliSpaceCurves{0,k}\).

As any three distinct points on \(\ProjectiveSpace{1}\) can be mapped by a unique Möbius transformation to the three points \(0\), \(1\) and \(\infty\) of \(\ProjectiveSpace{1}\) the moduli space \(\ModuliStackCurves{0,3}\) is a point.
The choice of a fourth marked point yields \(\ModuliStackCurves{0,4}=\ProjectiveSpace{1}\).
In general, the moduli space of stable \(k\)-pointed curves of genus zero is a scheme of dimension \(d_k=k-3\).
For explicit constructions see~\cite{K-ITMSSNPCG0}.

Moreover, the moduli spaces come with several structure maps:
Forgetting the last marked point and stabilization yield a forgetful map \(f\colon \ModuliStackCurves{0,k+1}\to \ModuliStackCurves{0,k}\) which coincides with \(\UniversalCurve{0,k}\to\ModuliStackCurves{0,k}\).
Identifying the last marked point of a curve with the first marked point of another curve one obtains a nodal curve.
The corresponding map of moduli spaces are called gluing maps
\begin{equation}
	gl\colon \ModuliStackCurves{0,k_1+1}\times\ModuliStackCurves{0,k_2+1}\to \ModuliStackCurves{0,k_1+k_2}.
\end{equation}
For any element \(\sigma\in\SymG{k}\) of the symmetric group of \(\Set{1,\dotsc, k}\) there is a relabeling automorphism \(s_\sigma\colon \ModuliStackCurves{0,k}\to \ModuliStackCurves{0,k}\) that sends the marked point \(p_i\) to \(p_{\sigma(i)}\).

In genus zero, any stable curve can be uniquely equipped with a spinor bundle as described in Section~\ref{SSec:SUSYNormalBundlesForStableCurves}.
Consequently, the moduli space of stable spin curves of genus zero with \(k\) marked points coincides with \(\ModuliSpaceCurves{0,k}\).
However, every stable spin curve of fixed tree type \(\tau\) has automorphism group \(\Integers_2^{\# V_\tau}\) acting on the stable spin curve by sending the spinor bundle on some irreducible components to its negative.
Hence, the moduli stack \(\ModuliStackSpinCurves{0,k}\) has the same closed points as \(\ModuliStackCurves{0,k}\) but every point representing a curve of tree type \(\tau\) has an isotropy group \(\Integers_2^{\# V_\tau}\).

Moduli stacks of stable spin curves have been constructed in~\cites{C-MCTC}{J-GMHSC}{AJ-MTSC} in a more general setup.
We restrict our attention here to the moduli stack \(\ModuliStackSpinCurves{0,k}\) of stable spin curves of genus zero and \(k\) marked points of Neveu--Schwarz type.
The moduli stack \(\ModuliStackSpinCurves{0,k}\) is a smooth and proper Deligne--Mumford stack.
The universal curve \(\UniversalSpinCurve{0,k}\to\ModuliStackSpinCurves{0,k}\) has
\begin{itemize}
	\item
		\(k\) sections \(p^{spin}_1, \dotsc,p^{spin}_k\colon \ModuliStackSpinCurves{0,k}\to\UniversalSpinCurve{0,k}\), and
	\item
		a universal spinor bundle \(\UniversalSpinorBundleC\to\UniversalSpinCurve{0,k}\).
\end{itemize}
For every family \((c\colon C\to B, q_1, \dotsc, q_k, S_C)\) of stable spin curves there is a unique map \(b\colon B\to \ModuliStackSpinCurves{0,k}\) such that the given stable spin curve arises from pullback of the universal curve along \(b\).
More precisely, \(C= B\times_{\ModuliStackSpinCurves{0,k}}\UniversalSpinCurve{0,k}\) and the projection on the first factor coincides with \(c\) and the projection \(pr_2\) on the second factor satisfies \(p^{spin}_j \circ b = pr_2^*\circ q_j\) and \(S_C=pr_2^*\UniversalSpinorBundleC\).

Every stable spin curve gives a stable curve by forgetting the spinor bundle.
In genus zero the resulting map \(F\colon \ModuliStackSpinCurves{0,k}\to\ModuliStackCurves{0,k}\) induces an isomorphism on the coarse moduli spaces because every stable curve of genus zero can be equipped with a unique spin structure.

Furthermore, there are
\begin{itemize}
	\item
		a forgetful map \(f^{spin}\colon \ModuliStackSpinCurves{0,k+1}\to \ModuliStackSpinCurves{0,k}\) forgetting the last marked point,
	\item
		a gluing map \(gl^{spin}\colon \ModuliStackSpinCurves{0,k_1+1}\times \ModuliStackSpinCurves{0,k_2+1}\to \ModuliStackSpinCurves{0,k_1+k_2}\) sending a pair of stable spin curves to the one obtained by identifying the last marked point of the first curve to the first marked point of the second curve, and
	\item
		for any \(\sigma\in \SymG{k}\) there is an automorphism \(s_\sigma^{spin}\colon \ModuliStackSpinCurves{0,k}\to\ModuliStackSpinCurves{0,k}\) such that \(p^{spin}_j \circ s_\sigma^{spin} = p^{spin}_{\sigma(j)}\).
\end{itemize}
Those maps on moduli stacks of stable spin curves are lifts of the corresponding maps on moduli stacks of stable curves, that is, for example, \(F\circ f^{spin} = f\circ F\).
Hence, we will drop the superscript and write \(p_i\), \(f\), \(gl\) and \(s_\sigma\) instead of \(p^{spin}_i\), \(f^{spin}\), \(gl^{spin}\) and \(s_\sigma^{spin}\).

For supergeometric reasons which we discuss in Section~\ref{SSec:SNMotivation} we expect that the SUSY normal bundles extend to the moduli stacks of stable curves curves as follows:
\begin{conjecture}\label{conj:Nk}
	There is a unique vector bundle \(\overline{N}_k\) of rank \(r_k=k-2\) on \(\ModuliStackSpinCurves{0,k}\) such that for any stable spin curve \(c\colon C\to B\) of genus zero with \(k\) marked points and fixed tree type it holds \(b^*\overline{N}_k=N_C\).
	Here \(b\colon B\to\ModuliStackSpinCurves{0,k}\) is the classifying map as above.

	The SUSY normal bundle \(\overline{N}_k\) satisfies
	\begin{enumerate}
		\item\label{item:PropertiesNk:N3N4}
			The vector bundle \(\overline{N}_3\) is trivial of rank one.
		\item\label{item:PropertiesNk:ForgettingMarkedPoint}
			There is a short exact sequence
			\begin{diag}\matrix[mat](m) {
					0 & \UniversalSpinorBundleC_{k+1} & \overline{N}_{k+1} & f^*\overline{N}_{k} & 0 \\
				} ;
				\path[pf] {
					(m-1-1) edge (m-1-2)
					(m-1-2) edge (m-1-3)
					(m-1-3) edge (m-1-4)
					(m-1-4) edge (m-1-5)
				};
			\end{diag}
			Here \(\UniversalSpinorBundleC_{k+1}=p_{k+1}^*\UniversalSpinorBundleC\) is the pullback of the universal spinor bundle along the \(k+1\)st marked point.
		\item\label{item:PropertiesNk:Glueing}
			It is compatible with gluing, that is \(gl^*\overline{N}_{k_1+k_2}=\overline{N}_{k_1+1}\oplus \overline{N}_{k_2+1}\).
		\item\label{item:PropertiesNk:Permutation}
			It is compatible with relabeling, that is \(s_\sigma^*\overline{N}_k=\overline{N}_k\).
	\end{enumerate}
\end{conjecture}
A very similar bundle has been constructed in~\cite{N-NCCMSC} and its relationship to supergeometry has been discussed in~\cite{N-EGMSSRS}.

\subsection{SUSY normal bundles for stable spin maps}\label{SSec:SUSYNormalBundlesForStableMaps}
In this section we describe the SUSY normal bundles for stable spin maps and some basic properties.

First, we recall the notion of stable maps, see, for example,~\cites[Definition~2.4.1]{KM-GWCQCEG}[Definition~1.1]{FP-NSMQC}.
Let \(\phi\colon C\to X\) be a map from a \(J\)-marked prestable curve over \(\Spec\C\) and \(\tau\) the dual tree of \(C\).
The map~\(\phi\) induces a map
\begin{equation}
	\begin{split}
		\beta_\tau\colon V_\tau &\to H^+_2(X) \\
		v &\mapsto \phi_*[C_v]
	\end{split}
\end{equation}
that sends every vertex to the pushforward of the fundamental cycle of the irreducible component \(C_v\) represented by \(v\).
Here \(H_2^+(X)\) is the cone of effective second homology classes of \(X\).
The map \(\beta_\tau\) satisfies
\begin{equation}
	\beta = \phi_*[C]
	= \sum_{v\in V_\tau} \beta_\tau(v)
	= \sum_{v\in V_\tau} \phi_*[C_v].
\end{equation}

The map \(\phi\) is called a stable map precisely if for every \(v\in V_\tau\) such that \(\beta_\tau(v)=0\) there are at least three flags bounding \(v\).
That is, \(\phi\) is stable if every irreducible component of \(C\) which is mapped by \(\phi\) to a point in \(X\) contains at least three special points.
A family of maps \(\phi\colon C\to X\) for a family \(c\colon C\to B\) is stable if the resulting map on every geometric fiber is stable.
We say that a stable map \(\phi\colon C\to X\) over an arbitrary base \(B\) is of constant tree type \((\tau, \beta_\tau)\) if every geometric fiber of \(C\to B\) has dual tree \(\tau\) and the image of the fundamental class of the irreducible components is given by \(\beta_\tau(v)\).
For a systematic treatment of marked trees \((\tau, \beta_\tau)\) see~\cite{BM-SSMGMI}.

For this work we are interested in the case when the domain prestable curve is equipped with a spinor bundle.
A \(J\)-marked stable spin map of genus zero is a tuple
\begin{equation}
	\left(
		c\colon C\to B,
		\Set{p_j\colon B\to C}_{j\in J},
		S_C,
		\phi\colon C\to X
	\right)
\end{equation}
such that \(\left(C, \Set{p_j}, S_C\right)\) is a \(J\)-marked prestable spin curve of genus zero and \(\phi\) is a stable map.
A morphisms
\begin{equation}
	\left(c\colon C\to B, \Set{p_j}_{j\in J}, S_C, \phi\right)
	\to
	\left(c'\colon C'\to B, \Set{p'_{j'}}_{j'\in J'}, S_{C'}, \phi'\right)
\end{equation}
of stable spin maps over \(B\) is a morphism of prestable maps \((l, g, s)\) such that \(\phi=\phi'\circ g\).
In genus zero generic stable maps without spin structure do not have automorphisms while generic stable spin maps of tree type \((\tau, \beta_\tau)\) have automorphism group \(\Integers^{\#V_\tau}\).

We assume in addition that \(X\) is convex according to~\cite[Definition~2.4.2]{KM-GWCQCEG}.
That is, for every map \(\phi\colon \ProjectiveSpace{1}\to X\) it holds \(H^1(\phi^*\tangent{X})=0\).
It follows that \(H^1(\phi^*\tangent{X})=0\) for all stable maps of genus zero as explained in~\cite[Lemma~10]{FP-NSMQC}.
Examples of convex varieties include \(\ProjectiveSpace{n}\) and more generally homogeneous spaces \(X=\faktor{G}{P}\) where \(P\) is a parabolic subgroup of a Lie group~\(G\).
We need the following property of convex varieties:
\begin{lemma}\label{lemma:H1SvphiTX}
	Let \(X\) be a convex variety.
	Then for any stable spin map \(\phi\colon C\to X\)
	\begin{align}\label{eq:H1SvphiTXRR}
		\dim H^0\left(\dual{S}_C\otimes \phi^*\tangent{X}\right)
		&= \left<c_1(\tangent{X}), \phi_*[C]\right>, &
		\dim H^1\left(\dual{S}_C\otimes \phi^*\tangent{X}\right) &= 0.&
	\end{align}
	Moreover, for any non-constant stable spin map \(\phi\colon \ProjectiveSpace{1}\to X\)
	\begin{equation}
		\dim H^0\left(\dual{S}_{\ProjectiveSpace{1}}\otimes \phi^*\tangent{X}\right)
		\geq 1.
	\end{equation}
\end{lemma}
\begin{proof}
	If \(C=\ProjectiveSpace{1}\), the Theorem of Riemann--Roch implies that
	\begin{equation}
		\begin{split}
			\MoveEqLeft
			\dim H^0\left(\dual{S}_{\ProjectiveSpace{1}}\otimes \phi^*\tangent{X}\right) - \dim H^1\left(\dual{S}_{\ProjectiveSpace{1}}\otimes \phi^*\tangent{X}\right)\\
			&= \int_{\ProjectiveSpace{1}} c_1\left(\dual{S}_{\ProjectiveSpace{1}}\otimes \phi^*\tangent{X}\right) + \dim X
			= \int_{\ProjectiveSpace{1}} c_1\left(\phi^*\tangent{X}\right)
			= \left<c_1\left(\tangent{X}\right), \phi_*[\ProjectiveSpace{1}]\right>.
		\end{split}
	\end{equation}
	The bundles \(\phi^*\tangent{X}\) and \(\dual{S}_{\ProjectiveSpace{1}}\otimes\phi^*\tangent{X}\) decompose into a sum of line bundles
	\begin{align}
		\phi^*\tangent{X} &= \bigoplus_{i=1}^{\dim X}\cO(d_i) &
		\dual{S}_{\ProjectiveSpace{1}}\otimes \phi^*\tangent{X} &= \bigoplus_{i=1}^{\dim X}\cO(d_i-1)
	\end{align}
	for some integers \(d_i\) by the Grothendieck--Birkhoff Theorem and because \(\dual{S}_{\ProjectiveSpace{1}}=\cO(-1)\).
	As argued in the proof of~\cite[Lemma~10]{FP-NSMQC} the numbers~\(d_i\) must be non-negative.
	Hence, \(H^1\left(\dual{S}_{\ProjectiveSpace{1}}\otimes \phi^*\tangent{X}\right) = 0\) because \(d_i-1\geq -1\).
	In addition, as argued in~\cite[Lemma~11]{FP-NSMQC} if \(\phi\) is non-constant one of the numbers~\(d_i\) is at least two.
	Consequently, \(\dim H^0\left(\dual{S}_{\ProjectiveSpace{1}}\otimes \phi^*\tangent{X}\right)\geq 1\).

	For arbitrary prestable curves \(C\) every irreducible component \(C_j\) is isomorphic to \(\ProjectiveSpace{1}\) and~\eqref{eq:H1SvphiTXRR} follows from
	\begin{equation}
		\dual{S}_C\otimes\phi^*\tangent{X}
		= \left(\bigoplus_j \dual{S}_{C_j}\right)\otimes \phi^*\tangent{X}
		= \bigoplus_j \dual{S}_{C_j}\otimes\phi|_{C_j}^*\tangent{X}
	\end{equation}
	and \([C]=\sum_j [C_j]\).
\end{proof}

\begin{defn}\label{defn:NCphi}
	Let \(\phi\colon C\to X\) be a family of stable spin maps over \(B\) of constant tree type \((\tau, \beta_\tau)\).
	The SUSY normal bundle \(N_{(C,\phi)}\) is the coherent sheaf on \(B\) given by
	\begin{diag}\label{SES:DefnNCphiB}
		\matrix[mat, column sep=small](m) {
			0 & c_*S_C & {\displaystyle\bigoplus_{t\in T_\tau} p_t^*S_{C} \bigoplus_{e\in E_\tau} p_e^*S_{C} \oplus c_*\left(\dual{S}_C\otimes \phi_*\tangent{X}\right)} & N_{(C, \phi)} & 0 \\
				& s & \bigoplus p_t^*s \bigoplus p_e^*s \oplus -\left<s, \differential{\phi}\right> & & \\
		} ;
		\path[pf] {
			(m-1-1) edge (m-1-2)
			(m-1-2) edge (m-1-3)
			(m-1-3) edge (m-1-4)
			(m-1-4) edge (m-1-5)
			(m-2-2) edge[commutative diagrams/mapsto] (m-2-3)
		};
	\end{diag}
	Here \(\left<s, \differential{\phi}\right>\) is the section of \(\dual{S_C}\otimes \phi^*\tangent{X}\) that arises from the contraction of \(\differential{\phi}\) seen as a section of \(\cotangent{C}\otimes \phi^*\tangent{X}= \dual{S_C}\otimes \dual{S_C}\otimes \phi^*\tangent{X}\) with the section \(s\) of \(S_C\).
\end{defn}

Any isomorphism of stable spin maps induces an isomorphism of SUSY normal bundles.

\begin{lemma}
	The sheaf \(N_{(C, \phi)}\) is locally free of rank \(r_{k,\beta}=\left<c_1(\tangent{X}), \beta\right>+k-2\).
\end{lemma}
\begin{proof}
	The sheaf \(N_{(C, \phi)}\) is coherent by construction and has the same rank on all closed points and hence is locally free as in the proof of Lemma~\ref{lemma:NCBLocallyFree}.
	To see that \(N_{(C, \phi)}\) has the same rank at all closed points, let \(b\) be a closed point of \(B\) and \(r\) the number of vertices of \(\tau\).
	The rank of the sheaf \(c_*(\dual{S}_{C}\otimes \phi^*\tangent{X})\) at \(b\) is \(\left<c_1(\tangent{X}), \beta\right>\) by Lemma~\ref{lemma:H1SvphiTX}.
	Hence the sheaf in the middle of~\eqref{SES:DefnNCphiB} has rank \(2(r-1)+k+\left<c_1(\tangent{X}), \beta\right>\).
	The sheaf \(c_*S_C\) has rank \(2r\) at \(b\), consequently \(N_{(C,\phi)}\) has rank \(r_{k,\beta}=\left<c_1(\tangent{X}), \beta\right> + k-2\) independent of \(b\) and \(r\).
\end{proof}

The following Example will be important in Section~\ref{Sec:SGWOfPn}.
\begin{ex}\label{ex:NCphiPn}
	Let \(C=\ProjectiveSpace{1}\) with \(k\)-marked points \(p_i\) for \(i=1,\dotsc, k\) and
	\begin{equation}
		\begin{split}
			\phi\colon \ProjectiveSpace{1} &\to \ProjectiveSpace{n} \\
			\left[Z^1:Z^2\right] &\mapsto \left[0: \dots  :0: {\left(Z^1\right)}^d :0: \dots :0: {\left(Z^2\right)}^d :0: \dots :0\right]
		\end{split}
	\end{equation}
	be the map that sends \(\ProjectiveSpace{1}\) as a \(d\)-fold cover on the coordinate line with coordinates \(X^a\) and \(X^b\).
	We will make the construction of \(N_{(C, \phi)}\) explicit in local coordinates of \(\ProjectiveSpace{1}\) and~\(\ProjectiveSpace{n}\).
	In particular we obtain for \(k\geq 3\)
	\begin{equation}
		N_{(C, \phi)} = p_3^*S_{\ProjectiveSpace{1}}\oplus \dotsb \oplus p_k^*S_{\ProjectiveSpace{1}}\oplus H^0(\dual{S_{\ProjectiveSpace{1}}}\otimes \phi^*\tangent{X}),
	\end{equation}
	where \(p_i^*S_{\ProjectiveSpace{1}}\) are complex one-dimensional vector spaces and
	\begin{equation}
		H^0\left(\dual{S_{\ProjectiveSpace{1}}}\otimes\phi^*\tangent{\ProjectiveSpace{n}}\right)
		= H^0(\cO(2d-1) \oplus {\left(\cO(d-1)\right)}^{n-1})
		= \C^{2d} \oplus {\left(\C^d\right)}^{n-1}.
	\end{equation}

	Let \(V_k\subset \ProjectiveSpace{1}\), \(k=1,2\) be the open coordinate neighborhoods with coordinates \(z_1=\frac{Z^1}{Z^2}\) and \(z_2=-\frac{Z^2}{Z^1}\), that is \(z_1=-\frac1{z_2}\).
	With this unconventional sign choice we conform to the conventions in~\cite{KSY-TAMSSSMGZ} and, in particular the spinor bundle \(S_{\ProjectiveSpace{1}}=\cO(1)\) has local frames \(s_1\), \(s_2\) such that \(s_i\otimes s_i=\partial_{z_i}\) and \(s_1=-z_2 s_2\).
	A local section \(s\in H^0(S_{\ProjectiveSpace{1}})\) can be written as
	\begin{equation}
		s= \left(uz_1 + v\right) s_1 = \left(u - vz_2\right) s_2.
	\end{equation}
	The restriction of \(s\) to a point \(p\) with coordinate \(z_1=p_1\) is then given by \(p^*s = \left(up_1 + v\right)p^*s_1\).

	Furthermore, we denote by \(U_l\subset\ProjectiveSpace{n}\) the coordinate neighborhoods with \(X^l=1\) for \(l=a,b\) and local coordinates \(x_l^q\) for \(q=0,\dotsc,l-1, l+1,\dotsc, n\).
	The coordinate change is given by
	\begin{align}
		x_a^b &= \frac{1}{x_b^a}, &
		x_a^m &= \frac{x_b^m}{x_b^a},\text{ for }m\neq a,b.
	\end{align}
	The tangent bundle \(\tangent{\ProjectiveSpace{n}}\) is trivialized over \(U_a\) by \(U_a\times \C^n\) with basis \(\partial_{x_a^q}\).
	The different trivializations glue via
	\begin{align}
		\partial_{x_b^a} &= -\frac{1}{{(x_b^a)}^2}\left(\partial_{x_a^b} + \sum_{m\neq a,b} x_b^m \partial_{x_a^m}\right), &
		\partial_{x_b^m} &= \frac{1}{x_b^a}\partial_{x_a^m}, \text{ for }m\neq a,b.
	\end{align}

	The map \(\phi\colon \ProjectiveSpace{1}\to\ProjectiveSpace{n}\) maps \(V_1\) to \(U_b\) and \(V_2\) to \(U_a\) by
	\begin{align}
		\phi^\# x_b^a &= {z_1}^d, &
		\phi^\# x_b^m &= 0,\text{ for }m\neq a,b, \\
		\phi^\# x_a^b &= {\left(-z_2\right)}^d, &
		\phi^\# x_a^m &= 0,\text{ for }m\neq a, b.
	\end{align}
	Consequently, the bundle \(\phi^*\tangent{\ProjectiveSpace{n}}\) is trivialized by \(V_1\times \C^n\) with basis \(\phi^*\partial_{x_b^q}\) and \(V_2\times \C^n\) with basis \(\phi^*\partial_{x_a^q}\) and gluing via
	\begin{align}
		\phi^*\partial_{x_b^a} &= -{\left(-z_2\right)}^{2d}\phi^*\partial_{x_a^b}, &
		\phi^*\partial_{x_b^m} &= {\left(-z_2\right)}^d\phi^*\partial_{x_a^m}, \text{ for }m\neq a,b.
	\end{align}
	This explains the explicit decomposition of \(\phi^*\tangent{\ProjectiveSpace{n}}=L_{ab}\bigoplus_{m\neq a,b} L_m\) into line bundles with \(L_{ab}=\cO(2d)\) and \(L_m=\cO(d)\).
	As \(S_{\ProjectiveSpace{1}}\otimes S_{\ProjectiveSpace{1}}=\tangent{\ProjectiveSpace{1}}=\cO(2)\) it follows that \(\dual{S}_{\ProjectiveSpace{1}}=\cO(-1)\) and hence
	\begin{align}
		\dual{S}_{\ProjectiveSpace{1}}\otimes L_{ab} &= \cO(2d-1), &
		\dual{S}_{\ProjectiveSpace{1}}\otimes L_m &= \cO(d-1).
	\end{align}
	Their spaces of sections are given by polynomials in \(z_1\) or \(z_2\) of maximal degree \(2d-1\) and \(d-1\) respectively.
	Explicitly,
	\begin{align}
		\psi(z_1) &= \sum_{q=0}^{2d-1}\psi^{ab}_q {z_1}^q \dual{s}_1\otimes\phi^*\partial_{x_b^a} + \sum_{m\neq a,b}\sum_{q=0}^{d-1}\psi^m_p{z_1}^q \dual{s}_1\otimes \phi^*\partial_{x_b^m}, \\
		\psi(z_2) &= -\sum_{q=0}^{2d-1}\psi^{ab}_q {\left(-z_2\right)}^{2d-1-q} \dual{s}_2\otimes\phi^*\partial_{x_a^b} + \sum_{m\neq a,b}\sum_{q=0}^{d-1}\psi^m_q{\left(-z_2\right)}^{d-1-q} \dual{s}_2\otimes \phi^*\partial_{x_a^m}.
	\end{align}

	Finally, we calculate
	\begin{equation}
		-\left<s, \differential{\phi}\right>
		= -d\left(uz_1^d+ vz_1^{d-1}\right)\dual{s}_1\otimes \phi^*\partial_{x_a^b}
		= d\left(u{\left(-z_2\right)}^{d-1} + v{\left(-z_2\right)}^d\right)\dual{s}_2\otimes \phi^*\partial_{x_b^a}.
	\end{equation}
\end{ex}

\begin{lemma}\label{lemma:PullbackNCphi}
	Let \(b\colon \tilde{B}\to B\) be a flat morphism.
	Then \(N_{(C\times_B \tilde{B},\tilde{\phi})}=b^*N_{(C, \phi)}\).
\end{lemma}
\begin{proof}
	We use the notation as in the proof of Lemma~\ref{lemma:PullbackNC}:
	\begin{diag}\matrix[mat, row sep=large, column sep=large](m) {
			\tilde{C}=C\times_B \tilde{B} & C & & X\\
			\tilde{B} & B & \\
		} ;
		\path[pf] {
			(m-1-1) edge node[auto]{\(f\)} (m-1-2)
				edge node[auto]{\(\tilde{c}\)} (m-2-1)
				edge[bend left=30] node[auto]{\(\tilde{\phi}\)} (m-1-4)
			(m-1-2) edge node[auto]{\(c\)} (m-2-2)
				edge node[auto]{\(\phi\)} (m-1-4)
			(m-2-1) edge node[auto]{\(b\)} (m-2-2)
				edge[bend left=50] node[auto]{\(\tilde{p}_t\)} (m-1-1)
				edge[bend right=50] node[auto,swap]{\(\tilde{p}_e\)} (m-1-1)
			(m-2-2)
				edge[bend left=50] node[auto]{\(p_t\)} (m-1-2)
				edge[bend right=50] node[auto,swap]{\(p_e\)} (m-1-2)
		};
	\end{diag}
	The proof of this lemma follows as the proof of Lemma~\ref{lemma:PullbackNC} by pulling back the short exact sequence~\eqref{SES:DefnNCphiB} defining \(N_{(C,\phi)}\) along \(b\).
	The only additional term compared to the proof of Lemma~\ref{lemma:PullbackNC} is given by \(b^*c_*(\dual{S}_C\otimes \phi^*\tangent{X}) = \tilde{c}_*f^*(\dual{S}_C \otimes \phi^*\tangent{X}) = \tilde{c}_*(\dual{S}_{\tilde{C}}\otimes\tilde{\phi}^*\tangent{X})\).
\end{proof}

\begin{prop}
	Let \(\phi\colon C\to X\) be a constant stable spin map of constant tree type~\((\tau, 0)\).
	Then \(N_{(C, \phi)}=N_C\).
\end{prop}
\begin{proof}
	In this case, the tree \(\tau\) is a stable tree and \(c_*(\dual{S}_C\otimes \phi^*\tangent{X})=0\).
	Hence \(N_{(C, \phi)}=N_C\) by definition.
\end{proof}

\begin{prop}\label{prop:GluingNCphi}
	For two tails \(t_1\in T_{\tau_1}\) and \(t_2\in T_{\tau_2}\) of trees \(\tau_1\) and \(\tau_2\) denote by \(\sigma\) the tree obtained from grafting \(\tau_1\) and \(\tau_2\) along \(t_1\) and \(t_2\).
	Let \(\phi_1\colon C_1\to X\) and \(\phi_2\colon C_2\to X\) two families of stable spin maps of constant tree type \((\tau_1, \beta_{\tau_1})\) and \((\tau_2, \beta_{\tau_2})\) respectively, such that \(\phi_1\circ p_{t_1}=\phi_2\circ p_{t_2}\).
	The stable spin map \(\phi\colon C\to X\) obtained from gluing \(\phi_1\) and \(\phi_2\) along \(p_{t_1}\) and \(p_{t_2}\) is of constant tree type \((\sigma, \beta_\sigma = \beta_{\tau_1}\coprod \beta_{\tau_2})\) and \(N_{(C,\phi)} = N_{(C_1,\phi_1)}\oplus N_{(C_1, \phi_2)}\).
\end{prop}
\begin{proof}
	The proof is analogous to the proof of Proposition~\ref{prop:GluingNC}.
	The only additional summand  to consider is
	\begin{equation}
		c_*\left(\dual{S}_{C}\otimes \phi^*\tangent{X}\right)={\left(c_1\right)}_*\left(\dual{S}_{C_1}\otimes\phi_1^*\tangent{X}\right)\oplus {\left(c_2\right)}_*\left(\dual{S}_{C_2}\otimes \phi_2^*\tangent{X}\right).
		\tag*{\qedhere}
	\end{equation}
\end{proof}

\begin{prop}
	Let \(\phi\colon C\to X\) be a stable spin map of constant tree type \((\tau, \beta_\tau)\).
	For \(t_0\in T_\tau\) let  \(\tilde{\tau}\) be the graph that arises from \(\tau\) by removing the tail \(t_0\) and stabilization such that \((\tilde{\tau}, \beta_{\tilde{\tau}}=\beta_{\tau})\) is stable.
	We denote by \(\tilde{\phi}\colon \tilde{C}\to X\) the stable spin map of constant tree type \((\tilde{\tau}, \beta_{\tilde{\tau}})\) that arises from \(\phi\) after removing the marked point \(p_{t_0}\) and stabilizing and by \(st\colon C\to \tilde{C}\) the stabilization map.
	Then there is a short exact sequence
	\begin{diag}\matrix[mat](m) {
			0 & p_{t_0}^*S_C & N_{(C,\phi)} & N_{\left(\tilde{C},\tilde{\phi}\right)} & 0 \\
		} ;
		\path[pf] {
			(m-1-1) edge (m-1-2)
			(m-1-2) edge (m-1-3)
			(m-1-3) edge (m-1-4)
			(m-1-4) edge (m-1-5)
		};
	\end{diag}
\end{prop}
\begin{proof}
	Let \(v\) be the vertex that bounds the removed tail \(t_0\).
	The vertex \(v\) needs to be contracted under stabilization precisely if \(\beta_\tau(v)=0\) and \(v\) has precisely three bounding flags including \(t_0\).
	Hence the three cases in the proof of Proposition~\ref{prop:StabilizationNC} generalize to the case of stable maps by adding the additional summand \(c_*(\dual{S}_C\otimes \phi^*\tangent{X})=\tilde{c}_*(\dual{S}_{\tilde{C}}\otimes \tilde{\phi}^*\tangent{X})\).
\end{proof}

\subsection{Conjecture on the SUSY normal bundle on the moduli stack of stable spin maps}\label{SSec:ConjectureSUSYNormalBundleStableSpinMaps}
In this section we formulate the conjecture that the SUSY normal bundles extend to a vector bundle on the stack of stable spin maps.

Recall the following properties of the moduli stack of stable maps, see, for example,~\cite{M-FMQCMS}.
Let \(\ModuliStackMaps{0,k}{X, \beta}\) be the moduli stack of stable maps of genus zero with \(k\) marked points in \(X\) representing the homology class \(\beta\).
Over \(\ModuliStackMaps{0,k}{X,\beta}\) there is a universal curve \(\UniversalMap{0,k}{X,\beta}\to\ModuliStackMaps{0,k}{X,\beta}\) together with \(k\) sections \(p_1, \dotsc, p_k\colon \ModuliStackMaps{0,k}{X,\beta}\to \UniversalMap{0,k}{X,\beta}\) and a map \(\ev_{k+1}\colon\UniversalMap{0,k}{X,\beta}\to X\).
For any family
\begin{equation}
	(c\colon C\to B, q_1, \dotsc, q_k\colon B\to C, \phi\colon C\to X)
\end{equation}
of stable maps there is a unique map \(b\colon B\to\ModuliStackMaps{0,k}{X,\beta}\) such that \(C=B\times_{\ModuliStackMaps{0,k}{X,\beta}} \UniversalMap{0,k}{X,\beta}\) where the projection on the first factor coincides with \(c\colon C\to B\) and the projection on the second factor satisfies \(pr_2\circ q_i= p_i\circ b\) and \(\phi=\ev_{k+1}\circ pr_2\).
\(\ModuliStackMaps{0,k}{X,\beta}\) is a smooth and proper Deligne--Mumford stack.
The convexity assumption implies that it is of pure dimension \(d_{k,\beta} = \dim X + \left<c_1(\tangent{X}), \beta\right> + k -3\).

We denote the maps forgetting marked points again by \(f\colon \ModuliStackMaps{0,k+1}{X, \beta}\to \ModuliStackMaps{0,k}{X, \beta}\) and the forgetful maps to the moduli space of stable curves by \(\pi\colon \ModuliStackMaps{0,k}{X,\beta}\to \ModuliStackCurves{0,k}\).
The forgetful map \(f\colon \ModuliStackMaps{0,k+1}{X, \beta}\to \ModuliStackMaps{0,k}{X,\beta}\) is isomorphic to the universal family \(\UniversalMap{0,k}{X,\beta}\to \ModuliStackMaps{0,k}{X,\beta}\).
Furthermore, we have the evaluation maps \(\ev_i\colon \ModuliStackMaps{0,k}{X,\beta}\to X\), \(i=1, \dotsc, k\) such that the following diagram is commutative
\begin{diag}\matrix[mat](m) {
		\UniversalMap{0,k}{X,\beta} = \ModuliStackMaps{0,k+1}{X, \beta} && X^{k+1} & \\
		& \ModuliStackMaps{0,k}{X, \beta} && X^k \\
		\UniversalCurve{0,k} = \ModuliStackCurves{0,k+1} && & \\
		& \ModuliStackCurves{0,k} && \\
	} ;
	\path[pf] {
		(m-1-1) edge node[auto]{\(\pi\)} (m-3-1)
						edge node[auto]{\((\ev_1, \dotsc, \ev_{k+1})\)} (m-1-3)
						edge node[auto]{\(f\)} (m-2-2)
		(m-1-3) edge (m-2-4)
		(m-2-2) edge node[auto]{\((\ev_1, \dotsc, \ev_k)\)} (m-2-4)
						edge node[auto]{\(\pi\)} (m-4-2)
		(m-3-1) edge node[auto]{\(f\)} (m-4-2)
	};
\end{diag}

There is also a gluing map
\begin{equation}
	gl_X\colon \ModuliStackMaps{0,k_1+1}{X,\beta_1}\times_X \ModuliStackMaps{0,k_2+1}{X, \beta_2}\to \ModuliStackMaps{0,k_1+k_2}{X, \beta_1+\beta_2}
\end{equation}
identifying the last marked point of the first stable map to the first marked point of the second stable map if their image in \(X\) coincides as well as a relabeling map
\begin{equation}
	s_\sigma\colon \ModuliStackMaps{0,k}{X,\beta}\to\ModuliStackMaps{0,k}{X,\beta}
\end{equation}
for any \(\sigma\in\SymG{k}\).

The construction of moduli stacks of stable maps has been extended to stable spin maps in~\cite{JKV-SGWI}.
We denote the moduli stack of stable spin maps of genus zero with \(k\) marked points of Neveu--Schwarz type representing the homology class \(\beta\in H^+_2(X)\) by \(\ModuliStackSpinMaps{0,k}{X,\beta}\).
There is again a universal curve \(\UniversalSpinMap{0,k}{X,\beta}\to\ModuliStackSpinMaps{0,k}{X,\beta}\) with \(k\) sections \(p^{spin}_1,\dotsc, p_k^{spin}\colon \ModuliStackSpinMaps{0,k}{X, \beta}\to\UniversalSpinMap{0,k}{X,\beta}\) as well as a map \(\ev_{k+1}\colon\UniversalSpinMap{0,k}{X, \beta}\to X\) and a spinor bundle \(\UniversalSpinorBundleM\to \UniversalSpinMap{0,k}{X,\beta}\) such that for every every family of stable spin maps can be obtained via unique pullback.
That is, for any family \((c\colon C\to B, \Set{q_j}_{j=1,\dotsc, k}, S_C, \phi)\) of stable spin maps there is a unique map \(b\colon B\to \ModuliStackSpinMaps{0,k}{X,\beta}\) such that \(C=B\times_{\ModuliStackSpinMaps{0,k}{X,\beta}}\UniversalSpinMap{0,k}{X,\beta}\) and the projection on the first factor coincides with \(b\) and the projection \(pr_2\) on the second factor satisfies \(p_j^{spin}\circ b = pr_2\circ q_j\) as well as \(S_C=pr_2^*\UniversalSpinorBundleM\) and \(\phi=\ev_{k+1}\circ pr_2\).
The stack \(\ModuliStackSpinMaps{0,k}{X,\beta}\) is a smooth and proper Deligne--Mumford stack.

In genus zero the forgetful map \(F\colon \ModuliStackSpinMaps{0,k}{X,\beta}\to \ModuliStackMaps{0,k}{X,\beta}\) induces an isomorphism of coarse moduli spaces because any prestable curve of genus zero carries a unique spin structure.
Moreover there are
\begin{itemize}
	\item
		a forgetful map \(f^{spin}\colon \ModuliStackSpinMaps{0,k+1}{X,\beta}\to \ModuliStackSpinMaps{0,k}{X,\beta}\) forgetting the last marked point
	\item
		a forgetful map \(\pi^{spin}\colon \ModuliStackSpinMaps{0,k}{X,\beta}\to \ModuliStackSpinCurves{0,k}\) forgetting the map,
	\item
		a gluing map \(gl_X\colon \ModuliStackSpinMaps{0,k_1+1}{X,\beta_1}\times_X \ModuliStackSpinMaps{0,k_2+1}{X, \beta_2}\to \ModuliStackSpinMaps{0,k_1+k_2}{X, \beta_1+\beta_2}\) identifying the last marked point of the first stable spin map to the first marked point of the second stable spin map if their image in \(X\) coincides, and
	\item
		for any \(\sigma\in \SymG{k}\) there is a relabeling automorphism \(s_\sigma^{spin}\colon \ModuliStackSpinMaps{0,k}{X,\beta}\to\ModuliStackSpinMaps{0,k}{X,\beta}\) such that \(p_i^{spin}\circ s_\sigma^{spin}=p^{spin}_{\sigma(j)}\).
\end{itemize}
All those maps on moduli stacks of stable spin maps are lifts of the corresponding maps on moduli stacks of stable maps, that is, for example, \(F\circ f^{spin} = f\circ F\).
Hence,  we will drop the superscript and write \(p_i\), \(f\), \(\pi\), \(gl\) and \(s_\sigma\) instead of \(p^{spin}_i\), \(f^{spin}\), \(\pi^{spin}\), \(gl^{spin}\) and \(s_\sigma^{spin}\) in the sequel.

\begin{conjecture}\label{conj:NkA}
	There is a unique vector bundle \(\overline{N}_{k,\beta}\) of rank \(r_{k,\beta} = \left<c_1(X), \beta\right> + k -2\) on \(\ModuliStackSpinMaps{0,k}{X, \beta}\) such that for every family \((c\colon C\to B, \Set{q_i}, S_C, \phi)\) of stable spin maps of genus zero and fixed tree type \(N_{(C,\phi)}=b^*\overline{N}_{k,\beta}\) where \(b\colon B\to \ModuliStackSpinMaps{0,k}{X,\beta}\) is the unique classifying map.
	The vector bundle \(\overline{N}_{k,\beta}\) satisfies
	\begin{enumerate}
		\item\label{item:PropertiesNkA:Nk0Nk}
			If \(\beta=0\) and \(k\geq 3\), \(\overline{N}_{k,\beta}=\pi^*\overline{N}_k\).
		\item\label{item:PropertiesNkA:ForgettingMarkedPoint}
			There is a short exact sequence
			\begin{diag}\matrix[mat, column sep=large](m) {
					0 & p_{k+1}^*\UniversalSpinorBundleM & \overline{N}_{k+1,\beta} & f^*\overline{N}_{k,\beta} & 0.\\
				} ;
				\path[pf] {
					(m-1-1) edge (m-1-2)
					(m-1-2) edge (m-1-3)
					(m-1-3) edge (m-1-4)
					(m-1-4) edge (m-1-5)
				};
			\end{diag}
		\item\label{item:PropertiesNkA:Glueing}
			Let
			\begin{equation}
				\iota\colon \ModuliStackSpinMaps{0,k_1+1}{X, \beta_1}\times_X\ModuliStackSpinMaps{0,k_2+1}{X, \beta_2}\to \ModuliStackSpinMaps{0,k_1+1}{X, \beta_1}\times\ModuliStackSpinMaps{0,k_2+1}{X, \beta_2}.
			\end{equation}
			be the canonical inclusion.
			Then
			\begin{equation}
				gl_X^*\overline{N}_{k_1+k_2,\beta_1+\beta_2}= \iota^*\left(\overline{N}_{k_1+1,\beta_1}\oplus \overline{N}_{k_2+1,\beta_2}\right).
			\end{equation}
		\item\label{item:PropertiesNkA:Symk}
			It holds \(s_\sigma^*\overline{N}_{k,\beta}=\overline{N}_{k,\beta}\).
	\end{enumerate}
\end{conjecture}
We will discuss in Section~\ref{SSec:SNMotivation} how this conjecture is part of a more general conjecture on the existence of moduli stack of super stable maps.

\subsection{Motivation}\label{SSec:SNMotivation}
In this section we explain how Definitions~\ref{defn:NC} and~\ref{defn:NCphi}, as well as the Conjectures~\ref{conj:Nk} and~\ref{conj:NkA} are motivated from the super differential geometric study of super stable curves and super stable maps in~\cites{KSY-SJC}{KSY-SQCI}{KSY-TAMSSSMGZ}.
This section is logically independent of most of the rest of the article except for Sections~\ref{SSec:GWMotivation} and~\ref{SSec:Generalizations}.

Supergeometry is a generalization of geometry to spaces with so called odd dimensions for which the coordinates anti-commute.
For background on supergeometry and more references we refer to~\cite{EK-SGSRSSCA}.
A super Riemann surface is a complex supermanifold~\(M\) of dimension \(1|1\) together with a distribution \(\cD\subset \tangent{M}\) such that the commutator induces an isomorphism \(\cD\otimes \cD =\faktor{\tangent{M}}{\cD}\).
By the uniformization of super Riemann surfaces \(\ProjectiveSpace[\C]{1|1}\) is the only super Riemann surface of genus zero.
There is a holomorphic embedding \(\ProjectiveSpace[\C]{1}\to \ProjectiveSpace[\C]{1|1}\) and its normal bundle is given by \(S_{\ProjectiveSpace[\C]{1}}=\cO(1)\).
Moreover, \(\ProjectiveSpace[\C]{1|1}\) is holomorphically split, that is, the supermanifold is completely determined by \(\ProjectiveSpace[\C]{1}\) and the normal bundle \(S_{\ProjectiveSpace[\C]{1}}\).

A Neveu--Schwarz marked point on \(\ProjectiveSpace[\C]{1|1}\) over a base \(B\) is a map \(P\colon B\to \ProjectiveSpace[\C]{1|1}\).
As \(\ProjectiveSpace[\C]{1|1}\) is split, any Neveu--Schwarz marked point is determined by a point \(p\colon B\to \ProjectiveSpace[\C]{1}\) and a section~\(s\in p^*S_{\ProjectiveSpace[\C]{1}}\), see~\cite[Lemma~4.2.2]{KSY-SQCI}.
All summands of the form \(p^*S_C\) in the Definitions~\ref{defn:NC} and~\ref{defn:NCphi} are motivated by Neveu--Schwarz marked points \(P\) of the super Riemann surface \(\ProjectiveSpace[\C]{1|1}\).

In~\cite{KSY-SQCI} we have constructed the moduli space \(\SmModuliSpaceSuperCurves{0,k}\) of super Riemann surfaces of genus zero with \(k\) marked points as a superorbifold of dimension \(k-3|k-2\).
That is, \(\SmModuliSpaceSuperCurves{0,k}\) represents, in the sense of Shvarts and Molotkov--Sachse, see~\cites{S-DSP}{M-IDCSM}{S-GAASTS} the functor \(\underline{\SmModuliSpaceSuperCurves{0,k}}\colon \cat{SPoint^{op}}\to \cat{Man}\) from the opposite of the category of superpoints to the category of smooth manifolds given by
\begin{equation}
	\underline{\SmModuliSpaceSuperCurves{0,k}}(B)
	= \faktor{\Set{P_1, \dotsc, P_k\colon B\to \ProjectiveSpace[\C]{1|1}\given P_i\times_B\R^{0|0}\text{ and }P_j\times_B\R^{0|0}\text{ distinct}}}{\Aut_B(\ProjectiveSpace[\C]{1|1})}.
\end{equation}
The automorphism group \(\Aut_B(\ProjectiveSpace[\C]{1|1})\) is an extension of the group of Möbius transformations and has complex dimension \(3|2\).

For any three \(B\)-points \(P_1\), \(P_2\), \(P_3\) of \(\ProjectiveSpace[\C]{1|1}\) there is a unique automorphism of \(\ProjectiveSpace[\C]{1|1}\) sending \(P_1\) to \(0\), \(P_2\) to \(1_\epsilon\) and \(P_3\) to \(\infty\), see~\cites[Section~3.1]{KSY-SQCI}[Chapter~2, 2.12]{M-TNCG}.
Here \(0\) and \(\infty\) are the lift of the corresponding point in \(\ProjectiveSpace[\C]{1}\) to \(\ProjectiveSpace[\C]{1|1}\) such that its spinor part vanishes.
The point \(1_\epsilon\) is the lift of the point \(1\) in \(\ProjectiveSpace[\C]{1}\) such that its spinor part is given by \(\epsilon\).
The spinor part \(\epsilon\) depends on the triple \((P_1, P_2, P_3)\) and is well defined up to sign.
Consequently,
\begin{equation}
	\SmModuliSpaceSuperCurves{0,k}
	= \faktor{\left.{\left(\ProjectiveSpace[\C]{1|1}\right)}^k\right|_{Z_k}\hspace{.5em}}{\Aut(\ProjectiveSpace[\C]{1|1})}
	= \faktor{\C^{0|1}\times \left.{\left(\ProjectiveSpace[\C]{1|1}\right)}^{k-3}\right|_{\tilde{Z}_{k-3}}\hspace{.5em}}{\Z_2}
\end{equation}
Here \(Z_k\subset{\left(\ProjectiveSpace[\C]{1}\right)}^k\) is the submanifold such that no two projections to \(\ProjectiveSpace[\C]{1}\) coincide and \(\tilde{Z}_{k-3}\subset{\left(\ProjectiveSpace[\C]{1}\right)}^{k-3}\) is the submanifold such that no two projections coincide nor agree with \(0\), \(1\) and \(\infty\).

The superorbifold structure on \(\SmModuliSpaceSuperCurves{0,k}\) extends the manifold structure on the differential geometric moduli space of curves of genus zero with \(k\) marked points
\begin{equation}
	\SmModuliSpaceCurves{0,k}=\faktor{\left.{\left(\ProjectiveSpace[\C]{1}\right)}^k\right|_{Z_k}}{\Aut_\C(\ProjectiveSpace[\C]{1})},
\end{equation}
that is, there is an embedding \(\SmModuliSpaceCurves{0,k}\to \SmModuliSpaceSuperCurves{0,k}\).
We have calculated the normal bundle~\(N^{sm}_k\) to that embedding (up to the \(\Z_2\)-action) in~\cite{KSY-TAMSSSMGZ} to be given by
\begin{diag}
	\matrix[mat](m) {
		0 & H^0(S_{\ProjectiveSpace[\C]{1}}) & {\displaystyle \bigoplus_{j=1}^k p_j^*S_{\ProjectiveSpace[\C]{1}}} & N^{sm}_k & 0 \\
		& s & \bigoplus p_j^*s & & \\
	} ;
	\path[pf] {
		(m-1-1) edge (m-1-2)
		(m-1-2) edge (m-1-3)
		(m-1-3) edge (m-1-4)
		(m-1-4) edge (m-1-5)
		(m-2-2) edge[commutative diagrams/mapsto] (m-2-3)
	};
\end{diag}
Here the middle term is the sum of the spinorial part of the marked points \(P_1, \dotsc, P_k\) and the left hand part is the odd part of the automorphism group of \(\ProjectiveSpace[\C]{1|1}\).
This motivates Definition~\ref{defn:NC} in the case of a tree with a single vertex.
To motivate the case of general trees, note that a stable supercurve is obtained by gluing together smooth marked curves of genus zero and the normal bundles add.

While the above constructions were done using the language of super differential geometry, algebro-geometric treatment of moduli problems of super stable curves has made much progress in the recent years, see~\cites{D-LaM}{OV-SMSGZSUSYCRP}{BR-SMSCRP}{FKP-MSSCCLB}{MZ-EHSTIISFTV}.
Besides more general cases a moduli stack \(\ModuliStackSuperCurves{0,k}\) of stable super curves of genus zero with \(k\) Neveu--Schwarz punctures has been constructed as a smooth and proper Deligne--Mumford superstack of dimension \(k-3|k-2\).
As any stable spin curve gives rise to a unique super stable map there is an inclusion \(i\colon \ModuliStackSpinCurves{0,k}\to \ModuliStackSuperCurves{0,k}\) which induces a bijection on \(\Spec\C\)-points.
The inclusion \(i\) gives rise to a normal bundle \(\overline{N}_k\) of rank \(0|k-2\).
It is natural to assume that the differential geometric moduli space \(\SmModuliSpaceSuperCurves{0,k}\) coincides with the open interior of \(\ModuliStackSuperCurves{0,k}\) and consequently also that the bundle \(\overline{N}_k\) coincides with \(N^{sm}_k\).
This is the content of Conjecture~\ref{conj:Nk}.

We now turn to a similar discussion for the moduli space of super stable maps.
That is, we first give a brief summary of the differential geometric results on super stable maps obtained in~\cites{KSY-SJC}{KSY-SQCI}{KSY-TAMSSSMGZ} and then give a conjectural algebro-geometric analogue.
A super \(\targetACI\)-holomorphic curve of genus zero in a Kähler manifold \(X\) is a map \(\Phi\colon \ProjectiveSpace[\C]{1|1}\times B\to X\) such that
\begin{equation}\label{eq:DefnDJBar}
	\DJBar\Phi
	=\frac12\left(1+ \ACI\otimes \targetACI\right)\left.\differential\Phi\right|_{\cD}
	= 0.
\end{equation}
Here \(\ACI\) is the almost complex structure on \(\ProjectiveSpace[\C]{1|1}\) and \(\targetACI\) is the almost complex structure on the target manifold \(X\).
We have defined super \(\targetACI\)-holomorphic curves in~\cite{KSY-SJC} and argued that they are a good generalization of classical \(\targetACI\)-holomorphic maps because the Equation~\eqref{eq:DefnDJBar} implies that the map \(\Phi\) is holomorphic and a critical point of a supergeometric action functional.

In the case at hand, that is, when the domain is \(\ProjectiveSpace[\C]{1|1}\) and the target is Kähler, the map \(\Phi\) is completely determined by a tuple \((\varphi, \psi)\), where \(\varphi\colon \ProjectiveSpace[\C]{1}\times B\to X\) is a map and \(\psi\in \VSec{\dual{S}_{\ProjectiveSpace[\C]{1}}\otimes_\R \varphi^*\tangent{X}}\) is a twisted spinor satisfying
\begin{align}
	\DelJBar \varphi &= 0 &
	\left(1+\ACI\otimes \targetACI\right)\psi &= 0&
	\Dirac\psi &= 0
\end{align}
Here the map \(\varphi\) is given by \(\varphi=\Phi\circ i\), where \(i\colon \ProjectiveSpace[\C]{1}\to \ProjectiveSpace[\C]{1|1}\) is the embedding from above.

In~\cite{KSY-SJC}, we have constructed the moduli space \(\SmModuliSpaceSuperJMaps{X, \beta}\) of super \(\targetACI\)-holomorphic curves with fixed homology class \(\beta=[\im \Phi]\in H_2^+(X)\) as a real supermanifold.
The construction uses an implicit function theorem in the infinite dimensional space of all maps \(\Phi\colon M\to X\).
In order to assure surjectivity of the differential, the linearizations of the \(\DelJBar\)-operator and the Dirac operator \(\Dirac\) need to be surjective, which translates to the conditions on the target manifold \(X\) mentioned at the beginning of Section~\ref{SSec:SUSYNormalBundlesForStableMaps}.
The supermanifold \(\SmModuliSpaceSuperJMaps{X, \beta}\) has real dimension
\begin{equation}
	\dim_\R X + 2\left<c_1(\tangent{X}), \beta\right>|2\left<c_1(\tangent{X}), \beta\right>
\end{equation}
and comes with an embedding \(\SmModuliSpaceJMaps{X, \beta}\to \SmModuliSpaceSuperJMaps{X, \beta}\) of the moduli space \(\SmModuliSpaceJMaps{X, \beta}\) of classical \(\targetACI\)-holomorphic curves of genus zero.
The normal bundle \(N^{sm}_\beta\) to this embedding has fibers \(H^0\left(\dual{S}_{\ProjectiveSpace[\C]{1}}\otimes \varphi^*\tangent{X}\right)\) over \(\varphi\in \SmModuliSpaceJMaps{X, \beta}\).
That is the fibers over \(\varphi\) are the space of fields~\(\psi\).
When the domain is \(\ProjectiveSpace[\C]{1|1}\) and the target is Kähler, the moduli space \(\SmModuliSpaceSuperJMaps{X,\beta}\) is actually (smoothly) split, that is, the supermanifold \(\SmModuliSpaceSuperJMaps{X,\beta}\) is completely determined by \(\SmModuliSpaceJMaps{X, \beta}\) and the normal bundle \(N^{sm}_\beta\).

In~\cite{KSY-SQCI}, we have defined the moduli space of super \(\targetACI\)-holomorphic curves of genus zero with \(k\)-marked points as
\begin{equation}
	\SmModuliSpaceSuperMaps{0,k}{X, \beta}
	= \faktor{\left(\left.{\left(\ProjectiveSpace[\C]{1|1}\right)}^k\right|_{Z_k}\times \SmModuliSpaceSuperJMaps{X, \beta}\right)}{\Aut(\ProjectiveSpace[\C]{1|1})}.
\end{equation}
Ignoring the \(\Integers_2\)-action, this moduli space comes with a smooth embedding \(\SmModuliSpaceMaps{0,k}{X, \beta}\to \SmModuliSpaceSuperMaps{0,k}{X, \beta}\) and its normal bundle is as in the short exact sequence in in Definition~\ref{defn:NCphi}.
The terms in the middle of the short exact sequence come from the spinor parts of the marked points and the map.
The left hand part of the short exact sequences is the spinorial part of the automorphism group of \(\ProjectiveSpace[\C]{1|1}\) which acts on marked points and the map.

More generally, we have constructed moduli spaces \(\SmModuliSpaceSuperMaps{\tau}{X, \beta_\tau}\) of super stable maps of fixed tree type and homology classes \(\beta_\tau\) by gluing moduli spaces of the form \(\SmModuliSpaceSuperMaps{0,k}{X, \beta}\) together along pairs of marked points \(P\) and \(Q\) such that \(\Phi_1(P)=\Phi_2(Q)\).
We could show in~\cite{KSY-SQCI} that the resulting moduli spaces of super stable maps of fixed tree type are superorbifolds of (real) dimension
\begin{equation}
	\dim_\R X + 2\left<c_1(\tangent{X}), \beta\right> - 2\#E_\tau + 2k - 6| 2\left<c_1(\tangent{X}), \beta\right> + 2k - 4.
\end{equation}
It turns out that the condition \(\Phi_1(P)=\Phi_2(Q)\) is automatically satisfied on the level of normal bundles and normal bundles can be constructed by summing over the vertices of the tree \(\tau\).
The union of their point functors
\begin{equation}
	\begin{split}
		\underline{\SmCModuliSpaceSuperMaps{0,k}{X, \beta}}\colon \cat{SPoint^{op}} &\to \cat{Top} \\
		C &\mapsto \bigcup_{(\tau, \beta_\tau)} \underline{\SmModuliSpaceSuperMaps{\tau}{X, \beta_\tau}}(C)
	\end{split}
\end{equation}
can be equipped with a generalization of Gromov topology such that \(\underline{\SmCModuliSpaceSuperMaps{0,k}{X, \beta}}(\R^{0|0})=\SmCModuliSpaceMaps{0,k}{X, \beta}\) is compact and the strata carry the usual orbifold topology.

The algebro-geometric study of super stable maps and their moduli spaces has recently been initiated in~\cite{BMHR-SNM}.
In particular a very general definition of super stable maps has been given and their moduli stacks have been defined.
By analogy to our differential geometric treatment of super stable maps of genus zero with Neveu--Schwarz punctures we expect the following result:
\begin{conjecture}\label{conj:ModuliStackSuperStableMaps}
	Let \(X\) be a convex (non-super) scheme.
	Then the moduli stack \(\ModuliStackSuperMaps{0,k}{X,\beta}\) of super stable maps of genus zero with \(k\) Neveu--Schwarz punctures in \(X\) representing the cohomology class \(\beta\in H_2^+(X)\) is a smooth and proper Deligne--Mumford super stack of dimension
	\begin{equation}
		\dim X + \left<c_1(\tangent{X}), \beta\right> + k - 3| \left<c_1(\tangent{X}), \beta\right> + k - 2.
	\end{equation}
	There is a smooth inclusion \(i\colon \ModuliStackSpinMaps{0,k}{X,\beta}\to \ModuliStackSuperMaps{0,k}{X,\beta}\) inducing a bijection on \(\Spec\C\)-points.
\end{conjecture}
Similarly, we expect that forgetful maps
\begin{align}
	\pi\colon \ModuliStackSuperMaps{0,k}{X,\beta} &\to \ModuliStackSuperCurves{0,k} &
	f\colon \ModuliStackSuperMaps{0,k+1}{X,\beta} &\to \ModuliStackSuperMaps{0,k}{X,\beta}
\end{align}
as well as gluing and relabeling maps compatible with the ones for stable spin maps can be constructed.
The normal bundle to the inclusion \(i\) yields a vector bundle of rank \(\left<c_1(\tangent{X}), \beta\right> + k - 2\).
It is natural to assume that the algebro-geometric description of \(\overline{N}_{k,\beta}\) coincides with the one we have obtained in the differential geometric description.
This is the content of Conjecture~\ref{conj:NkA}.
 
\section{Definition of super Gromov--Witten invariants via localization of the odd directions}\label{Sec:DefnSGW}

In this section we give the definition of super Gromov--Witten invariants and show that they satisfy generalizations of the Kontsevich--Manin axioms.

In Section~\ref{SSec:GWMotivation} we give the motivation for super Gromov--Witten invariants by assuming that we had a supergeometric cohomology theory that allows for torus localization and a torus action that acts on the odd directions while leaving the even directions invariant.
The resulting formula for super Gromov--Witten invariants can be interpreted in the equivariant cohomology of the moduli spaces of classical stable maps \(\ModuliStackMaps{0,k}{X,\beta}\) and used as a definition in Section~\ref{SSec:DefnSuperGWInvariants}.
Super Gromov--Witten invariants are not trivial when the target space is a point.
In Section~\ref{SSec:SGWIofaPoint}, we show that the super Gromov--Witten invariants of a point can be expressed as integrals of descendant classes on \(\ModuliStackCurves{0,k}\).
Generalizations of Kontsevich--Manin axioms for super Gromov--Witten invariants are proven in Section~\ref{SSec:Axioms}.
In Section~\ref{SSec:Generalizations}, we give some comments on possible generalizations.

\subsection{Motivation}\label{SSec:GWMotivation}
In this section we motivate the Definitions~\ref{defn:SGWClass} and~\ref{defn:SGWNumbers} of super Gromov--Witten classes and super Gromov--Witten invariants.
Using a set of rather strong assumptions on existence and properties of a supergeometric equivariant cohomology theory and localization, we end up with the formulas of Definitions~\ref{defn:SGWClass} and~\ref{defn:SGWNumbers} which only depend on Conjectures~\ref{conj:Nk} and~\ref{conj:NkA} and can be evaluated in classical, non-super equivariant cohomology.
The assumptions on supergeometric cohomology theories are suggested by their analogy to the non-super theory.

Some approaches for cohomology theories for supermanifolds can be found in the literature:
In~\cite{VMP-ESG} a supergeometric Riemann--Roch Theorem is proven in a \(K\)-theoretic approach.
For example, the book~\cite{BBH-GS} and also~\cite{P-DRCSV} discuss the supergeometric analogue of de Rham cohomology which is shown in particular to reproduce de Rham cohomology of the underlying reduced space.
A variant of de Rham formalism that also allows to construct integrals over arbitrary codimension is given in~\cite{V-GITSM}.
The preprint~\cite{BMHR-SNM} defines a supergeometric Chow group in the cases that the odd dimension is either zero or one.
Yet none of those approaches did seem to give a suitable definition of super Gromov--Witten invariants.

To motivate our definition of super Gromov--Witten invariants, assume Conjecture~\ref{conj:ModuliStackSuperStableMaps} that is, the construction of the moduli stack \(\ModuliStackSuperMaps{0,k}{X, \beta}\) of super stable maps to the classical, non-super, convex target scheme \(X\) together with maps
\begin{diag}\matrix[mat](m) {
		\ModuliStackSpinMaps{0,k}{X, \beta} && & X^k \\
		& \ModuliStackSuperMaps{0,k}{X, \beta} && \\
		\ModuliStackSpinCurves{0,k} && & \\
		& \ModuliStackSuperCurves{0,k} && \\
	} ;
	\path[pf] {
		(m-1-1) edge node[auto]{\(i\)} (m-2-2)
			edge node[auto]{\(\ev\)} (m-1-4)
			edge node[auto]{\(\pi\)} (m-3-1)
		(m-3-1) edge (m-4-2)
		(m-2-2) edge node[auto, swap]{\(\Ev\)} (m-1-4)
						edge node[auto]{\(\Pi\)} (m-4-2)
	};
\end{diag}
If we assume furthermore, that we had good supergeometric cohomology rings \(SH^\bullet(\ModuliStackSuperCurves{0,k})\) and \(SH^\bullet(\ModuliStackSuperMaps{0,k}{X,\beta})\) with the usual operations, the natural definition of super Gromov--Witten invariants would be
\begin{equation}\label{eq:MotivationSGWI}
	\left<SGW^{X, \beta}_{0,k}\right>(\alpha_1, \dotsc, \alpha_k)
	= \int_{\ModuliStackSuperMaps{0,k}{X, \beta}} \Ev^*\left(\alpha_1\otimes \dotsm\otimes \alpha_k\right).
\end{equation}

In~\cite{KSY-TAMSSSMGZ}, we have constructed group actions of the torus \(K=\C^*\) on all moduli spaces \(\mathcal{M}_\tau(X, \beta_\tau)\) such that the fixed point locus is given by the embedding \(\SmModuliSpaceMaps{\tau}{X, \beta_\tau}\to \SmModuliSpaceSuperMaps{\tau}{X, \beta_\tau}\) of the classical counterparts.
Roughly, the torus action is given by the multiplication of the spinorial part of the marked points and the maps by the torus element.
We assume, for a moment, that this torus action has an algebraic counterpart
\begin{equation}
	a\colon K\times \ModuliStackSuperMaps{0,k}{X, \beta} \to \ModuliStackSuperMaps{0,k}{X, \beta}
\end{equation}
such that the fixed point locus is given by \(\ModuliStackSpinMaps{0,k}{X, \beta}\).

Torus localization with respect to this torus action would yield an isomorphism of the equivariant cohomologies:
\begin{equation}
	\begin{split}
		SH_K^\bullet\left(\ModuliStackSuperMaps{0,k}{X, \beta}\right)\otimes \C(\kappa) &\to H_K^\bullet(\ModuliStackSpinMaps{0,k}{X, \beta})\otimes \C(\kappa) \\
		\alpha &\mapsto \frac{i^*\alpha}{2e^K(\overline{N}_{k, \beta})}
	\end{split}
\end{equation}
Here \(\kappa\) is the equivariant character of \(K\) and \(e^K(\overline{N}_{k, \beta})\) the equivariant Euler class of the normal bundle \(\overline{N}_{k, \beta}\) to the embedding \(i\colon \ModuliStackSpinMaps{0,k}{X, \beta}\to \ModuliStackSuperMaps{0,k}{X, \beta}\).
The \(2\) in the denominator is the order of the isotropy group \(\Z_2\) of the fixed point set \(\ModuliStackSpinMaps{0,k}{X,\beta}\).
The integral in Equation~\eqref{eq:MotivationSGWI} reduces to
\begin{equation}\left<SGW^{X, \beta}_{0,k}\right>(\alpha_1, \dotsc, \alpha_k)
	= \int_{\ModuliStackSpinMaps{0,k}{X, \beta}} \frac{\ev^*\left(\alpha_1\otimes \dotsm\otimes \alpha_k\right)}{2e^K(\overline{N}_{k, \beta})}.
\end{equation}
We use the right hand side of the above equation in the definition of super Gromov--Witten invariants which can be evaluated in classical algebraic geometry.

Using the same localization argument, we motivate Gromov--Witten classes via:
\begin{equation}
	SGW^{X, \beta}_{0,k}(\alpha_1, \dotsc, \alpha_k)
	= \Pi_*\left(\Ev_1^*\alpha_1\cup \dotsm\cup \Ev_k^*\alpha_k\right)
	= \pi_*\left(\frac{\ev_1^*\alpha_1\cup \dotsm \cup \ev_k^*\alpha_k}{2e^K(\overline{N}_{k, \beta})}\right)
\end{equation}

When defining the Gromov--Witten classes and Gromov--Witten invariants below, we will ignore the torus action on the classes \(\ev_i^*\alpha_i\) and consider them as the classes \(\ev_i^*\alpha_i\otimes 1\in H^\bullet(\ModuliStackSpinMaps{0,k}{X, \beta})\otimes\C(\kappa)\).
Consequently, the resulting Gromov--Witten classes and numbers will be monomials in \(\kappa\) with strictly negative exponent in contrast to what the above motivation via torus localization suggests.
We will see in the next Section~\ref{SSec:DefnSuperGWInvariants}, that this interpretation allows to define invariants that pick up information from the normal bundles \(\overline{N}_{k, \beta}\).
We hope that our invariants are an approximation to results of a suitable, yet to be developed supergeometric intersection theory for supermoduli spaces.

\subsection{Definition of super Gromov--Witten invariants}\label{SSec:DefnSuperGWInvariants}
The forgetful maps \(F\colon \ModuliStackSpinMaps{0,k}{X,\beta}\to\ModuliStackMaps{0,k}{X,\beta}\) and \(F\colon \ModuliStackSpinCurves{0,k}\to \ModuliStackCurves{0,k}\) induce isomorphisms on the coarse moduli spaces.
Consequently, the rational cohomologies \(H^\bullet(\ModuliStackSpinMaps{0,k}{X,\beta}, \RationalNumbers)\) and \(H^\bullet(\ModuliStackMaps{0,k}{X,\beta},\RationalNumbers)\) as well as \(H^\bullet(\ModuliStackSpinCurves{0,k},\RationalNumbers)\) and \(H^\bullet(\ModuliStackCurves{0,k},\RationalNumbers)\) are isomorphic.
We use here Chow cohomology with rational coefficients and a cohomological grading, that is, a cycle of codimension \(d\) induces a class in \(H^{2d}\).
For a brief introduction see~\cite[Chapter~V]{M-FMQCMS}.
The cup product is denoted \(\cup\).

The moduli stacks \(\ModuliStackCurves{0,k}\) and \(\ModuliStackMaps{0,k}{X,\beta}\) carry a trivial \(K=\C^*\) action leaving all points invariant.
The equivariant cohomology of those spaces is hence given by \(H^\bullet(\ModuliStackCurves{0,k},\C)\otimes \C[\kappa]\) and \(H^\bullet(\ModuliStackMaps{0,k}{X,\beta}, \C)\otimes\C[\kappa]\).
The variable of the polynomial ring~\(\C[\kappa]\) is the equivariant character of \(K\) given by the identity \(\kappa\colon K=\C^*\to \C\).
For a brief introduction to equivariant cohomology in algebraic geometry, we refer to~\cite[Section~9.1]{CK-MSAG}.

We will use the following slightly larger equivariant cohomology
\begin{align}
	H_K^\bullet\left(\ModuliStackCurves{0,k}\right) &= H^\bullet\left(\ModuliStackCurves{0,k},\RationalNumbers\right)\otimes\C(\kappa),\\
	H_K^\bullet\left(\ModuliStackMaps{0,k}{X, \beta}\right) &= H^\bullet\left(\ModuliStackMaps{0,k}{X, \beta}, \RationalNumbers\right)\otimes\C(\kappa).
\end{align}
that is, the field of fractions of polynomials in the equivariant character \(\kappa\) with coefficients in cohomology.
This equivariant cohomology has two gradings, a cohomological degree~\(d_c\) and a polynomial degree~\(d_p\).
\begin{align}
	H_K^\bullet\left(\ModuliStackCurves{0,k}\right)
	&= \bigoplus_{d_c=0}^{2d_{k}}\bigoplus_{d_p=-\infty}^\infty H^{d_c}\left(\ModuliStackCurves{0,k}, \C\right)\otimes \kappa^{d_p}, \\
	H_K^\bullet\left(\ModuliStackMaps{0,k}{X,\beta}\right)
	&= \bigoplus_{d_c=0}^{2d_{k, \beta}}\bigoplus_{d_p=-\infty}^\infty H^{d_c}\left(\ModuliStackMaps{0,k}{X,\beta}, \C\right)\otimes \kappa^{d_p}.
\end{align}
The cup product on cohomology \(H^\bullet\) extends to a product on \(H_K^\bullet\) and yields a \(\C(\kappa)\)-algebra structure which is compatible with the gradings.

We equip the vector bundles \(\overline{N}_k\to \ModuliStackSpinCurves{0,k}\) and \(\overline{N}_{k, \beta}\to \ModuliStackSpinMaps{0,k}{X,\beta}\) with the \(K\)-action given by multiplication of the fibers by \(t\in \C^*=K\).
Hence the vector bundles \(\overline{N}_k\) and \(\overline{N}_{k, \beta}\) are \(K\)-equivariant with respect to the trivial action on the base and consider their equivariant Euler classes
\begin{align}
	e^K(\overline{N}_k) &\in H_K^\bullet\left(\ModuliStackCurves{0,k}\right), &
	e^K(\overline{N}_{k, \beta}) &\in H_K^\bullet\left(\ModuliStackMaps{0,k}{X,\beta}\right).
\end{align}

Notice that the equivariant Euler class \(e^K(\overline{N}_{k, \beta})\) is invertible in \(H_K^\bullet(\ModuliStackMaps{0,k}{X,\beta})\) because it is a nilpotent perturbation of an invertible polynomial:
By the splitting principle we can assume that the bundle \(\overline{N}_{k, \beta}\) is a sum of complex line bundles \(N_1\oplus\dotsc\oplus N_{r_{k, \beta}}\) with \(c_1(N_j) = n_j\in H^2(\ModuliStackMaps{0,k}{X,\beta})\).
As the action of the torus \(K=\C^*\) is rescaling of the fibers by \(t\in K\), we have that the equivariant Chern roots are given by \(c_1^K(N_j)=n_j + \kappa\) and hence
\begin{equation}
	\begin{split}
		e^K(\overline{N}_{k,\beta})
		&= c_1^K(N_1)\cup \dotsm \cup c_1^K(N_{r_{k, \beta}})
		= (n_1 + \kappa) \cup \dotsm \cup (n_{r_{k, \beta}} + \kappa) \\
		&= \sum_{i=0}^{r_{k, \beta}} \sigma_{r_{k, \beta}-i}(n_1, \dotsc, n_{r_{k, \beta}}) \kappa^i
	\end{split}
\end{equation}
where \(\sigma_i(n_1, \dots, n_{r_{k, \beta}})\) is the \(i\)-th elementary symmetric polynomial in the Chern roots~\(n_j\).
The summands \(\sigma_{r_{k, \beta}-i}(n_1, \dotsc, n_{r_{k, \beta}})\kappa^i\) have cohomological degree \(d_c=2(r_{k,\beta}-i)\) and polynomial degree \(d_p=i\).
The term with \(d_p=0\) equals the non-equivariant Euler class of \(\overline{N}_{k, \beta}\) and the term with \(d_c=0\) is \(\kappa^{r_{k, \beta}}\).

The top term \(\kappa^{r_{k, \beta}}\) is invertible as we are working with \(\C(\kappa)\) and all terms involving \(n_j\) are nilpotent hence \(e^K(\overline{N}_{k, \beta})\) is invertible.
An explicit formula for the inverse of the equivariant Euler class can be given as follows:
\begin{equation}\label{eq:InverseEulerClass}
	\begin{split}
		\frac1{e^K(\overline{N}_{k, \beta})}
		&= \frac1{(n_1 + \kappa)\cup \dotsm \cup (n_{r_{k, \beta}} + \kappa)}
		= \frac1{\kappa^{r_{k, \beta}}}\frac1{(\frac{n_1}{\kappa} + 1)\cup \dotsm \cup (\frac{n_{r_{k, \beta}}}{\kappa} + 1)} \\
		&= \frac1{\kappa^{r_{k, \beta}}}\sum_{i_1\geq 0}{\left(-\frac{n_1}{\kappa}\right)}^{i_1}\cup \dotsm \cup \sum_{i_{r_{k, \beta}}\geq0}{\left(-\frac{n_{r_{k, \beta}}}{\kappa}\right)}^{i_{r_{k, \beta}}} \\
		&= \sum_{j\geq0}{(-1)}^j\frac1{\kappa^{r_{k, \beta}+j}}\sum_{\substack{i_1\geq 0, \dotsc, i_{r_{k, \beta}}\geq 0\\ i_1+\dotsb+ i_{r_{k, \beta}}=j}}{\left(n_1\right)}^{i_1}\cup \dotsm \cup {\left(n_{r_{k, \beta}}\right)}^{i_{r_{k, \beta}}}
	\end{split}
\end{equation}
All sums are finite because \(n_j\in H^2(\ModuliStackMaps{0,k}{X,\beta})\) are nilpotent in \(H_K^\bullet(\ModuliStackMaps{0,k}{X,\beta})\).
All summands are homogeneous, that is for any choice of \((i_1, \dotsc, i_{r_{k, \beta}})\) with \(i_1 + \dotsb + i_{r_{k, \beta}}=j\) we have
\begin{equation}
	d_c\left({\left(n_1\right)}^{i_1}\cup \dotsm \cup {\left(n_{r_{k, \beta}}\right)}^{i_{r_{k, \beta}}}\right) = 2j \\
\end{equation}
The summand of highest polynomial degree is \(\kappa^{-r_{k, \beta}}\) for \(i_1=\dotsb=i_{r_{k, \beta}}=0\).

The equivariant Euler class of \(\overline{N}_k\) is invertible in \(H_K^\bullet(\ModuliStackCurves{0,k})\) by the same arguments.

\begin{defn}\label{defn:SGWClass}
	For \(k\geq 3\), the super Gromov--Witten class is the multi-linear map
	\begin{equation}
		\begin{split}
			SGW_{0,k}^{X, \beta}\colon {(H^\bullet(X))}^k &\to H_K^\bullet(\ModuliStackCurves{0,k})\\
			(\alpha_1, \dotsc, \alpha_k) &\mapsto \pi_*\frac{\ev_1^*\alpha_1 \cup \dotsm \cup \ev_k^*\alpha_k}{2e^K(\overline{N}_{k, \beta})}
		\end{split}
	\end{equation}
	Here the product \(\ev_1^*\alpha_1 \cup \dotsm \cup \ev_k^*\alpha_k\) is the cup product of cohomology classes in \(H^\bullet(\ModuliStackMaps{0,k}{X,\beta})\) and \(\pi_*\colon H_K^\bullet(\ModuliStackMaps{0,k}{X,\beta})\to H_K^\bullet(\ModuliStackCurves{0,k})\) is the pushforward along
	\begin{equation}
		\pi\colon \ModuliStackMaps{0,k}{X,\beta}\to \ModuliStackCurves{0,k}.
	\end{equation}
\end{defn}

As the super Gromov--Witten classes are linear in \(\alpha_i\), we can assume without loss of generality that the classes \(\alpha_i\) are homogeneous with respect to cohomological degree.
In that case,
\begin{equation}
	SGW_{0,k}^{X, \beta}(\alpha_1, \dotsc, \alpha_k) \in \bigoplus_j H^{\deg\alpha + 2\left(j - d_{k, \beta} + d_k\right)}(\ModuliStackCurves{0,k})\otimes \kappa^{-r_{k, \beta}-j}\subset H_K^{\bullet}(\ModuliStackCurves{0,k})
\end{equation}
with \(\deg \alpha=\sum_i \deg \alpha_i\).
Here the index \(j\) runs over all non-negative integers such that \(0\leq\deg\alpha + 2\left(j - d_{k, \beta} + d_k\right)\leq 2d_k\).

\begin{defn}\label{defn:SGWNumbers}
	For \(k\geq 1\) and \((\alpha_1, \dotsc, \alpha_k)\in {\left(H^\bullet(X)\right)}^k\) we define the super Gromov--Witten numbers by
	\begin{equation}\left<SGW_{0,k}^{X, \beta}\right>(\alpha_1, \dotsc, \alpha_k)
		= \int_{\ModuliStackMaps{0,k}{X,\beta}} \frac{\ev_1^*\alpha_1 \cup \dotsm \cup \ev_k^*\alpha_k}{e^K(\overline{N}_{k, \beta})}
		\in \C(\kappa)
	\end{equation}
\end{defn}
For \(k\geq 3\), the super Gromov--Witten numbers are integrals over super Gromov--Witten classes:
\begin{equation}\begin{split}
		\MoveEqLeft
		\left<SGW_{0,k}^{X, \beta}\right>(\alpha_1, \dotsc, \alpha_k)
		= \int_{\ModuliStackMaps{0,k}{X,\beta}} \frac{\ev_1^*\alpha_1 \cup \dotsm \cup \ev_k^*\alpha_k}{e^K(\overline{N}_{k, \beta})} \\
		&= \int_{\ModuliStackCurves{0,k}} \pi_*\frac{\ev_1^*\alpha_1 \cup \dotsm \cup \ev_k^*\alpha_k}{e^K(\overline{N}_{k, \beta})}
		= \int_{\ModuliStackSpinCurves{0,k}} \pi_*\frac{\ev_1^*\alpha_1 \cup \dotsm \cup \ev_k^*\alpha_k}{2e^K(\overline{N}_{k, \beta})} \\
		&=\int_{\ModuliStackSpinCurves{0,k}} SGW_{0,k}^{X, \beta}(\alpha_1, \dotsc, \alpha_k)
	\end{split}
\end{equation}
Note that \([\ModuliStackSpinCurves{0,k}]=2[\ModuliStackCurves{0,k}]\) in \(H_K^\bullet(\ModuliStackCurves{0,k})\) because the generic point of \(\ModuliStackSpinCurves{0,k}\) has isotropy group \(\Integers_2\).

Using~\eqref{eq:InverseEulerClass} to expand the inverse of the equivariant Euler class, we obtain:
\begin{equation}\begin{split}
		\left<SGW_{0,k}^{X, \beta}\right>(\alpha_1, \dotsc, \alpha_k)
		= {(-1)}^{j}\kappa^{-r_{k, \beta}-j}&\int_{\ModuliStackMaps{0,k}{X,\beta}} \ev_1^*\alpha_1\cup\dotsm \cup \ev_k^*\alpha_k \\
			&\cup\sum_{\substack{i_1\geq 0, \dotsc, i_{r_{k, \beta}}\geq 0\\ i_1+\dotsb+ i_{r_{k, \beta}}=c}} {\left(n_1\right)}^{i_1}\cup \dotsm \cup {\left(n_{r_{k, \beta}}\right)}^{i_{r_{k, \beta}}}
	\end{split}
\end{equation}
where \(2c=2d_{k, \beta}-\deg \alpha\).
In particular, if \(\deg \alpha\) is odd or \(\deg\alpha\geq 2d_{k, \beta}\), the super Gromov--Witten numbers vanish.
But, in contrast to classical Gromov--Witten numbers the super Gromov--Witten numbers \(\left<SGW_{0,k}^{X, \beta}\right>(\alpha_1, \dotsc, \alpha_k)\) may be non-zero if \(0\leq\deg \alpha\leq2d_{k, \beta}\).

For low values of \(c\) we obtain:
\begin{itemize}
	\item[\(c=0\)]
		\begin{equation}
			\left<SGW_{0,k}^{X, \beta}\right>(\alpha_1, \dotsc, \alpha_k)
			= \kappa^{-r_{k, \beta}}\int_{\ModuliStackMaps{0,k}{X,\beta}} \ev_1^*\alpha_1\cup\dotsm \cup \ev_k^*\alpha_k
		\end{equation}
		Up to the polynomial prefactor this are classical Gromov--Witten invariants.
	\item[\(c=1\)]
		\begin{equation}
				\left<SGW_{0,k}^{X, \beta}\right>(\alpha_1, \dotsc, \alpha_k)
				= -\kappa^{-r_{k, \beta}-1}\int_{\ModuliStackMaps{0,k}{X,\beta}} \ev_1^*\alpha_1\cup\dotsm \cup \ev_k^*\alpha_k\cup c_1(\overline{N}_{k, \beta})
		\end{equation}
	\item[\(c=2\)]
		\begin{equation}
			\begin{split}
				\MoveEqLeft
				\left<SGW_{0,k}^{X, \beta}\right>(\alpha_1, \dotsc, \alpha_k) \\
				&= \kappa^{-r_{k, \beta}-2}\int_{\ModuliStackMaps{0,k}{X,\beta}} \ev_1^*\alpha_1\cup\dotsm \cup \ev_k^*\alpha_k\cup\left({\left(c_1(\overline{N}_{k, \beta})\right)}^2 - c_2(\overline{N}_{k, \beta})\right)
			\end{split}
		\end{equation}
	\item[\(c=3\)]
		\begin{equation}
			\begin{split}
				\MoveEqLeftR
				\left<SGW_{0,k}^{X, \beta}\right>(\alpha_1, \dotsc, \alpha_k) \\
				={}& -\kappa^{-r_{k, \beta}-3}\int_{\ModuliStackMaps{0,k}{X,\beta}} \ev_1^*\alpha_1\cup\dotsm \cup \ev_k^*\alpha_k \\
					&\hspace{7em}\cup\left({\left(c_1(\overline{N}_{k, \beta})\right)}^3 - 5c_2(\overline{N}_{k, \beta})c_1(\overline{N}_{k, \beta}) - 2c_3(\overline{N}_{k, \beta})\right)
			\end{split}
		\end{equation}
\end{itemize}

\subsection{Super Gromov--Witten invariants of a point}\label{SSec:SGWIofaPoint}
The simplest case for which super Gromov--Witten numbers are non-trivial is the case of the target space being a point.
In that case \(\ModuliStackMaps{0,k}{X,\beta}=\ModuliStackCurves{0,k}\) as there is precisely one map to a point and \(H_2(\pt)=\emptyset\).
This motivates:
\begin{defn}
	For \(k\geq 3\), we call
	\begin{equation}
		\left<SGW^{\pt}_{0,k}\right>
		= \int_{\ModuliStackCurves{0,k}} \frac1{e^K(\overline{N_k})}
	\end{equation}
	the \(k\)-point super Gromov--Witten numbers of a point.
\end{defn}

\begin{prop}\label{prop:SGWPoint}
	The \(k\)-point super Gromov--Witten numbers of a point are given by
	\begin{equation}
		\left<SGW^{\pt}_{0,k}\right>
		= \frac{{\left(-1\right)}^{k-3}}{2^{k-3}\kappa^{2k-5}}\sum_{\substack{i_4\geq 0, \dotsc, i_{k}\geq 0\\ i_4+\dotsb+ i_{k}=k-3}} \int_{\ModuliStackCurves{0,k}} {\left(f^*\right)}^{k-4}\psi_4^{i_4}\cup \dotsm \cup f^*\psi_{k-1}^{i_{k-1}} \cup \psi_{k}^{i_k},
	\end{equation}
	where the \(\psi_j=c_1(p_j^*\omega)\in H^2(\ModuliStackCurves{0,k})\).
\end{prop}
\begin{proof}
	By Conjecture~\ref{conj:Nk},~\ref{item:PropertiesNk:N3N4}, the vector bundle \(\overline{N}_3\) is trivial of rank one and the \(K\)-action given by multiplication.
	Using Conjecture~\ref{conj:Nk},~\ref{item:PropertiesNk:ForgettingMarkedPoint} and \(f^*\UniversalSpinorBundleC_j=\UniversalSpinorBundleC_j\) one sees
	\begin{equation}
		\overline{N}_k = \C\oplus\bigoplus_{j=4}^k {\left(f^*\right)}^{k-j}\UniversalSpinorBundleC_j.
	\end{equation}
	for \(k>3\).
	For any \(j>3\) we have \(e^K(\UniversalSpinorBundleC_j) = \kappa - \frac12\psi_j\) where \(\psi_j=c_1(\omega_j)\) because \(\UniversalSpinorBundleC_j\otimes \UniversalSpinorBundleC_j=\dual{\omega}_j\).
	Hence, using a calculation similar to~\eqref{eq:InverseEulerClass}, we obtain
	\begin{equation}
		\frac1{e^K(\overline{N}_{k+1})}
		= \sum_{j\geq0}\frac{{\left(-1\right)}^j}{2^j\kappa^{k-2+j}}\sum_{\substack{i_4\geq 0, \dotsc, i_{k}\geq 0\\ i_4+\dotsb+ i_{k}=j}} {\left(f^*\right)}^{k-4}\psi_4^{i_4}\cup \dotsm \cup f^*\psi_{k-1}^{i_{k-1}} \cup \psi_{k}^{i_k}
	\end{equation}
	and the claim follows because only terms with \(j=k-3\) are of top degree in cohomology.

	Note that if \(i_4+\dotsb + i_l > l-3\) for some \(l< k\), the summand vanishes for degree reasons.
\end{proof}

Integrals over tautological classes such as the \(\psi\)-classes appearing in the above expression for the super Gromov--Witten invariants of the point have been calculated in~\cites{K-ITMSCMAF}{W-TDGISTMC}.
For more pedagogical approaches see also~\cites{Z-ITMSCIS}{K-NPC} and for a computer algebra program to calculate with tautological classes, see~\cite{DSZ-admSPCTRMSSC}.
We will calculate as examples the super Gromov--Witten invariants of a point for \(k=3,4,5,6\) marked points.
The method of using projection formula and push-downs of \(\psi\)-classes iteratively as in the examples \(k=5,6\) works in principle for any \(k\).

\begin{ex}[Super Gromov--Witten invariants of a point for \(k=3\) marked points]
	For \(k=3\), the moduli space \(\ModuliStackCurves{0,3}\) is a point.
	Hence,
	\begin{equation}
		\left<SGW^{\pt}_{0,3}\right>
		= \int_{\ModuliStackCurves{0,3}} \frac1{e^K(\overline{N}_{0,3})}
		= \int_{\pt} \frac1{\kappa}
		= \kappa^{-1}.
	\end{equation}
\end{ex}

\begin{ex}[Super Gromov--Witten invariants of a point for \(k=4\) marked points]
	In the case \(k=4\) there is only one possible summand in the sum formula of Proposition~\ref{prop:SGWPoint}.
	Hence, by Proposition~\ref{prop:SGWPoint}:
	\begin{equation}
		\left<SGW^{\pt}_{0,4}\right>
		= \frac{-1}{2\kappa^3}\int_{\ModuliStackCurves{0,4}}\psi_4
		= -\frac1{2\kappa^3}.
	\end{equation}
	Integrals over products of powers of \(\psi\)-classes such as \(\int \psi_4=1\) can be calculated directly using, for example,~\cite[Section~3.3.2]{K-ERCVTA}.
\end{ex}

\begin{ex}[Super Gromov--Witten invariants of a point for \(k=5\) marked points]
	For \(k=5\) there are two summands:
	\begin{equation}
		\left<SGW^{\pt}_{0,5}\right>
		= \frac1{4\kappa^5}\int_{\ModuliStackCurves{0,5}} \left(f^*\psi_4\right) \psi_5 + \psi_5^2
		= \frac1{4\kappa^5}\left(2+1\right)
		= \frac3{4\kappa^5}.
	\end{equation}
	While the second summand \(\int \psi_5^2=1\) can be calculated directly by~\cite[Section~3.3.2]{K-ERCVTA}, we use projection formula for the second summand:
	\begin{equation}
		\int_{\ModuliStackCurves{0,5}} \left(f^*\psi_4\right) \psi_5
		= \int_{\ModuliStackCurves{0,4}} \psi_4 \left(f_*\psi_5\right)
		= 2\int_{\ModuliStackCurves{0,4}} \psi_4
		=2
	\end{equation}
	Here \(f_*\psi_5 = 2[\ModuliStackCurves{0,4}]\) as explained, for example, in~\cite[Section~2.2]{K-NPC}.
\end{ex}
\begin{ex}[Super Gromov--Witten invariants of a point for \(k=6\) marked points]
	\begin{equation}
		\begin{split}
			\left<SGW^{\pt}_{0,6}\right>
			&= \frac{-1}{8\kappa^7}\int_{\ModuliStackCurves{0,6}} \left({\left(f^*\right)}^2\psi_4\right) \left(f^*\psi_5\right) \psi_6 + \left({\left(f^*\right)}^2\psi_4\right) \psi_6^2 + \left(f^*\psi_5\right)\psi_6 + \psi_6^3\\
			&= -\frac1{8\kappa^7}\left(6+2+3+1\right)
			= -\frac3{2\kappa^7}.
		\end{split}
	\end{equation}
	For the first summand we use again the projection formula and the \(\kappa\)-class \(\kappa_0=f_*\psi_{l+1}=\left(l-2\right)[\ModuliStackCurves{0,l}]\):
	\begin{equation}
		\begin{split}
			\int_{\ModuliStackCurves{0,6}} \left({\left(f^*\right)}^2\psi_4\right) \left(f^*\psi_5\right) \psi_6
			&= \int_{\ModuliStackCurves{0,5}} \left(f^*\psi_4\right) \psi_5 \kappa_0
			= 3\int_{\ModuliStackCurves{0,5}} \left(f^*\psi_4\right) \psi_5\\
			&= 3\int_{\ModuliStackCurves{0,4}} \psi_4 \kappa_0
			= 6
		\end{split}
	\end{equation}
	The for the second and third summand we need higher \(\kappa\)-classes \(\kappa_a = f_*\psi_{l+1}^{a+1}\), as well as their pullback formula \(\kappa_a = f^*\kappa_a + \psi_l^a\), see~\cite[Lemma~2.2.3]{K-NPC}:
	\begin{equation}
		\int_{\ModuliStackCurves{0,6}}\left({\left(f^*\right)}^2\psi_4\right) \psi_6^2
		= \int_{\ModuliStackCurves{0,5}} \left(f^*\psi_4\right) \kappa_1
		= \int_{\ModuliStackCurves{0,5}} \left(f^*\psi_4\right) \left(f^*\kappa_1 + \psi_5\right)
		= \int_{\ModuliStackCurves{0,5}} f^*\psi_4 \psi_5
		= 2
	\end{equation}
	\begin{equation}
		\begin{split}
			\int_{\ModuliStackCurves{0,6}}\left(f^*\psi_5\right) \psi_6^2
			&= \int_{\ModuliStackCurves{0,5}} \psi_5 \kappa_1
			= \int_{\ModuliStackCurves{0,5}} \psi_5 \left(f^*\kappa_1+\psi_5\right)
			= \int_{\ModuliStackCurves{0,4}} \kappa_1\kappa_0 + \kappa_1\\
			&= 3\int_{\ModuliStackCurves{0,4}} \kappa_1
			= 3\int_{\ModuliStackCurves{0,5}} \psi_5^2
			= 3
		\end{split}
	\end{equation}
\end{ex}

\subsection{Axioms}\label{SSec:Axioms}
In~\cite[Chapter~2]{KM-GWCQCEG} there is a list of axioms that Gromov--Witten classes and Gromov--Witten numbers satisfy.
In this section we will give their extension to the super Gromov--Witten classes defined in Definition~\ref{defn:SGWClass} and the super Gromov--Witten invariants defined in Definition~\ref{defn:SGWNumbers}.

\begin{axiom}[Linearity]
	Super Gromov--Witten classes and super Gromov--Witten invariants are linear as maps from \({\left(H^\bullet(X)\right)}^k\) to \(H_K^\bullet(\ModuliStackCurves{0,k})\) resp. \(\C(\kappa)\).
\end{axiom}

\begin{axiom}[Grading]\label{axiom:Grading}
	The map \(SGW_{0,k}^{X, \beta}\colon {\left(H^\bullet(X)\right)}^k\to H_K^\bullet(\ModuliStackCurves{0,k})\) is a sum of homogeneous maps of degree \((2j,-r_{k, \beta}-j)\), that is, for homogeneous \(\alpha_i\in H^{\deg \alpha_i}(X)\) we have
	\begin{equation}
		SGW_{0,k}^{X, \beta}(\alpha_1, \dotsc, \alpha_k)\in \bigoplus_j H^{\deg \alpha + 2\left(j - d_{k, \beta} + d_k\right)}(\ModuliStackCurves{0,k})\otimes \kappa^{-r_{k, \beta}-j} \subset H_K^\bullet(\ModuliStackCurves{0,k})
	\end{equation}
	where \(\deg \alpha = \sum_{i=1}^k \deg \alpha_i\) and the index \(j\) runs over non-negative integers such that \(0\leq \deg \alpha + 2\left(j - d_{k, \beta} + d_k\right)\leq 2d_k\).

	The super Gromov--Witten numbers \(\left<SGW_{0,k}^{X, \beta}\right>(\alpha_1, \dotsc, \alpha_k)\) vanish if \(\deg \alpha\) is odd or \(\deg \alpha>2d_{k, \beta}\).
	Otherwise, the super Gromov--Witten number \(\left<SGW_{0,k}^{X, \beta}\right>(\alpha_1, \dotsc, \alpha_k)\) is a monomial in \(\C(\kappa)\) of degree
	\begin{equation}
		-r_{k, \beta} - d_{k, \beta} + \frac12\deg \alpha
		= -\dim X -2\left<c_1(\tangent{X}), \beta\right> - 2k + 5 + \frac12\deg\alpha
		\leq -r_{k,\beta}
		\leq 0.
	\end{equation}
\end{axiom}

\begin{axiom}[Extension of classical Gromov--Witten invariants]\label{axiom:Extension}
	The restriction of the super Gromov--Witten class to \(H^{d_c}\otimes \kappa^{d_p}\) with \(d_c=\deg\alpha - 2d_{k, \beta} + 2d_k\) and \(d_p = -r_{k, \beta}\) yields classical Gromov--Witten classes up to a factor of \(\kappa^{-r_{k, \beta}}\):
	\begin{equation}
		\begin{split}
			\left.SGW_{0,k}^{X, \beta}(\alpha_1, \dotsc, \alpha_k)\right|_{H^{\deg\alpha - 2d_{k, \beta} + 2d_k}\otimes \kappa^{-r_{k, \beta}}}
			&= \kappa^{-r_{k, \beta}}\pi_*\left(\ev_1^*\alpha_1\cup \dotsm \cup \ev_k^*\alpha_k\right) \\
			&= \kappa^{-r_{k, \beta}} GW_{0,k}^{X, \beta}(\alpha_1, \dotsc, \alpha_k)
		\end{split}
	\end{equation}
	If \(\deg\alpha = 2d_{k, \beta}\) the super Gromov--Witten numbers coincide with Gromov--Witten numbers up to the factor \(\kappa^{-r_{k, \beta}}\):
	\begin{equation}
		\begin{split}
			\left<SGW_{0,k}^{X, \beta}\right>(\alpha_1, \dotsc, \alpha_k)
			&= \kappa^{-r_{k, \beta}}\int_{\ModuliStackMaps{0,k}{X,\beta}} \ev_1^*\alpha_1\cup \dotsm \cup \ev_k^*\alpha_k\\
			&= \kappa^{-r_{k, \beta}}\left<GW_{0,k}^{X, \beta}\right>(\alpha_1, \dotsc, \alpha_k)
		\end{split}
	\end{equation}
\end{axiom}
In this sense, all of the following axioms reduce to their classical counterpart when imposing the right restrictions on \(d_c\), \(d_p\) and \(\deg\alpha\).

\begin{axiom}[Effectivity Axiom]
	If \(\beta\in H_2(X)\) is not effective, then \(SGW_{0,k}^{X, \beta}=0\).
\end{axiom}
For every stable map \(\phi\colon C\to X\) the class \([\im \phi]\in H_2(X)\) is effective.
Consequently, the moduli space \(\ModuliStackMaps{0,k}{X,\beta}\) is empty if the class \(\beta\) is not effective and all super Gromov--Witten classes vanish.

\begin{axiom}[\(\SymG{k}\)-equivariance]\label{axiom:SkEquivariance}
The symmetric group \(\SymG{k}\) acts on the super vector space \({\left(H^\bullet(X)\right)}^k\) and on \(H_K^\bullet(\ModuliStackCurves{0,k})\) via permutation of the marked points.
	The super Gromov--Witten class is equivariant with respect to this action.
	That is, for any \(\alpha_i\in H^\bullet(X)\)
	\begin{multline}
		SGW_{0,k}^{X, \beta}(\alpha_1, \dotsc, \alpha_i, \alpha_{i+1}, \dotsc, \alpha_k) \\
		= {\left(-1\right)}^{\deg \alpha_i \cdot \deg \alpha_{i+1}} SGW_{0,k}^{X, \beta}(\alpha_1, \dotsc, \alpha_{i+1}, \alpha_i, \dotsc, \alpha_k).
	\end{multline}
\end{axiom}
This follows from the super commutativity of the product in cohomology and Conjecture~\ref{conj:NkA},~\ref{item:PropertiesNkA:Symk}.

\begin{axiom}[Fundamental class]\label{axiom:FundamentalClass}
	Let \([X]\in H^0(X)\) be the fundamental class of \(X\).
	Then for \(k\geq3\)
	\begin{equation}
		\begin{split}
			\MoveEqLeft
			\left<SGW_{0,k+1}^{X, \beta}\right>(\alpha_1, \dotsc, \alpha_{k}, [X]) \\
			&= \frac1\kappa \int_{\ModuliStackMaps{0,k+1}{X, \beta}} f^*\left(\frac{\ev_1^*\alpha_1\cup \dotsm\cup \ev_k^*\alpha_k}{e^K(\overline{N}_{k, \beta})}\right)\cup \sum_{j>0}{\left(-\frac{c_1(\UniversalSpinorBundleM_{k+1})}{\kappa}\right)}^j.
		\end{split}
	\end{equation}
	Note that the pullback class has at least real codimension two and hence the summation runs over \(j>0\).
\end{axiom}
\begin{proof}
	By Conjecture~\ref{conj:NkA},~\ref{item:PropertiesNk:ForgettingMarkedPoint} it holds \(e^K(\overline{N}_{k+1,A})=e^K(\UniversalSpinorBundleM_{k+1})\cdot e^K(f^*\overline{N}_{k, \beta})\) where \(\UniversalSpinorBundleM_{k+1}\) is equipped with the \(K\) action that multiplies each fiber with \(t\in K\).
	Then
	\begin{equation}
		\begin{split}
			\frac{e^K(f^*\overline{N}_{k, \beta})}{e^K(\overline{N}_{k+1,A})}
			&= \frac1{e^K(\UniversalSpinorBundleM_{k+1})}
			= \frac1{\kappa + c_1(\UniversalSpinorBundleM_{k+1})}
			= \frac1{\kappa\left(1+\frac{c_1(\UniversalSpinorBundleM_{k+1})}{\kappa}\right)} \\
			&= \frac1\kappa \sum_{j\geq 0}{\left(-\frac{c_1(\UniversalSpinorBundleM_{k+1})}{\kappa}\right)}^j
		\end{split}
	\end{equation}
	Note that \(\ev_i\circ f = \ev_i\) for all \(i\leq k\).
	Hence,
	\begin{equation}\label{eq:PullbackQuotient}
		\begin{split}
			\MoveEqLeft
			\frac{\ev_1^*\alpha_1\cup \dotsm \cup \ev_k^*\alpha_k\cup \ev_{k+1}^*\alpha_{k+1}}{e^K(\overline{N}_{k+1,A})}\\
			&= f^*\frac{\ev_1^*\alpha_1\cup\dotsm\cup \ev_k^*\alpha_k}{e^K(\overline{N}_{k, \beta})}\cup\ev_{k+1}^*\alpha_{k+1}\cup \frac{e^K(f^*\overline{N}_{k, \beta})}{e^K(\overline{N}_{k+1,A})} \\
			&= \frac1\kappa f^*\frac{\ev_1^*\alpha_1\cup\dotsm\cup \ev_k^*\alpha_k}{e^K(\overline{N}_{k, \beta})}\cup\ev_{k+1}^*\alpha_{k+1}\cup \sum_{j\geq 0}{\left(-\frac{c_1(\UniversalSpinorBundleM_{k+1})}{\kappa}\right)}^j
		\end{split}
	\end{equation}
	Integrating over \(\ModuliStackMaps{0,k+1}{X, \beta}\) and using \(\alpha_{k+1}=1\) we obtain the result.
\end{proof}
\begin{rem}
	We conjecture that the following stronger Fundamental Class Axiom for super Gromov--Witten classes holds:
	\begin{equation}\label{eq:Axiom:FundamentalClassGWClasses}
		SGW_{0,k+1}^{X, \beta}(\alpha_1, \dotsc, \alpha_{k}, [X])
		= \frac1\kappa f^*SGW_{0,k}^{X, \beta}(\alpha_1, \dotsc, \alpha_{k}) \cup \sum_{j\geq 0}{\left(-\frac{c_1(\UniversalSpinorBundleC_{k+1})}{\kappa}\right)}^j.
	\end{equation}
\end{rem}

\begin{axiom}[Divisor]
	Let \(k\geq3\) and \(f\colon \ModuliStackCurves{0,k+1}\to \ModuliStackCurves{0,k}\) the map forgetting the last marked point.
	Then
	\begin{equation}
		\begin{split}
			\MoveEqLeft
			f_*SGW_{0,k+1}^{X, \beta}(\alpha_1, \dotsc, \alpha_k, \alpha_{k+1}) \\
			&= \frac1\kappa SGW_{0,k}^{X, \beta}(\alpha_1, \dotsc, \alpha_k) \cup \pi_*f_*\left(\ev_{k+1}^*\alpha_{k+1} \cup\sum_{j\geq 0}{\left(-\frac{c_1(\UniversalSpinorBundleM_{k+1})}{\kappa}\right)}^j \right)
		\end{split}
	\end{equation}
\end{axiom}
\begin{proof}
	Applying \(f_*\pi_*\) to~\eqref{eq:PullbackQuotient}, using \(f_*\pi_*=\pi_*f_*\) and the projection formula yields the result.
\end{proof}
To compare this divisor axiom to the one in~\cite{KM-GWCQCEG}, assume  \(\sum_{i=1}^k\deg \alpha_i = 2d_{k, \beta}\) as well as \(\deg \alpha_{k+1}=2\) and obtain using the Extension Axiom~\ref{axiom:Extension}:
\begin{equation}
	\begin{split}
		\MoveEqLeft
		\left.SGW_{0,k+1}^{X, \beta}(\alpha_1, \dotsc, \alpha_{k+1})\right|_{H^{\deg\alpha - 2d_{k+1,\beta} + 2d_{k+1}}\otimes \kappa^{-r_{k+1,\beta}}}\\
		&= \frac1\kappa \left.SGW_{0,k}^{X, \beta}(\alpha_1,\dotsc, \alpha_k)\right|_{H^{\deg\alpha - 2d_{k, \beta} + 2d_{k}}\otimes \kappa^{-r_{k, \beta}}} \pi_*f_*\ev_{k+1}^*\alpha_{k+1} \\
		&= \kappa^{-r_{k+1,\beta}}GW_{0,k}^{X, \beta}(\alpha_1,\dotsc, \alpha_k) \left<\alpha_{k+1}, \beta\right>
	\end{split}
\end{equation}

\begin{axiom}[Mapping to a point]\label{axiom:MappingToPoint}
	Suppose that \(\beta=0\) and \(k\geq 3\).
	The super Gromov--Witten classes are given by
	\begin{equation}
		SGW^{X, 0}_{0,k}(\alpha_1, \dotsc, \alpha_k)
		= \left(\int_X \alpha_1 \cup \dotsm \cup \alpha_k\right)  \frac1{e^K(\overline{N}_k)},
	\end{equation}
	and the super Gromov--Witten numbers are given by
	\begin{equation}
		\left<SGW^{X, 0}_{0,k}\right>(\alpha_1, \dotsc, \alpha_k)
		= \left<SGW^{\pt}_{0,k}\right> \int_X \alpha_1 \cup \dotsm \cup \alpha_k.
	\end{equation}
	Here \(\left<SGW^{\pt}_{0,k}\right>\) is a monomial in \(\C(\kappa)\) of degree \(5-2k\) independent of \(X\), see Section~\ref{SSec:SGWIofaPoint}.
	That is, both super Gromov--Witten classes and super Gromov--Witten numbers vanish unless \(\deg \alpha=\dim_\R X\).
\end{axiom}
\begin{proof}
	When the homology class \(\beta\) vanishes the moduli space of stable maps can be written as a product \(\ModuliStackCurves{0,k}(X, \beta=0)=\ModuliStackCurves{0,k}\times X\).
	The map \(\pi\colon\ModuliStackCurves{0,k}(X, \beta=0)\to \ModuliStackCurves{0,k}\) coincides with projection on the first factor and the evaluation maps coincide with the projection on the second factor.
	The normal bundle \(\overline{N}_{k,\beta=0}\) coincides with \(\pi^*\overline{N}_k\) by Conjecture~\ref{conj:NkA},~\ref{item:PropertiesNkA:Nk0Nk}.
	\begin{equation}
		\begin{split}
			SGW^{X, 0}_{0,k}(\alpha_1, \dotsc, \alpha_k)
			&= \pi_*\frac{\ev_1^*\alpha_1\cup \dotsm\cup \ev_k^*\alpha_k}{e^K(\overline{N}_{k, \beta})}
			= {\left(p_1\right)}_*\frac{p_2^*\alpha_1\cup \dotsm \cup p_2^*\alpha_k}{e^K(\pi^*\overline{N}_k)} \\
			&= \left(\int_X \alpha_1 \cup \dotsm \cup \alpha_k\right)  \frac1{e^K(\overline{N}_k)}.
		\end{split}
	\end{equation}
	Integrating further over \(\ModuliStackCurves{0,k}\) yields
	\begin{equation}
		\left<SGW^{X, 0}_{0,k}\right>(\alpha_1, \dotsc, \alpha_k)
		=\int_X \alpha_1 \cup \dotsm \cup \alpha_k \cdot \int_{\ModuliStackCurves{0,k}} \frac{1}{e^K(\overline{N}_k)}.
	\end{equation}
	The second factor of the last line is by definition \(\left<SGW_{0,k}^{pt}\right>\).
\end{proof}

\begin{axiom}[Splitting Axiom]\label{axiom:SplittingAxiom}
	For non-negative integers \(k_1\), \(k_2\) such that \(k=k_1+k_2\) and effective homology classes \(\beta_1, \beta_2\in H_2(X)\) such that \(\beta=\beta_1+\beta_2\) let \(D=[\ModuliStackMaps{0,k_1+1}{X, \beta_1}\times_X \ModuliStackMaps{0,k_2+1}{X,\beta_2}]\) be the divisor in \(\ModuliStackMaps{0,k}{X, \beta}\) obtained as the image of the map \(gl_X\).
	Recall, \(gl_X\colon \ModuliStackMaps{0,k_1+1}{X,\beta_1}\times_X \ModuliStackMaps{0,k_2+1}{X,\beta_2}\to \ModuliStackMaps{0,k}{X,\beta}\) is the gluing map connecting the last point of the first factor to the first marked point of the second factor to a node if the map agrees on the two marked points.
	Then
	\begin{equation}
		\begin{split}
			\MoveEqLeft
			\int_D \frac{\ev_1^*\alpha_1\cup\dotsb\cup \ev_{k_1+k_2}^*\alpha_k}{e^K(\overline{N}_{k,\beta})}\\
			&= g^{ab} \left<SGW_{0,k_1+1}^{X, \beta_1}\right>(\alpha_1, \dotsc, \alpha_{k_1}, T_a) \left<SGW_{0,k_2+1}^{X, \beta_2}\right>(T_b, \alpha_{k_1+1}, \dotsc, \alpha_k).
		\end{split}
	\end{equation}
	Here \(T_j\in H^*(X)\), \(j=0, \dotsc, \dim H^\bullet(X)-1\) is a homogeneous basis of \(H^\bullet(X)\) with \(T_0=1=[X]\), \(g_{ab}=\int_X T_a \cup T_b\), and \(g^{ab}\) its inverse matrix.
	Note that \(g_{ba}={\left(-1\right)}^{d_c(T_a)d_c(T_b)}g_{ab}\).
	We use the convention to implicitly sum over indices that appear once as an upper and once as a lower index.
\end{axiom}
\begin{proof}
	We first prove a variant of~\cite[Lemma 16]{FP-NSMQC}.
	Consider the following commutative diagram:
	\begin{diag}\matrix[mat, column sep=huge](m){
			\ModuliStackMaps{0,k_1+1}{X, \beta_1}\times \ModuliStackMaps{0,k_2+1}{X, \beta_2} & X^{k+2}\\
			\ModuliStackMaps{0,k_1+1}{X, \beta_1}\times_X \ModuliStackMaps{0,k_2+1}{X, \beta_2} & X^{k+1} \\
			\ModuliStackMaps{0,k}{X,\beta} & X^k \\
		} ;
		\path[pf] {
			(m-1-1) edge node[auto]{\(\rho\times \sigma\)} (m-1-2)
			(m-2-1) edge node[auto]{\((\rho, \sigma)\)} (m-2-2)
				edge node[auto]{\(gl\)} (m-3-1)
				edge node[auto]{\(\iota\)} (m-1-1)
			(m-2-2) edge node[auto]{\(p\)} (m-3-2)
				edge node[auto]{\(\delta\)} (m-1-2)
			(m-3-1) edge node[auto]{\(\ev\)} (m-3-2)
		};
	\end{diag}
	Here
	\begin{align}
		\rho\colon \ModuliStackMaps{0,k_1+1}{X, \beta_1} &\to X^{k_1+1} &
		\sigma\colon \ModuliStackMaps{0,k_2+1}{X, \beta_2} &\to X^{k_2+1}
	\end{align}
	are the evaluation maps, \(p\colon X^{k+1}\to X^k\) the projection omitting the \(k_1+1\)st factor, \(\delta\colon X^{k+1}\to X^{k+2}\) the diagonal embedding doubling the \(k_1+1\)st factor.
	It follows that the lower square is a fiber square.

	Using the above commutative diagram, the class of the diagonal \(\Delta = g^{ab}T_a\times T_b\) as well as Conjecture~\ref{conj:NkA}\ref{item:PropertiesNkA:Glueing} one shows
	\begin{align}
		\MoveEqLeftR
		\iota_* gl_X^*\frac{\ev_1^*\alpha_1\cup\dotsb\cup \ev_{k_1+k_2}^*\alpha_k}{e^K(\overline{N}_{k,\beta})}\\
		={} &\iota_* \frac{gl_X^*\ev^*\left(\alpha_1\times \dotsb \times \alpha_k\right)}{e^K(gl_X^*\overline{N}_{k,\beta})}\\
		={} & \iota_* \left(\iota^*\left({\left(e^K(\overline{N}_{k_1+1,\beta_1})\right)}^{-1} {\left(e^K(\overline{N}_{k_2+1, \beta_2})\right)}^{-1}\right) \cap {(\rho,\sigma)}^*p^*\left(\alpha_1\times \dotsb \times \alpha_k\right)\right)\\
		={} &{\left(e^K(\overline{N}_{k_1+1,\beta_1})\right)}^{-1} {\left(e^K(\overline{N}_{k_2+1, \beta_2})\right)}^{-1}\\
			&\cap \iota_* {(\rho,\sigma)}^*\left(\alpha_1\times \dotsb\times \alpha_{k_1}\times [X] \times \alpha_{k_1+1}\times \dotsb \times \alpha_k\right)\displaybreak[0]\\
		={} &{\left(e^K(\overline{N}_{k_1+1,\beta_1})\right)}^{-1} {\left(e^K(\overline{N}_{k_2+1, \beta_2})\right)}^{-1}\\
			&\cap {(\rho\times \sigma)}^*\delta_*\left(\alpha_1\times \dotsb\times \alpha_{k_1}\times [X] \times \alpha_{k_1+1}\times \dotsb \times \alpha_k\right)\displaybreak[0]\\
		={} &g^{ab}{\left(e^K(\overline{N}_{k_1+1,\beta_1})\right)}^{-1} {\left(e^K(\overline{N}_{k_2+1, \beta_2})\right)}^{-1}\\
			&\cap {(\rho\times \sigma)}^*\left(\alpha_1\times \dotsb\times \alpha_{k_1}\times T_a \times T_b \times \alpha_{k_1+1}\times \dotsb \times \alpha_k\right)\\
		={} & g^{ab} \frac{\rho_1^*\alpha_1 \cup\dotsb \cup \rho_{k_1}^*\alpha_{k_1} \cup \rho_{k_1+1}^*T_a}{e^K(\overline{N}_{k_1+1,\beta_1})}\times \frac{\sigma_1^*T_b \cup \sigma_2^*\alpha_{k_1+1} \cup\dotsb \cup \sigma_{k_2+1}^*\alpha_{k}}{e^K(\overline{N}_{k_2+1,\beta_2})}
	\end{align}

	The claim follows by integrating the above equation over \(D\):
	\begin{equation}
		\begin{split}
			\MoveEqLeftR
			\int_D \frac{\ev_1^*\alpha_1\cup\dotsb\cup \ev_{k_1+k_2}^*\alpha_k}{e^K(\overline{N}_{k,\beta})}\\
			={} &\int_{\ModuliStackMaps{0, k_1+1}{X, \beta_1}\times \ModuliStackMaps{0,k_2}{X, \beta_2}} \iota_* gl_X^*\frac{\ev_1^*\alpha_1\cup\dotsb\cup \ev_{k_1+k_2}^*\alpha_k}{e^K(\overline{N}_{k,\beta})}\\
			={} &\int g^{ab} \frac{\rho_1^*\alpha_1 \cup\dotsb \cup \rho_{k_1}^*\alpha_{k_1} \cup \rho_{k_1+1}^*T_a}{e^K(\overline{N}_{k_1+1,\beta_1})}\times \frac{\sigma_1^*T_b \cup \sigma_2^*\alpha_{k_1+1} \cup\dotsb \cup \sigma_{k_2+1}^*\alpha_{k}}{e^K(\overline{N}_{k_2+1,\beta_2})}\\
			={} &g^{ab} \left<SGW_{0,k_1+1}^{X, \beta_1}\right>(\alpha_1, \dotsc, \alpha_{k_1}, T_a) \left<SGW_{0,k_2+1}^{X, \beta_2}\right>(T_b, \alpha_{k_1+1}, \dotsc, \alpha_{k_1+k_2}).
			\qedhere
		\end{split}
	\end{equation}
\end{proof}

\begin{rem}
	The maps \(gl\) and \(gl_X\) fit into the following commutative diagram:
	\begin{diag}\matrix[mat](m) {
			\ModuliStackMaps{0,k_1+1}{X, \beta_1}\times_X \ModuliStackMaps{0,k_2+1}{X, \beta_2} & \ModuliStackMaps{0,k}{X,\beta} \\
			\ModuliStackMaps{0,k_1+1}{X, \beta_1}\times \ModuliStackMaps{0,k_2+1}{X, \beta_2} & \\
			\ModuliStackCurves{0,k_1+1}\times \ModuliStackCurves{0,k_2+1} & \ModuliStackCurves{0,k} \\
		} ;
		\path[pf] {
			(m-1-1) edge node[auto]{\(\iota\)} (m-2-1)
				edge node[auto]{\(gl_X\)} (m-1-2)
			(m-1-2) edge node[auto]{\(\pi\)} (m-3-2)
			(m-2-1) edge node[auto]{\(\pi\times \pi\)} (m-3-1)
			(m-3-1) edge node[auto]{\(gl\)} (m-3-2)
		};
	\end{diag}
	If one could show
	\begin{equation}
		\begin{split}
			\MoveEqLeft
			gl^*\pi_*\left({\left(e^K(\overline{N}_{k, \beta})\right)}^{-1}\cup\ev_1^*\alpha_1\cup \dotsm \cup \ev_k^*\alpha_k\right) \\
			={}& {\left(\pi\times \pi\right)}_*\iota_*gl_X^*\left({\left(e^K(\overline{N}_{k, \beta})\right)}^{-1}\cup\ev_1^*\alpha_1\cup \dotsm \cup \ev_k^*\alpha_k\right) \\
		\end{split}
	\end{equation}
	the stronger splitting equation for super Gromov--Witten classes
	\begin{equation}
		\begin{split}
			\MoveEqLeft
			gl^*SGW_{0,k}^{X, \beta}(\alpha_1, \dotsc, \alpha_k) \\
			={}& 2\sum_{\beta=\beta_1 + \beta_2} g^{ab} SGW_{0,k_1+1}^{X, \beta_1}(\alpha_1, \dotsc, \alpha_{k_1}, T_a) \otimes SGW_{0,k_2+1}^{X, \beta_2}(T_b, \alpha_{k_1+1}, \dotsc, \alpha_k).
		\end{split}
	\end{equation}
	would follow.
	In the classical case, that is without the inverse of the equivariant Euler class, the above equation is shown using dimension considerations, see~\cite{BM-SSMGMI}, which do not apply in presence of the inverse of the equivariant Euler class.
\end{rem}

\begin{conjecture}[Deformation Axiom]
	Let \(x\colon X\to B\) be a smooth proper map with connected base \(B\) and \(X_b=x^{-1}(b)\).
	The super Gromov--Witten classes
	\begin{equation}
			SGW_{0,k}^{X_b, \beta_b}\colon {(H^\bullet(X_b))}^k \to H_K^\bullet(\ModuliStackCurves{0,k})\\
	\end{equation}
	is independent of \(b\) for locally constant \(\beta_b\in H_2^+(X_b)\).
\end{conjecture}

\subsection{Remarks on Generalizations}\label{SSec:Generalizations}
\begin{rem}[Non-convex targets]
	In principle, it should be possible to generalize the construction of super Gromov--Witten invariants to non-convex targets by using the virtual fundamental classes, see~\cite{BM-SSMGMI}.
	Our restriction to convex target is motivated by simplicity of exposition and the fact that we have needed assumptions similar to convexity in the construction of the moduli spaces of super \(\targetACI\)-holomorphic curves and super stable maps in~\cites{KSY-SJC}{KSY-SQCI}.

	Technically, we are using convexity to assure that the moduli spaces of stable maps are of the correct dimension and that the rank of \(N_{(C, \phi)}\) is constant on the moduli space of stable maps.
	The latter requirement can be preserved while omitting \(H^1(\dual{S}\otimes\phi^*\tangent{X})=0\) when defining \(N_{(C, \phi)}\) as an element in \(K\)-theory.
	To this end, define
	\begin{equation}
		[N_{(C,\phi)}] = \sum_{e\text{ node}}[p_e^*S]+\sum_{i=1}^k[p_i^*S]+\left[R^0c_*\left(\dual{S}\otimes\phi^*\tangent{X}\right)\right] - \left[R^1c_*\left(\dual{S}\otimes\phi^*\tangent{X}\right)\right] - [c_*S]
	\end{equation}
	in \(K\)-theory.

	In the special case that \(X\) is Calabi--Yau, an argument similar to Proposition~\ref{prop:SGWPoint} shows that \(\overline{N}_{k,\beta} = \bigoplus_{j=3}^k {\left(f^*\right)}^{k-j}\UniversalSpinorBundleM_j\) and hence all super Gromov--Witten invariants reduce to classical Gromov--Witten invariants with descendant classes.
\end{rem}
\begin{rem}[Ramond punctures]\label{rem:GeneralizationsRamondPunctures}
	The most general constructions of moduli spaces of super Riemann surfaces and super stable curves include the case where the nodes and marked points can be of Ramond type in addition to the special points of Neveu--Schwarz type considered here, see~\cites{D-LaM}{FKP-MSSCCLB}{MZ-EHSTIISFTV}{BR-SMSCRP}{OV-SMSGZSUSYCRP}.
	For a definition of super stable maps with Ramond punctures see~\cite{BMHR-SNM}.
	Yet, moduli spaces of super \(\targetACI\)-holomorphic curves or super stable maps from a domain with Ramond punctures have, to our knowledge, not yet been constructed.
	The construction of the moduli of super \(J\)-holomorphic curves in~\cite{KSY-SJC} and the construction of moduli spaces of super stable maps in~\cite{KSY-SQCI} do not directly allow to include Ramond punctures because the component field approach has not been developed for super Riemann surfaces with Ramond punctures.

	In particular, the case of stable super curves of genus zero with only Neveu--Schwarz marked points is the only one where the moduli space contains no super stable curves with Ramond nodes.
	That is, to consider moduli spaces of super stable curves or super stable maps of genus higher than zero requires understanding of Ramond nodes.
	For further combinatorial properties of moduli spaces of super stable curves see~\cite{KMW-MSSCTO}.

	The same principles as in Sections~\ref{SSec:SNMotivation} and~\ref{SSec:GWMotivation} can be used to obtain a proposal for super Gromov--Witten invariants.
	Let \(\ModuliStackSuperCurves{0,k_{NS}, k_R}\) be the moduli stack of super curves of genus zero with \(k_{NS}\) marked points of Neveu--Schwarz type and \(k_R\) marked points of Ramond type as well as \(\ModuliStackSuperMaps{g,k_{NS}, k_R}{X, \beta}\) the conjectural moduli stack of stable maps from a prestable super curve of genus zero with \(k_{NS}\) marked points of Neveu--Schwarz type and \(k_R\) marked points of Ramond type and representing the homology class \(\beta\in H_2^+(X)\).
	The inclusions
	\begin{align}
		\ModuliStackSpinCurves{0,k_{NS}, k_R} &\hookrightarrow\ModuliStackSuperCurves{0,k_{NS}, k_R} &
		\ModuliStackSpinMaps{0,k_{NS}, k_R}{X, \beta} &\hookrightarrow\ModuliStackSuperMaps{0,k_{NS}, k_R}{X, \beta}
	\end{align}
	have normal bundles \(\overline{N}_{k_{NS}, k_R}\) and \(\overline{N}_{k_{NS}, k_R, \beta}\) respectively.
	The rank \(\rk\overline{N}_{k_{NS}, k_R}=k_{NS}+\frac12k_R -2\) as well as a few more properties of \(\overline{N}_{k_{NS}, k_R}\) can be obtained from the construction of the moduli stack of super stable curves cited above.
	Properties of \(\overline{N}_{k_{NS}, k_R, \beta}\), even the rank, are unknown to us at the moment.
	Dividing by the equivariant Euler class of the normal bundle \(\overline{N}_{k_{NS}, k_R, \beta}\) as in Definition~\ref{defn:SGWClass}, it should be possible to develop a theory of super Gromov--Witten invariants with marked points of Neveu--Schwarz and Ramond type.
	Particularly interesting in this theory would be the splitting axiom for Ramond marked points because the gluing is not unique as can be seen in~\cite[Section~8]{FKP-MSSCCLB}.
\end{rem}
\begin{rem}[Higher Genus]
	As discussed in Remark~\ref{rem:GeneralizationsRamondPunctures} in order to construct moduli stacks of super stable curves and super stable maps it is necessary to understand Ramond nodes.
	In addition, in genus higher than zero the forgetful maps
	\begin{align}
		F\colon \ModuliStackSpinCurves{g,k_{NS}, k_R}&\to \ModuliStackCurves{g, k_{NS}+k_R} &
		F\colon \ModuliStackSpinMaps{g,k_{NS}, k_R}{X,\beta}&\to \ModuliStackMaps{g, k_{NS}+k_R}{X,\beta} &
	\end{align}
	do not induce isomorphisms on coarse moduli spaces but rather finite covers.
\end{rem}
\begin{rem}
	The axioms of super Gromov--Witten invariants depend essentially on the properties of the SUSY normal bundles \(\overline{N}_{k}\) and \(\overline{N}_{k,\beta}\) laid out in Conjectures~\ref{conj:Nk} and~\ref{conj:NkA}.
	If one had another collection \(E_k\) and \(E_{k,\beta}\) of vector bundles over \(\ModuliStackCurves{0,k}\) and \(\ModuliStackMaps{0,k}{X,\beta}\) respectively satisfying Conjectures~\ref{conj:Nk} and~\ref{conj:NkA}, a corresponding theory of \enquote{super Gromov--Witten invariants} obtained by using \(E_k\) and \(E_{k,\beta}\) instead of \(\overline{N}_{k}\) and \(\overline{N}_{k,\beta}\) would have very similar properties.
\end{rem}
 
\counterwithin{equation}{section}
\section{Super Small Quantum Cohomology}\label{Sec:SuperSmallQuantumCohomology}
In this section we extend the definition of the quantum product on \(H^\bullet(X)\) to a super quantum product on \(H^\bullet(X)\otimes \C[[\kappa^{-1}]]\) using the previously defined super Gromov--Witten invariants.

Let again \(T_j\in H^\bullet(X)\), \(j=0, \dotsc, \dim H^\bullet(X)-1\) be a homogeneous basis of \(H^\bullet(X)\) with \(T_0=1=[X]\), \(g_{ab}=\int_X T_a \cup T_b\), and \(g^{ab}\) its inverse matrix.
Let furthermore \(R\) be the ring freely generated over \(\C\) by symbols \(q^\beta\) for \(\beta \in H_2(X)\) with relations \(q^{\beta_1}\cdot q^{\beta_2}=q^{\beta_1 + \beta_2}\) and \(q^0=1\).
\begin{defn}
	For two homogeneous classes \(\alpha=\overline{\alpha}\otimes 1\otimes 1, \gamma=\overline{\gamma}\otimes 1\otimes 1\in H^\bullet(X)\otimes R\otimes \C[[\kappa^{-1}]]\) we define their super quantum product as
	\begin{equation}\label{defn:SuperQuantumProduct}
		\alpha \star \gamma
		= \sum_{\beta \in H_2(X)} \left<SGW_{0,3}^{X, \beta}\right> (\overline{\alpha}, \overline{\gamma}, T_a) T^a q^\beta \kappa^{r_{3,\beta}}
	\end{equation}
	and extend \(R\otimes \C[[\kappa^{-1}]]\)-linearly to all of \(H^\bullet(X)\otimes R\otimes \C[[\kappa^{-1}]]\).
	Here we use again the summation convention for \(a\) and \(T^a= T_b g^{ba}\).
	The powers of \(\kappa\) in \(\alpha\star \beta\) are all less or equal to zero by the Grading Axiom~\ref{axiom:Grading}.
	The sum appearing in the definition of \(\alpha\star \gamma\) is infinite and will be treated as formal here.
\end{defn}

\begin{prop}
	Let \(X\) be a convex variety such that \(H_2^+(X)\) has a finite \(\Integers_{\geq 0}\)-basis, for example, a homogeneous space.
	Then \(\alpha\star \gamma\) is a formal power series in \(\kappa^{-1}\) where each coefficient is given by a finite sum.
\end{prop}
\begin{proof}
	The summand \(\left<SGW_{0,3}^{X, \beta}\right> (\overline{\alpha}, \overline{\gamma}, T_a) T^a q^\beta \kappa^{r_{3,\beta}}\) in~\eqref{defn:SuperQuantumProduct} is of polynomial degree
	\begin{equation}
		d_p\left(\beta, \overline{\alpha}, \overline{\gamma}, a\right)
		= \frac12\left(\deg\overline{\alpha}+\deg\overline{\gamma}+\deg T_a\right) - \dim X - \left<c_1(\tangent{X}), \beta\right>
		\leq 0.
	\end{equation}
	After reordering with respect to orders of \(\kappa\) we get
	\begin{equation}\alpha \star \gamma
		= \sum_{j\leq0}\sum_{a,b=0}^{\dim H^\bullet(X)-1}\sum_{\substack{\beta \in H_2(X)\\d_p(\beta, \overline{\alpha}, \overline{\gamma}, a) = j}} \left<SGW_{0,3}^{X, \beta}\right> (\overline{\alpha}, \overline{\gamma}, T_a) g^{ab} T_b q^\beta \kappa^{r_{3,\beta}}
	\end{equation}
	For fixed \(j\) the remaining sum is finite because
	\begin{equation}
		\Set{\beta\in H_2(X) \given d_p(\beta, \overline{\alpha}, \overline{\gamma}, a)=j}
		\subset
		\Set{\beta\in H_2(X) \given \left<c_1(\tangent{X}), \beta\right> \leq 2\dim X - j}
	\end{equation}
	and the second set is finite because \(\beta\) is a finite \(\Integers_{\geq 0}\)-linear combination of the basis elements \(b_i\in H_2^+(X)\) and \(\left<c_1(X), b_i\right>\geq 2\).
	Compare also~\cite[Lemma~15]{FP-NSMQC}.
\end{proof}

Further questions of convergence and different possible interpretations of \(q^\beta\) are referred to later work, see also~\cite[Chapter~8.1]{CK-MSAG}.
Proving convergence of the infinite sum in~\eqref{defn:SuperQuantumProduct} might be harder than the results presented there because \(\left<SGW_{0,k}^{X, \beta}\right>(\overline{\alpha}, \overline{\gamma}, T_a)\) is non-vanishing for more combinations of \(\beta\) and \(T_a\) than \(\left<GW_{0,k}^{X, \beta}\right>(\overline{\alpha}, \overline{\gamma}, T_a)\).

\begin{prop}
	The super quantum product
	\begin{equation}
		\star \colon H^\bullet(X)\otimes R \otimes \C[[\kappa^{-1}]] \times H^\bullet(X)\otimes R\otimes \C[[\kappa^{-1}]]\to H^\bullet(X)\otimes R\otimes \C[[\kappa^{-1}]]
	\end{equation}
	is supercommutative with respect to the cohomological degree and associative.
	Hence it equips \(H^\bullet(X)\otimes R\otimes \C[[\kappa^{-1}]]\) with the structure of an \(R\otimes \C[[\kappa^{-1}]]\)-algebra.
\end{prop}
\begin{proof}
	The proof proof of this proposition is very similar to the proof in the non-super case, see, for example~\cite[Theorem~8.1.4]{CK-MSAG} and follows essentially from the Equivariance Axiom~\ref{axiom:SkEquivariance} and the Splitting Axiom~\ref{axiom:SplittingAxiom}.

	The super quantum product is indeed supercommutative with respect to the cohomological degree.
	That is,
	\begin{equation}
		\alpha \star \gamma = {\left(-1\right)}^{d_c(\alpha) \cdot d_c(\gamma)} \gamma \star \alpha,
	\end{equation}
	which follows from the Equivariance Axiom~\ref{axiom:SkEquivariance} because
	\begin{equation}
		\left<SGW_{0,3}^{X, \beta}\right>(\overline{\alpha}, \overline{\gamma}, T_a)
		= {\left(-1\right)}^{d_c(\alpha) \cdot d_c(\gamma)} \left<SGW_{0,3}^{X, \beta}\right>(\overline{\gamma}, \overline{\alpha},  T_a).
	\end{equation}

	To show associativity of the super quantum product we need to show \(\alpha\star \left(\gamma\star \delta\right) = \left(\alpha\star \gamma\right)\star \delta\) for all homogeneous classes \(\alpha=\overline{\alpha}\otimes 1\otimes 1\), \(\gamma=\overline{\gamma}\otimes 1\otimes 1\), \(\delta=\overline{\delta}\otimes 1\otimes 1\in H^\bullet(X)\otimes R\otimes \C[[\kappa^{-1}]]\).
	Expanding yields
	\begin{multline}
		\alpha \star \left(\gamma\star \delta\right) =\\
		\sum_{\beta \in H_2(X)} \sum_{\beta = \beta_1 + \beta_2} g^{ab}\left<SGW_{0,3}^{X, \beta_1}\right>(\overline{\gamma}, \overline{\delta}, T_a)\left<SGW_{0,3}^{X, \beta_2}\right>(\overline{\alpha}, T_b, T_c) q^\beta T^c \kappa^{r_{3,\beta_1+\beta_2}},
	\end{multline}
	and
	\begin{multline}
		\left(\alpha \star \gamma\right) \star \delta =\\
		\sum_{\beta \in H_2(X)} \sum_{\beta = \beta_1 + \beta_2} g^{ab}\left<SGW_{0,3}^{X, \beta_1}\right>(\overline{\alpha}, \overline{\gamma}, T_a)\left<SGW_{0,3}^{X, \beta_2}\right>(T_b, \overline{\delta}, T_c) q^\beta T^c \kappa^{r_{3,\beta_1+\beta_2}}.
	\end{multline}
	To show equality of the two expressions, we use the Splitting Axiom~\ref{axiom:SplittingAxiom}.
	\begin{equation}
		\begin{split}
			\MoveEqLeft
			g^{ab}\left<SGW_{0,3}^{X, \beta_1}\right>(\overline{\gamma}, \overline{\delta}, T_a)\left<SGW_{0,3}^{X, \beta_2}\right>(\overline{\alpha}, T_b, T_c) \\
			&= {\left(-1\right)}^{d_c(T_b)d_c(\overline{\alpha})} g^{ab}\left<SGW_{0,3}^{X, \beta_1}\right>(\overline{\gamma}, \overline{\delta}, T_a)\left<SGW_{0,3}^{X, \beta_2}\right>(T_b, \overline{\alpha}, T_c) \\
			&= {\left(-1\right)}^{d_c(T_b)d_c(\overline{\alpha})} \int_D \frac{\ev_1^*\overline{\gamma}\cup \ev_2^*\overline{\delta}\cup \ev_3^*\overline{\alpha} \cup \ev_4^*T_c}{e^K(\overline{N}_{4,\beta})}\\
			&= \int_D \frac{\ev_1^*\overline{\alpha} \cup \ev_2^*\overline{\gamma}\cup \ev_3^*\overline{\delta}\cup \ev_4^*T_c}{e^K(\overline{N}_{4,\beta})}\\
			&= g^{ab}\left<SGW_{0,3}^{X, \beta_1}\right>(\overline{\alpha}, \overline{\gamma}, T_a)\left<SGW_{0,3}^{X, \beta_2}\right>(T_b, \overline{\delta}, T_c)
		\end{split}
	\end{equation}
	Here we have used that the divisors associated to splitting moduli spaces with four marked points into two moduli spaces with two marked points are identical.
	To verify the signs note that \(d_c(\overline{\gamma})+d_c(\overline{\delta})+d_c(T_a)\) is even by the grading axiom and \(d_c(T_a) + d_c(T_b)\) is even because \(T_a\cup T_b\) is a class of degree \(2\dim X\).
\end{proof}
\begin{ex}[Super Small Quantum Cohomology of \(\ProjectiveSpace{n}\)]
	When the target \(X\) is projective space of dimension \(n\), any effective homology class \(\beta\in H_2(\ProjectiveSpace{n})\) is a non-negative multiple of a line, that is \(\beta=d[l]\) with \(d\geq 0\).
	In that case,
	\begin{align}
		d_{3,\beta} &= n + d(n+1) &
		r_{3,\beta} &= d(n+1) + 1
	\end{align}

	The cohomology ring of \(\ProjectiveSpace{n}\) is given by
	\begin{equation}
		H^*(\ProjectiveSpace{n}) = \faktor{\Z[\Lambda]}{\Lambda^{n+1}}
	\end{equation}
	for a generator \(\Lambda\in H^2(\ProjectiveSpace{n})\).
	The generator \(\Lambda\) is the Poincaré dual of the hyperplane class.
	We use \(T_a=\Lambda^a\) with \(a=0,\dotsc, n\) as basis for the cohomology ring and with respect to this basis
	\begin{equation}
		g_{ab}
		= g(T_a, T_b)
		= \int_{\ProjectiveSpace{n}} \Lambda^a \cup \Lambda^b
		= \delta_{a,n-b}.
	\end{equation}
	Hence, \(T^a= \Lambda^{n-a}\).

	To obtain the super small quantum cohomology ring of \(\ProjectiveSpace{n}\) it would be sufficient, by linearity, to compute all three-point super Gromov--Witten invariants
	\begin{equation}
		\left<SGW_{0,3}^{\ProjectiveSpace{n},d}\right>(\Lambda^a, \Lambda^b, \Lambda^c).
	\end{equation}
	As the Grading Axiom~\ref{axiom:Grading} is more permissive than the classical grading axiom there are for given \(a\) and \(b\) possibly infinitely many combinations of power \(c\) and degree \(d\) for which the three-point super Gromov--Witten invariants do not vanish.
	We will see in Section~\ref{Sec:SGWOfPn} that we are, in principle, able to compute three-point super Gromov--Witten invariants for \(\ProjectiveSpace{n}\) but do not have a closed formula for all of them even when \(n=1\).
	Hence, we will only give some further restrictions derived from the Grading Axiom~\ref{axiom:Grading} and the Point-Mapping Axiom~\ref{axiom:MappingToPoint}.

	For \(d=0\) the Point-Mapping-Axiom~\ref{axiom:MappingToPoint} implies that
	\begin{equation}
		\left<SGW_{0,3}^{\ProjectiveSpace{n},0}\right>(\Lambda^a, \Lambda^b, \Lambda^c) =
		\begin{cases}
			\kappa^{-1} & \text{if }a+b+c=n\\
			0 & \text{else}
		\end{cases}
	\end{equation}
	If \(d=1\) the Grading Axiom implies that three-point super Gromov--Witten invariants vanish for \(a+b+c> 2n+1\).
	For \(d>1\), no super three-point Super Gromov--Witten invariants vanish for degree reasons.

	It follows for \(a+b\leq n\) that
	\begin{equation}
		\Lambda^a \star \Lambda^b = \Lambda^{a+b} + \sum_{d=1}^{\infty}\sum_{c=0}^n \left<SGW_{0,3}^{\ProjectiveSpace{n},d}\right>\left(\Lambda^a, \Lambda^b, \Lambda^c\right) \Lambda^{n-c}q^d \kappa^{d(n+1)+1}
	\end{equation}
	The three point super Gromov--Witten invariant has polynomial degree
	\begin{equation}
		a+b+c - n - 2d\left(n+1\right) - 1.
	\end{equation}
	Consequently, the first term is the one of highest degree in \(\kappa\).

	If \(a+b>n\),
	\begin{equation}
		\begin{split}
			\Lambda^a \star \Lambda^b
			={} & \sum_{c=0}^{2n+1-a-b}\left<SGW_{0,3}^{\ProjectiveSpace{n},1}\right>(\Lambda^a,\Lambda^b, \Lambda^c) \Lambda^{n-c} q \kappa^{n+2} \\
				&+ \sum_{d=2}^{\infty}\sum_{c=0}^n \left<SGW_{0,3}^{\ProjectiveSpace{n},d}\right>\left(\Lambda^a, \Lambda^b, \Lambda^c\right) \Lambda^{n-c}q^d\kappa^{d(n+1)+1}
		\end{split}
	\end{equation}
	Here the term of highest polynomial degree is
	\begin{equation}
		\left<SGW_{0,3}^{\ProjectiveSpace{n},1}\right>(\Lambda^a,\Lambda^b, \Lambda^{2n+1-a-b}) \Lambda^{a+b-n-1} q \kappa^{n+2}
		= \Lambda^{a+b - n - 1} q.
	\end{equation}
\end{ex}
\begin{prop}
	Let \(\overline{p}\colon \C[[\kappa^{-1}]]\to \C\) the ring homomorphism which sends \(\kappa^{-1}\) to zero.
	Then
	\begin{equation}
		p=\id\otimes \overline{p}\colon H^\bullet(X)\otimes R\otimes \C[[\kappa^{-1}]] \to H^\bullet(X)\otimes R
	\end{equation}
	maps the super quantum product \(\star\) to the classical quantum product \(\diamond\) and extends to a homomorphism of \(R\)-algebras.
\end{prop}
\begin{proof}
	By the Grading Axiom~\ref{axiom:Grading}, the only terms of polynomial degree zero in
	\begin{equation}
		\alpha \star \gamma
		= \sum_{\beta \in H_2(X)} \left<SGW_{0,3}^{X, \beta}\right> (\alpha, \gamma, T_a) T^a q^\beta \kappa^{r_{3,\beta}}
	\end{equation}
	are the terms such that \(\deg \alpha + \deg \gamma + \deg T_a = 2\dim \ModuliStackMaps{0,3}{X, \beta}\).
	In this case by the Extension Axiom~\ref{axiom:Extension}, super Gromov--Witten invariants agree with Gromov--Witten invariants up to a polynomial prefactor
		\begin{align}
			p(\alpha \star \gamma)
			&= \sum_{\beta \in H_2(X)} \left<GW_{0,3}^{X, \beta}\right> (p(\alpha), p(\gamma), p(T_a)) T^a q^\beta \\
			&= p(\alpha)\diamond p(\gamma).
			\qedhere
		\end{align}
\end{proof}
\begin{rem}
	There is no known unit for the super quantum product \(\star\).
	Recall that \(T_0=1\) is the unit for the non-super quantum product~\(\diamond\).
	But for the super quantum product we have
	\begin{equation}
		\begin{split}
			\alpha\star T_0
			={}& \sum_{\beta \in H_2(X)}\left<SGW_{0,3}^{X, \beta}\right> (\overline{\alpha}, T_0, T_a) T^a q^\beta \kappa^{r_{3,\beta}}\\
			={}& \left(\int_X \alpha \cup T_a\right) T^a \\
				&+ \sum_{\beta\neq 0}\frac1{\kappa} \left(\int_{\ModuliStackMaps{0,3}{X, \beta}} f^*\frac{\ev_1^*\alpha\cup \ev_2^*T_a}{\overline{N}_{2,\beta}} \cup \sum_{j>0}{\left(-\frac{c_1(\UniversalSpinorBundleM_3)}{\kappa}\right)}^j\right) q^\beta T^a \kappa^{r_{3,\beta}}
		\end{split}
	\end{equation}
	by the Point-Mapping Axiom~\ref{axiom:MappingToPoint} and the Fundamental class Axiom~\ref{axiom:FundamentalClass}.
	The first summand is equal to \(\alpha\) by the definition of \(T^a\).
	That is, the super quantum product of \(\alpha\) with \(T_0\) yields \(\alpha\) up to lower order terms in \(\kappa\).
\end{rem}
\begin{lemma}
	Let \(g\) be the \(R\otimes \C[[\kappa^{-1}]]\)-bilinear form
	\begin{equation}
		g\colon H^\bullet(X)\otimes R\otimes \C[[\kappa^{-1}]] \times H^\bullet(X)\otimes R\otimes \C[[\kappa^{-1}]]\to R\otimes \C[[\kappa^{-1}]]
	\end{equation}
	such that \(g(\overline{\alpha}\otimes 1\otimes 1, \overline{\gamma}\otimes 1\otimes 1) = \int_X \overline{\alpha}\cup\overline{\gamma}\).
	Then
	\begin{align}
		g(\alpha\star \gamma, \delta) &= g(\alpha, \gamma\star \delta) = \sum_{\beta \in H_2(X)} \left<SGW_{0,3}^{X, \beta}\right>(\alpha, \gamma, \delta) q^\beta \kappa^{3,\beta} \\
		\int_X \alpha\star \gamma &= g(T_0\star \alpha, \gamma).
	\end{align}
\end{lemma}

 \counterwithin{equation}{subsection}

\section{Super Gromov--Witten invariants of \texorpdfstring{\(\ProjectiveSpace{n}\)}{PCn}}\label{Sec:SGWOfPn}

In this section we explain a method that allows, in principle, to calculate arbitrary super Gromov--Witten invariants of projective spaces \(\ProjectiveSpace{n}\).
The idea is to use an additional geometric torus localization with respect to the torus action of \(T={\left(\C^*\right)}^{n+1}\) acting on \(\ProjectiveSpace{n}\) by rescaling the projective coordinates.
Fixed points and normal bundles of the induced torus action on the stable maps moduli spaces \(\ModuliStackMaps{0,k}{\ProjectiveSpace{n}, \beta}\) have been studied in the important paper~\cite{K-ERCVTA}.
As the stable maps which are invariant under the torus action are of fixed tree types we can use the gluing properties of SUSY normal bundles to reduce to simple building blocks which can be evaluated.
Expressions for super Gromov--Witten invariants are then obtained by working with sums of quotients of polynomials in the torus characters.
We give explicit calculations of super Gromov--Witten invariants of \(\ProjectiveSpace{1}\) in degree one and with one, two or three marked points.
For \(n\leq5\) we list super Gromov--Witten invariants obtained with the help of computer algebra software.

In Section~\ref{SSec:GeometricLocalization} we discuss how torus actions on the target \(X\) induce torus actions on the moduli spaces of super stable maps and the normal bundles.
Section~\ref{SSec:TorusActionOnPn} sets notation for the torus action on \(\ProjectiveSpace{n}\) and recalls an enumeration of fixed loci in the moduli spaces of stable maps via graphs.
In Section~\ref{SSec:EquivariantEulerClassOfSUSYNormalBundle} we calculate equivariant Euler class of the SUSY normal bundle restricted to fixed loci with particularly simple graphs.
In Sections~\ref{SSec:OnePointSGWofPnDeg1}--\ref{SSec:ThreePointSGWofPnDeg1} we calculate super Gromov--Witten invariants of \(\ProjectiveSpace{n}\) with one, two or three marked points for degree one.

\subsection{Geometric localization}\label{SSec:GeometricLocalization}
Assume that the target variety \(X\) is equipped with an algebraic action of of a torus \(T={\left(\C^*\right)}^m\):
\begin{equation}
	\begin{split}
		a\colon T\times X &\to X\\
		(t, x) &\mapsto a_t(x)
	\end{split}
\end{equation}
For fixed \(t\in T\) we obtain an automorphism \(a_t\colon X\to X\).

For any stable map \(\phi\colon C\to X\) the composition~\(a_t\circ \phi\) is again a stable map.
Hence we obtain a \(T\)-action on the moduli spaces \(\ModuliStackMaps{0,k}{X,\beta}\) which we denote again by \(a\).
Similarly, for any stable spin map \(\phi\colon C\to X\) from a prestable spin curve \(C\) the composition \(a_t\circ \phi\) is again a stable spin map.
The induced action on the moduli stack of stable spin maps commutes with the forgetful maps
\begin{diag}\matrix[mat](m) {
		T\times \ModuliStackSpinMaps{0,k}{X,\beta} & \ModuliStackSpinMaps{0,k}{X,\beta}\\
		T\times \ModuliStackMaps{0,k}{X,\beta} & \ModuliStackMaps{0,k}{X,\beta}\\
	} ;
	\path[pf] {
		(m-1-1) edge node[auto]{\(a\)} (m-1-2)
			edge node[auto,swap]{\(\id\times F\)} (m-2-1)
		(m-1-2) edge node[auto]{\(F\)} (m-2-2)
		(m-2-1) edge node[auto]{\(a\)} (m-2-2)
	};
\end{diag}
because the torus \(T\) does not act on the spinor bundles.

For any twisted spinor \(\psi\in H^0\left(\dual{S}_C\otimes\phi^*\tangent{X}\right)\), the differential of \(a_t\) induces a map
\begin{equation}
	\id_{\dual{S}_C\otimes \phi^*\differential{a_t}}\colon H^0\left(\dual{S}_C\otimes\phi^*\tangent{X}\right)\to H^0\left(\dual{S}_C\otimes {\left(a_t\circ \phi\right)}^*\tangent{X}\right)
\end{equation}
between the holomorphic sections of twisted spinors along \(\phi\) and twisted spinors along \(a_t\circ \phi\).
Applying this map to the summand in the middle term of~\eqref{SES:DefnNCphiB}, we see that \(\overline{N}_{k, \beta}\) is an equivariant bundle over \(a\colon T\times\ModuliStackSpinMaps{0,k}{X,\beta}\to \ModuliStackSpinMaps{0,k}{X,\beta}\).

The fixed point locus of the action \(a\) on \(\ModuliStackMaps{0,k}{X,\beta}\) is a union of smooth connected components \(M_\Gamma\) indexed by \(\Gamma\in I\) some finite set.
We denote by \(i_\Gamma\colon M_\Gamma\to \ModuliStackMaps{0,k}{X,\beta}\) the inclusion maps and by \(N_\Gamma\) its normal bundle.
As the \(T\)-action on stable spin maps does not operate on spinor bundles, the fixed points of the action \(a\colon \ModuliStackSpinMaps{0,k}{X,\beta}\to \ModuliStackSpinMaps{0,k}{X,\beta}\) are precisely the preimage \(M_\Gamma^{spin} = M_\Gamma \times_{\ModuliStackMaps{0,k}{X,\beta}}\ModuliStackSpinMaps{0,k}{X,\beta}\) of \(M_\Gamma\) under \(F\).
The inclusion \(M^{spin}_\Gamma\to \ModuliStackSpinMaps{0,k}{X,\beta}\) is likewise denoted \(i_\Gamma\) and the normal bundle is isomorphic to \(F^*N_\Gamma\).
This implies \(e^T(N_\Gamma) = e^T(F^*N_\Gamma)\in H_T^\bullet\left(\ModuliStackMaps{0,k}{X,\beta}\right)\).

Following the calculations explained in~\cite[Section~9.1.3]{CK-MSAG}, we can calculate super Gromov--Witten invariants using localization with respect to the \(T\)-action as follows:
\begin{equation}\label{eq:SGWviaLocalization}
	\begin{split}
		\MoveEqLeft
		\left<SGW_{0,k}^{X, \beta}\right>(\alpha_1, \dotsc, \alpha_k)
		= \int_{\ModuliStackMaps{0,k}{X,\beta}} \frac{\ev_1^*\alpha_1\cup\dotsm\cup \ev_k^*\alpha_k}{e^K(\overline{N}_{k, \beta})} \\
		&= \int_{\ModuliStackSpinMaps{0,k}{X,\beta}} \frac{\ev_1^*\alpha_1\cup\dotsm\cup \ev_k^*\alpha_k}{2e^K(\overline{N}_{k, \beta})}\\
		&= \sum_{\Gamma\in I} \left.\int_{M^{spin}_\Gamma}\left(\frac{i_\Gamma^*\ev_1^*\alpha_1\cup\dotsm \cup i_\Gamma^*\ev_k^*\alpha_k}{o_\Gamma e^{K\times T}(i_\Gamma^*\overline{N}_{k, \beta}) e^T(N_\Gamma)}\right)\right|_{d_p=-r_{k, \beta}-d_{k, \beta}+\frac12\deg\alpha} \\
		&= \sum_{\Gamma\in I} \left.\int_{M_\Gamma}\left(\frac{i_\Gamma^*\ev_1^*\alpha_1\cup\dotsm \cup i_\Gamma^*\ev_k^*\alpha_k}{e^{K\times T}(i_\Gamma^*\overline{N}_{k, \beta}) e^T(N_\Gamma)}\right)\right|_{d_p=-r_{k, \beta}-d_{k, \beta}+\frac12\deg\alpha}.
	\end{split}
\end{equation}
Here the restriction to terms of polynomial degree \(d_p=-r_{k, \beta}-d_{k, \beta}+\frac12\deg\alpha\) in \(\kappa\) selects the terms of top cohomological degree in the integral over \(\ModuliStackMaps{0,k}{X,\beta}\).
As in Definition~\ref{defn:SGWNumbers}, the order \(o_\Gamma\) of the isotropy group of the generic point of \(M^{spin}_\Gamma\) appearing in the denominator cancels against \([M^{spin}_\Gamma] = o_\Gamma[M_\Gamma]\).

\begin{rem}
	The \(T\)-action on \(\ModuliStackMaps{0,k}{X,\beta}\) and \(\overline{N}_{k, \beta}\) given above is motivated by the component field formalism for super \(\targetACI\)-holomorphic curves developed in~\cite{KSY-SJC}:
	Let \(\Phi\colon \ProjectiveSpace[\C]{1|1}\to X\) be a super \(\targetACI\)-holomorphic curve with component fields \((\varphi, \psi, F=0)\) with respect to \(i\colon \ProjectiveSpace[\C]{1}\to\ProjectiveSpace[\C]{1|1}\).
	Then, for every \(t\in T\), also \(\Phi_t = a_t\circ \Phi\) is a super \(\targetACI\)-holomorphic curve and its component fields are given by
	\begin{equation}\label{eq:TorusActionComponentFields}
		\begin{aligned}
			\varphi_t &= a_t \circ \Phi \circ i = a_t \circ \varphi \\
			\psi_t &= i^*\differential{\left(a_t \circ \Phi\right)}|_\cD
			= i^*a_t^*\differential{\Phi}\circ i^*\differential{\Phi}|_\cD
			= \varphi^*\differential{a_t}\circ \psi
		\end{aligned}
	\end{equation}
	Here \(\varphi_t\colon \ProjectiveSpace[\C]{1}\to X\) is a \(\targetACI\)-holomorphic curve and \(\psi_t \in\VSec{\dual{S}_C\otimes\varphi_t^*\tangent{X}}\) is a holomorphic twisted spinor.
	The map \(\Phi\mapsto a_t\circ \Phi\) induces an action of \(T\) on the moduli spaces \(\SmModuliSpaceSuperJMaps{X, \beta}\), \(\SmModuliSpaceSuperMaps{0,k}{X, \beta}\) and \(\SmModuliSpaceSuperMaps{\tau}{X, \beta}\).
	This action commutes with the action of \(K\) given by the rescaling of the odd directions.

	In the same way, there should be an algebraic \(T\)-action on the moduli stack \(\ModuliStackSuperMaps{0,k}{X,\beta}\) of algebraic super stable maps from Conjecture~\ref{conj:ModuliStackSuperStableMaps}.
	The inclusion \(\ModuliStackSpinMaps{0,k}{X,\beta}\to \ModuliStackSuperMaps{0,k}{X,\beta}\) is then \(T\)-equivariant and consequently, the normal bundle \(\overline{N}_{k,\beta}\) is an equivariant bundle.
\end{rem}

\subsection{Torus action on \texorpdfstring{\(\ProjectiveSpace{n}\)}{PnC}}\label{SSec:TorusActionOnPn}
In this section we describe an action of \(T={\left(\C^*\right)}^{n+1}\) on \(\ProjectiveSpace{n}\) and recall from~\cite{K-ERCVTA} that the fixed point loci of the induced action on stable maps moduli spaces are classified by certain labeled graphs.

Before turning to the torus action, let us fix the following notation:
The cohomology ring of \(\ProjectiveSpace{n}\) is a polynomial ring in one generator \(\Lambda\in H^2(\ProjectiveSpace{n}, \Z)\), such that
\begin{equation}
	H^*(\ProjectiveSpace{n}, \Z) = \faktor{\Z[\Lambda]}{\Lambda^{n+1}}.
\end{equation}
The generator \(\Lambda\) is the Poincaré dual of the hyperplane class.
Similarly, we denote the generator of \(H^2(\ProjectiveSpace{1}, \Z)\) by \(\lambda\).

For \(\beta=d[l]\) that is the \(d\)-fold multiple of the class of a line in \(\ProjectiveSpace{n}\), write \(\ModuliStackMaps{0,k}{\ProjectiveSpace{n},d}\) for the moduli spaces of stable maps of degree \(d\) in \(\ProjectiveSpace{n}\) with \(k\) marked points.
Similarly, we write \(\overline{N}_{k,d}\) for the normal bundle over \(\ModuliStackMaps{0,k}{\ProjectiveSpace{n},d}\).
Recall that \(\left<c_1\left(\tangent{\ProjectiveSpace{n}}\right), d[l]\right>=d(n+1)\) and hence the dimension of \(\ModuliStackMaps{0,k}{\ProjectiveSpace{n},d}\) and the rank of \(\overline{N}_{k,d}\) are given by
\begin{align}
	d_{k,d} &= n + d(n+1) + k - 3, &
	r_{k,d} &= d(n+1) + k - 2.
\end{align}

In the remainder of Section~\ref{Sec:SGWOfPn}, we use the action of the torus \(T={\left(\C^*\right)}^{n+1}\) on \(\ProjectiveSpace{n}\) that acts on homogeneous coordinates via
\begin{equation}\label{eq:TorusActionOnPN}
	\begin{split}
		a\colon {\left(\C^*\right)}^{n+1}\times\ProjectiveSpace{n} &\to \ProjectiveSpace{n}\\
		\left((t_0, \dotsc, t_n), [X^0: \dots :X^n]\right) &\mapsto \left[\frac{X^0}{t_0}: \dots :\frac{X^n}{t_n}\right].
	\end{split}
\end{equation}
We follow the notation introduced in~\cite[Chapter~9.2]{CK-MSAG}:
Fixed points of the \(T\)-action are called \(q_0, \dotsc, q_n\) where \(q_i\) is the point where all homogeneous coordinates except the \(i\)-th vanish.

The weights corresponding to the coordinates of \(T={\left(\C^*\right)}^{n+1}\) are denoted \(\tau_0,\dotsc, \tau_n\).
That is, \(H_T^\bullet(\pt)=\C[\tau_0, \dotsc, \tau_n]\).
In particular, the one-dimensional representation \((t_0,\dotsc, t_n)\cdot v = t_0^{d_0}\cdot \dotsm\cdot t_n^{d_n}\cdot v\) has weight \(-d_0\tau_0 - \dotsb - d_n\tau_n\).

The action of \(T\) induces an action on the moduli spaces \(\ModuliStackMaps{0,k}{\ProjectiveSpace{n},d}\) of stable maps in \(\ProjectiveSpace{n}\) by composition.
Stable maps fixed under the \(T\)-action map every marked point to one of the fixed points and components of the curve are either mapped to a fixed point or a coordinate line connecting the fixed points.
It was worked out in~\cite{K-ERCVTA} that stable curves fixed under the \(T\)-action can be characterized by a labeled graph \(\Gamma\) as follows:
\begin{itemize}
	\item
		vertices \(v\) of \(\Gamma\) correspond to the fixed points \(q_i\) that are in the image of the stable map.
		The vertex \(v\) is labeled by the number \(i_v\) of the fixed point it corresponds to and the set of indices of the marked point that map to \(q_i\).
	\item
		two vertices \(v\) and \(v'\) are connected by an edge \(e\) if the coordinate line between \(q_{i_v}\) and \(q_{i_v'}\) is in the image of the stable map.
		The edge is labeled by the degree of the map mapping to the coordinate line.
\end{itemize}

Let \(\phi\colon C\to \ProjectiveSpace{n}\) be a stable map of genus zero fixed under the torus action and with graph \(\Gamma\).
Then the curve \(C\) decomposes into irreducible components as follows:
\begin{itemize}
	\item
		one irreducible component per edge which is mapped to a coordinate line in \(\ProjectiveSpace{n}\).
		Up to a choice of identification with \(\ProjectiveSpace{1}\), the points \(0\) and \(\infty\) are mapped to torus fixed points of \(\ProjectiveSpace{n}\).
	\item
		if for a given vertex the number of outgoing edges plus marked points at that vertex are larger than two, there are components which are contracted under \(\phi\) to a fixed point of the \(T\)-action.
\end{itemize}
The graph describes the non-contracted components uniquely.
Consequently, the fixed point loci \(M_\Gamma\subset\ModuliStackMaps{0,k}{\ProjectiveSpace{n},d}\) of all stable maps fixed under the torus action and of type~\(\Gamma\) is roughly a product of moduli spaces of stable curves.
For the precise description of the stack structure of \(M_\Gamma\) we refer to~\cites{K-ERCVTA}{CK-MSAG}{GP-LVC}.

The equivariant Euler classes of the normal bundles \(N_\Gamma\) to the inclusions \(i_\Gamma\colon M_\Gamma\to \ModuliStackMaps{0,k}{\ProjectiveSpace{n},d}\) have been worked out in~\cite{K-ERCVTA} as well, see also~\cite[Theorem~9.2.1]{CK-MSAG}.

\begin{ex}\label{ex:TinvStableMapskMarkedPoints}
	\(T\)-invariant stable maps \(\ProjectiveSpace{1}\to \ProjectiveSpace{n}\) of degree one with \(k\) marked points are characterized by graphs of the form
	\begin{Tgraph}
		\node[vertex, label=above:\(A\)] (1) {$a$};
		\node[vertex, label=above:\(B\)] (2) [right of=1] {$b$};
		\draw (1) -- node [above]{\(d=1\)} (2);
	\end{Tgraph}
	Here \(a\) and \(b\) are elements of \(\Set{0,\dotsc, n}\) with \(a<b\) and \(A\), \(B\subset\Set{1,\dotsc,k}\) are a partition, that is  \(B=\Set{1,\dotsc, k}\setminus A\).
	We denote the graph above by \(\Gamma^{k,1}_{a,b, A}\).
	More generally, we denote the above graph with general degree \(d\) by \(\Gamma^{k,d}_{a,b,A}\).

	The graph \(\Gamma^{k,d}_{a,b,A}\) represents \(T\)-invariant stable maps \(\phi\) with \(k\)-marked points where the marked points \(p_i\) for \(i\in A\) are mapped to the fixed point \(q_a\in\ProjectiveSpace{n}\) and the marked points~\(p_j\) for \(j\in B\) are mapped to \(q_b\in\ProjectiveSpace{n}\).
	Hence, if \(A\) or \(B\) contain more than one point, the domain nodal curve of \(\phi\) needs to consist of several components.
	All of those components except one collapse either on \(q_a\) or \(q_b\).
	The remaining component is a degree~\(d\) map onto the line going through \(q_a\) and \(q_b\).
\end{ex}

\subsection{Equivariant Euler class of SUSY normal bundle for the moduli space \texorpdfstring{\(M_\Gamma\)}{MGamma}}\label{SSec:EquivariantEulerClassOfSUSYNormalBundle}
In this section we show how to calculate the equivariant Euler classes
\begin{equation}
	e^{K\times T}\left(i_\Gamma^*\overline{N}_{k,d}\right)
\end{equation}
of the SUSY normal bundles restricted to fixed point loci \(i_\Gamma\colon M^{spin}_\Gamma\to \ModuliStackSpinMaps{0,k}{\ProjectiveSpace{n},d}\).

If \(\Gamma=\Gamma^k_a\) is the graph that describes the moduli space of \(k\) marked points mapped to the fixed point \(q_a\in \ProjectiveSpace{n}\) by a degree zero map, the fixed point spaces are \(M_\Gamma=\ModuliStackCurves{0,k}\) and \(M^{spin}_\Gamma = \ModuliStackSpinCurves{0,k}\).
Hence the equivariant Euler classes of \(\overline{N}_k\) coincide with the ones obtained in Section~\ref{SSec:SGWIofaPoint}.

If instead \(\Gamma\) is one of the following graphs
\begin{align}
	\Gamma^{1,d}_{a,b,\emptyset},&&
	\Gamma^{1,d}_{a,b,\Set{1}},&&
	\Gamma^{2,d}_{a,b,\Set{1}},&&
	\Gamma^{2,d}_{a,b,\Set{2}},
\end{align}
describing a map of degree \(d\) from \(\ProjectiveSpace{1}\) to \(\ProjectiveSpace{n}\) going through the \(T\)-fixed points \(q_a\) and \(q_b\), the moduli space \(M_\Gamma\) is a point and we have described the SUSY normal bundles \(i_\Gamma^*\overline{N}_{k,d}\) in Example~\ref{ex:NCphiPn}.
In order to describe the equivariant Euler classes of the SUSY normal bundles \(i_\Gamma^*\overline{N}_{k,d}\) it remains to calculate the \(T\)-action.

For an arbitrary graph \(\Gamma\) describing a fixed point locus of the \(T\)-action the moduli space~\(M_\Gamma\) is a product of the moduli spaces corresponding to the above elementary graphs.
By the gluing properties of the SUSY normal bundles, see Proposition~\ref{prop:GluingNCphi}, the normal bundles add up and hence the equivariant Euler classes multiply.

\begin{prop}\label{prop:EqEulerClassN2d}
	Let \(\Gamma=\Gamma^{2,d}_{a,b,\Set{1}}\) or \(\Gamma=\Gamma^{2,d}_{a,b,\Set{2}}\).
	Then \(M_\Gamma\) is a point and the equivariant Euler class is given by
	\begin{equation}
		\begin{split}
			e^{K\times T}(i_\Gamma^*\overline{N}_{2,d})
			={}& \prod_{0\leq q\leq 2d-1}\left(\kappa + \frac{2d-2q-1}{2d}\tau_a - \frac{2d-2q-1}{2d}\tau_b\right)\\
				&\cdot\prod_{\substack{0\leq m \leq n\\m\neq a,b}}\prod_{0\leq q\leq d-1}\left(\kappa + \frac{2q-1}{2d}\tau_a - \frac{2d-2q-1}{2d} \tau_b + \tau_m\right).
		\end{split}
	\end{equation}
\end{prop}
\begin{proof}
	We treat the case \(\Gamma=\Gamma^{2,d}_{a,b,\Set{1}}\) first.
	That is, \(C=\ProjectiveSpace{1}\) with the two marked points \(p_1=0=[0:1]\), \(p_2=\infty=[1:0]\) and the map is
	\begin{equation}
		\begin{split}
			\phi\colon \ProjectiveSpace{1} &\to \ProjectiveSpace{n} \\
			\left[Z^1:Z^2\right] &\mapsto \left[0: \dots :0: {\left(Z^1\right)}^d :0: \dots :0: {\left(Z^2\right)}^d :0: \dots :0\right]
		\end{split}
	\end{equation}

	By Definition \(i_\Gamma^*\overline{N}_{2,d}=N_{(C,\phi)}\) is given by the short exact sequence
	\begin{diag}\matrix[mat](m) {
			0 & H^0\left(S_C\right) & p_1^*S_C\oplus p_2^*S_C\oplus H^0\left(\dual{S}_C\otimes \phi^*\tangent{\ProjectiveSpace{n}}\right) & N_{(C,\phi)} & 0 \\
		} ;
		\path[pf] {
			(m-1-1) edge (m-1-2)
			(m-1-2) edge (m-1-3)
			(m-1-3) edge (m-1-4)
			(m-1-4) edge (m-1-5)
		};
	\end{diag}
	The equivariant Euler class of \(N_{(C,\phi)}\) is consequently given by
	\begin{equation}\label{eq:EqEulerClassN2dasProduct}
		e^{K\times T}(i_\Gamma^*\overline{N}_{(C,\phi)})
		= \frac{e^{K\times T}(p_1^*S_C)\cup e^{K\times T}(p_2^*S_C) \cup e^{K\times T}\left(H^0\left(\dual{S}_C\otimes \phi^*\tangent{\ProjectiveSpace{n}}\right)\right)}{e^{K\times T}(H^0(S_C))}
	\end{equation}

	The torus \(K\) acts on all summands by rescaling.
	The torus \(T\) acts on the summand \(H^0(\dual{S}_C\otimes \phi^*\tangent{\ProjectiveSpace{n}})\) by
	\begin{equation}
		\id_{\dual{S}_C\otimes \phi^*\differential{a_t}}\colon H^0\left(\dual{S}_C\otimes\phi^*\tangent{\ProjectiveSpace{n}}\right)\to H^0\left(\dual{S}_C\otimes {\left(a_t\circ \phi\right)}^*\tangent{\ProjectiveSpace{n}}\right)
	\end{equation}
	while \(p_1^*S_C\) and \(p_2^*S_C\) as well as \(H^0(S_C)\) are invariant.
	This induces the \(K\times T\)-action on \(N_{(C,\phi)}\).

	However, the \(T\)-action does not map \(H^0(\dual{S}_C\otimes\phi^*\tangent{X})\) to itself because \(a_t\circ \phi\) differs from \(\phi\) by a Möbius transformation \(\xi_t\) such that \(a_t\circ \phi=\phi\circ \xi_t\).
	We will see below that the combined action of \(a_t\) and \(\xi_t\) maps \(H^0(\dual{S}_C\otimes\phi^*\tangent{X})\) to itself and even preserves the direct sum decomposition
	\begin{equation}
		H^0\left(\dual{S}_C\otimes\phi^*\tangent{\ProjectiveSpace{n}}\right)
		= H^0\left(\cO(2d-1)\right) \oplus {\left(H^0\left(\cO(d-1)\right)\right)}^{n-1}
	\end{equation}
	from Example~\ref{ex:NCphiPn}.
	Even though the Möbius transformation \(\xi_t\) is \(T\)-dependent, its induced action on \(N_{(C,\phi)}\) is trivial as discussed in~\cite{KSY-TAMSSSMGZ}.
	Hence, it will be easier to use the combined action of \(a_t\) and \(\xi_t\) to read off the \(T\)-action on \(H^0(\dual{S}_C\otimes\phi^*\tangent{\ProjectiveSpace{n}})\) at the expense of considering a non-trivial \(T\)-action on \(H^0(S_C)\) and \(p_i^*S_C\).

	To be more explicit, we will use the notation for coordinates and local frames from Example~\ref{ex:NCphiPn}.
	The torus action \(a_t\colon \ProjectiveSpace{n}\to\ProjectiveSpace{n}\) defined in Equation~\eqref{eq:TorusActionOnPN} is given in the coordinates~\(x_l^q\) by \(a_t^\# x_l^q = \frac{t_l}{t_q}x_l^q\).
	Hence, \(a_t\circ \phi\) is given in local coordinates by
	\begin{align}
		\phi^\# a_t^\#x_b^a &= \frac{t_b}{t_a}{z_1}^d, &
		\phi^\# a_t^\#x_a^m &= 0,\text{ for }m\neq a,b, \\
		\phi^\# a_t^\#x_a^b &= \frac{t_a}{t_b}{\left(-z_2\right)}^d, &
		\phi^\# a_t^\#x_a^m &= 0,\text{ for }m\neq a, b.
	\end{align}
	Hence,
	\begin{equation}
		g_t =
		\begin{pmatrix}
			\sqrt[2d]{\frac{t_b}{t_a}} & 0 \\
			0 & \sqrt[2d]{\frac{t_a}{t_b}} \\
		\end{pmatrix}
	\end{equation}
	is an element of \(\SGL(2)\) such that the induced Möbius transformation \(\xi_t\) satisfies \(a_t\circ \phi = \phi\circ \xi_t\).
	In order to understand the combined action of \(a_t\) and \(\xi_t\) on \(N_{(C, \phi)}\) we need to make the combined action on \(p_i^*S_C\), \(H^0(S_C)\) and \(H^0(\dual{S}_C\otimes \phi^*\tangent{\ProjectiveSpace{n}})\) explicit.

	The Möbius transformation~\(\xi_t\colon \ProjectiveSpace{1}\to \ProjectiveSpace{1}\) induced from \(g_t\) is given in local coordinates by \(\xi_t^\# z_1=A^2z_1\) and \(\xi_t^\# z_2 = A^{-2}z_2\) where we write \(A=\sqrt[2d]{\frac{t_b}{t_a}}\) for simplicity.
	The Möbius transformation \(\xi_t\) preserves the two points \(0\) and \(\infty\) so that the triples \((0,\infty, a_t\circ\phi)\) and \((0,\infty, \phi)\) represent the same point in \(\ModuliStackMaps{0,2}{\ProjectiveSpace{n},d}\).

	The element \(g_t\) induces the map \(\sigma_t\colon S_C\to \xi_t^*S_C\) given by
	\begin{align}
		\sigma(s_1) &= A\xi_t^*s_1, &
		\sigma(s_2) &= A^{-1}\xi_t^*s_2.
	\end{align}
	The map \(\sigma\) induces the following maps \(p_i^*\sigma\colon p_i^*S_C\to p_i^*\xi_t^*S = p_i^*S_C\)
	\begin{align}
		p_1^*S &\to p_1^*S_C &
		p_2^*S &\to p_2^*S_C \\
		p_1^*s_1 &\mapsto A p_1^*s_1 &
		p_2^*s_2 &\mapsto A^{-1} p_1^*s_2
	\end{align}
	Consequently, their equivariant Euler classes are given by
	\begin{align}\label{eq:EqEulerClasspiS}
		e^{K\times T}(p_1^*S)
		&= \kappa + \frac1{2d}\tau_a - \frac1{2d}\tau_b, &
		e^{K\times T}(p_2^*S)
		&= \kappa - \frac1{2d}\tau_a + \frac1{2d}\tau_b.
	\end{align}

	Composing sections \(s\colon \ProjectiveSpace{1}\to S_C\) with \(\sigma\) gives a map \(H^0(S_C)\to H^0(\xi^*S_C)\).
	As \(\xi_t\) is an isomorphism, all sections of \(H^0(\xi_t^*S_C)\) can be obtained as pullbacks from sections \(\tilde{s}\in H^0(S_C)\).
	The resulting map wich sends a section \(s\in H^0(S_C)\) to \(\sigma\circ s\in H^0(S_C)\) and then to \(\tilde{s}\) is given by
	\begin{align}
		H^0(S_C) & \to H^0(S_C) \\
		s = \left(u z_1 + v\right)s_1
		&\mapsto
		\tilde{s} = \left(A^{-1} uz_1 + Av\right)s_1
	\end{align}
	Hence,
	\begin{equation}\label{eq:EqEulerClassH0S}
		e^{K\times T}(H^0(S_C)) = \left(\kappa - \frac1{2d}\tau_a + \frac1{2d}\tau_b\right)\left(\kappa + \frac1{2d}\tau_a - \frac1{2d}\tau_b\right).
	\end{equation}

	Let now \(\psi\in H^0(\dual{S}_C\otimes \phi^*\tangent{X})\) given locally by
	\begin{align}
		\psi(z_1) &= \sum_{q=0}^{2d-1}\psi^{ab}_q {z_1}^q \dual{s}_1\otimes\phi^*\partial_{x_b^a} + \sum_{m\neq a,b}\sum_{q=0}^{d-1}\psi^m_q{z_1}^q \dual{s}_1\otimes \phi^*\partial_{x_b^m}, \\
		\psi(z_2) &= -\sum_{q=0}^{2d-1}\psi^{ab}_q {\left(-z_2\right)}^{2d-1-q} \dual{s}_2\otimes\phi^*\partial_{x_a^b} + \sum_{m\neq a,b}\sum_{q=0}^{d-1}\psi^m_q{\left(-z_2\right)}^{d-1-q} \dual{s}_2\otimes \phi^*\partial_{x_a^m}.
	\end{align}
	Composing with \(\id_{\dual{S}_C}\otimes \phi^*\differential{a_t}\) yields \(\left(\id_{\dual{S}_C}\otimes \phi^*\differential{a_t}\right)\psi\in H^0(\dual{S}_C\otimes\phi^*a_t^*\tangent{\ProjectiveSpace{n}})\) given by
	\begin{align}
		\left(\id_{\dual{S}_C}\otimes \phi^*\differential{a_t}\right)\psi(z_1)
		={}& \sum_{q=0}^{2d-1}\psi^{ab}_q \frac{t_b}{t_a} {z_1}^q \dual{s}_1\otimes\phi^*a_t^*\partial_{x_b^a} + \sum_{m\neq a,b}\sum_{q=0}^{d-1}\psi^m_q \frac{t_b}{t_m} {z_1}^q \dual{s}_1\otimes \phi^*a_t^*\partial_{x_b^m}, \\
		\begin{split}
			\left(\id_{\dual{S}_C}\otimes \phi^*\differential{a_t}\right)\psi(z_2)
			={}& -\sum_{q=0}^{2d-1}\psi^{ab}_q \frac{t_a}{t_b} {\left(-z_2\right)}^{2d-1-q} \dual{s}_2\otimes\phi^*a_t^*\partial_{x_a^b} \\
				&+ \sum_{m\neq a,b}\sum_{q=0}^{d-1}\psi^m_q \frac{t_a}{t_m} {\left(-z_2\right)}^{d-1-q} \dual{s}_2\otimes \phi^*a_t^*\partial_{x_a^m}
		\end{split}
	\end{align}

	The action of \(g_t\) on \(\left(\id_{\dual{S}_C}\otimes \phi^*\differential{a_t}\right)\psi\in H^0(\dual{S}_C\otimes\phi^*a_t^*\tangent{\ProjectiveSpace{n}})\) is obtained in two steps.
	First, composing with \({\left(\dual{\sigma}\right)}^{-1}\otimes \id_{\phi^*a_t^*\tangent{\ProjectiveSpace{n}}}\) yields a section of \(H^0(\xi_t^*\dual{S}_C\otimes \phi^*a_t^*\tangent{\ProjectiveSpace{n}}) = H^0(\xi_t^*\left(\dual{S}_C\otimes \phi^*\tangent{\ProjectiveSpace{n}}\right))\).
	In a second step the unique \(\tilde{\psi}\in H^0(\dual{S}_C\otimes \phi^*\tangent{\ProjectiveSpace{n}})\) such that
	\begin{equation}
		\left({\left(\dual{\sigma}\right)}^{-1}\otimes \id_{\phi^*a_t^*\tangent{\ProjectiveSpace{n}}}\right)\left(\id_{\dual{S}_C}\otimes \phi^*\differential{a_t}\right)\psi
		= \xi_t^*\tilde{\psi}
	\end{equation}
	is identified.
	The resulting combined action of \(a_t\) and \(g_t\) on \(H^0(\dual{S}_C\otimes\phi^*\tangent{\ProjectiveSpace{n}})\) is then given by \(\psi\mapsto \tilde{\psi}\).

	In the above local expressions we have
	\begin{align}
		\begin{split}
			\MoveEqLeftR
			\left({\left(\dual{\sigma}\right)}^{-1}\otimes \id_{\phi^*a_t^*\tangent{\ProjectiveSpace{n}}}\right)\left(\id_{\dual{S}_C}\otimes \phi^*\differential{a_t}\right)\psi(z_1) \\
			={}& \sum_{q=0}^{2d-1}\psi^{ab}_q \frac{t_b}{t_a}\sqrt[2d]{\frac{t_a}{t_b}} {z_1}^q \xi_t^*\left(\dual{s}_1\otimes\phi^*\partial_{x_b^a}\right) \\
				&+ \sum_{m\neq a,b}\sum_{q=0}^{d-1}\psi^m_q \frac{t_b}{t_m}\sqrt[2d]{\frac{t_a}{t_b}} {z_1}^q \xi_t^*\left(\dual{s}_1\otimes \phi^*\partial_{x_b^m}\right),
		\end{split}\displaybreak[0]\\
		\begin{split}
			\MoveEqLeftR
			\left({\left(\dual{\sigma}\right)}^{-1}\otimes \id_{\phi^*a_t^*\tangent{\ProjectiveSpace{n}}}\right)\left(\id_{\dual{S}_C}\otimes \phi^*\differential{a_t}\right)\psi(z_2) \\
			={}& -\sum_{q=0}^{2d-1}\psi^{ab}_q \frac{t_a}{t_b}\sqrt[2d]{\frac{t_b}{t_a}} {\left(-z_2\right)}^{2d-1-q} \xi_t^*\left(\dual{s}_2\otimes\phi^\partial_{x_a^b}\right)\\
				&+ \sum_{m\neq a,b}\sum_{q=0}^{d-1}\psi^m_q \frac{t_a}{t_m}\sqrt[2d]{\frac{t_b}{t_a}} {\left(-z_2\right)}^{d-1-q} \xi_t^*\left(\dual{s}_2\otimes \phi^*\partial_{x_a^m}\right),
		\end{split}
	\end{align}
	and hence
	\begin{align}
		\begin{split}
			\tilde{\psi}(z_1)
			={}& \sum_{q=0}^{2d-1}\psi^{ab}_q t_b^{\frac{2d-2q-1}{2d}}t_a^{-\frac{2d-2q-1}{2d}} {z_1}^q \dual{s}_1\otimes\phi^*\partial_{x_b^a}\\
				&+ \sum_{m\neq a,b}\sum_{q=0}^{d-1}\psi^m_q t_m^{-1}t_a^{-\frac{2q-1}{2d}}t_b^{\frac{2d-2q-1}{2d}} {z_1}^q \dual{s}_1\otimes \phi^*\partial_{x_b^m},
		\end{split}\displaybreak[0]\\
		\begin{split}
			\tilde{\psi}(z_2)
			={}& -\sum_{q=0}^{2d-1}\psi^{ab}_q t_b^{\frac{2d-2q-1}{2d}}t_a^{-\frac{2d-2q-1}{2d}} {\left(-z_2\right)}^{2d-1-q} \dual{s}_2\otimes\phi^*\partial_{x_a^b} \\
				&+ \sum_{m\neq a,b}\sum_{q=0}^{d-1}\psi^m_q t_m^{-1}t_a^{-\frac{2q-1}{2d}}t_b^{\frac{2d-2q-1}{2d}} {\left(-z_2\right)}^{d-1-q} \dual{s}_2\otimes \phi^*\partial_{x_a^m}.
		\end{split}
	\end{align}
	Consequently, the equivariant Euler class of \(H^0(\dual{S}_C\otimes \phi^*\tangent{\ProjectiveSpace{n}})\) is given by
	\begin{equation}
		\begin{split}
			e^{K\times T}\left(H^0(\dual{S}_C\otimes \phi^*\tangent{\ProjectiveSpace{n}})\right)
			={}& \prod_{0\leq q\leq 2d-1}\left(\kappa + \frac{2d-2q-1}{2d}\tau_a - \frac{2d-2q-1}{2d}\tau_b\right)\\
				&\cdot\prod_{\substack{0\leq m \leq n\\m\neq a,b}}\prod_{0\leq q\leq d-1}\left(\kappa + \frac{2q-1}{2d}\tau_a - \frac{2d-2q-1}{2d} \tau_b + \tau_m\right).
		\end{split}
	\end{equation}
	The claim for \(\Gamma=\Gamma^{2,d}_{a,b,\Set{1}}\) now follows from \( e^{K\times T}(i_\Gamma^*\overline{N}_{2,d}) = e^{K\times T}\left(H^0(\dual{S}_C\otimes \phi^*\tangent{\ProjectiveSpace{n}})\right)\) because the contributions of \(H^0(S_C)\) and \(p_i^*S_C\) in~\eqref{eq:EqEulerClassN2dasProduct} cancel by~\eqref{eq:EqEulerClasspiS} and~\eqref{eq:EqEulerClassH0S}.

	The case \(\Gamma=\Gamma^{2,d}_{a,b,\Set{2}}\) can be treated as above while reversing the roles of \(p_1\) and \(p_2\).
\end{proof}

\begin{prop}\label{prop:EqEulerClassN1d1}
	Let \(\Gamma=\Gamma^{1,d}_{a,b,\Set{1}}\).
	Then \(M_\Gamma\) is a point and the equivariant Euler class is given by
	\begin{equation}
		\begin{split}
			e^{K\times T}(i_\Gamma^*\overline{N}_{1,d})
			={}& \prod_{\substack{0\leq q \leq 2d-1\\q\neq d}}\left(\kappa + \frac{2d-2q-1}{2d}\tau_a - \frac{2d-2q-1}{2d}\tau_b\right)\\
				&\cdot\prod_{\substack{0\leq m \leq n\\m\neq a,b}}\prod_{0\leq q\leq d-1}\left(\kappa + \frac{2q-1}{2d}\tau_a - \frac{2d-2q-1}{2d} \tau_b + \tau_m\right).
		\end{split}
	\end{equation}
\end{prop}
\begin{proof}
	In this case, the moduli space \(M_\Gamma\) consists of the point \((p_1, \phi)\) where \(p_1=0\) is the only marked point on \(C=\ProjectiveSpace{1}\) and \(\phi\) is the \(d\)-fold cover of the coordinate line with coordinates \(X^a\) and \(X^b\) as in the proof of~\ref{prop:EqEulerClassN2d}.
	In this case the SUSY normal bundle is defined by
	\begin{diag}\matrix[mat](m) {
			0 & H^0\left(S_C\right) & p_1^*S_C\oplus H^0\left(\dual{S}_C\otimes \phi^*\tangent{\ProjectiveSpace{n}}\right) & N_{(C,\phi)} & 0 \\
		} ;
		\path[pf] {
			(m-1-1) edge (m-1-2)
			(m-1-2) edge (m-1-3)
			(m-1-3) edge (m-1-4)
			(m-1-4) edge (m-1-5)
		};
	\end{diag}
	and the equivariant Euler class of \(N_{(C,\phi)}\) is given by
	\begin{equation}
		e^{K\times T}(i_\Gamma^*\overline{N}_{(C,\phi)})
		= \frac{e^{K\times T}(p_1^*S_C)\cup e^{K\times T}\left(H^0\left(\dual{S}_C\otimes \phi^*\tangent{\ProjectiveSpace{n}}\right)\right)}{e^{K\times T}(H^0(S_C))}.
	\end{equation}
	The \(K\)- and \(T\) action on the bundles is as described in the proof of Proposition~\ref{prop:EqEulerClassN2d}.
	In particular from~\eqref{eq:EqEulerClasspiS} and~\eqref{eq:EqEulerClassH0S} we obtain
	\begin{equation}
		\frac{e^{K\times T}(p_1^*S_C)}{e^{K\times T}(H^0(S_C))}
		= \frac1{\kappa - \frac1{2d}\tau_a + \frac1{2d}\tau_b}
	\end{equation}
	which cancels the first factor with \(q=d\) in \(e^{K\times T}\left(H^0\left(\dual{S}_C\otimes \phi^*\tangent{\ProjectiveSpace{n}}\right)\right)\) and the claim follows.
\end{proof}

In the same way, using \(p_1=\infty\) instead of \(p_1=0\), one can prove:
\begin{prop}\label{prop:EqEulerClassN1d0}
	Let \(\Gamma=\Gamma^{1,d}_{a,b,\emptyset}\).
	Then \(M_\Gamma\) is a point and the equivariant Euler class is given by
	\begin{equation}
		\begin{split}
			e^{K\times T}(i_\Gamma^*\overline{N}_{1,d})
			={}& \prod_{\substack{0\leq q \leq 2d-1\\q\neq d-1}}\left(\kappa + \frac{2d-2q-1}{2d}\tau_a - \frac{2d-2q-1}{2d}\tau_b\right)\\
				&\cdot\prod_{\substack{0\leq m \leq n\\m\neq a,b}}\prod_{0\leq q\leq d-1}\left(\kappa + \frac{2q-1}{2d}\tau_a - \frac{2d-2q-1}{2d} \tau_b + \tau_m\right).
		\end{split}
	\end{equation}
\end{prop}

\subsection{One-point Super Gromov--Witten Invariants of \texorpdfstring{\(\ProjectiveSpace{n}\)}{Pn} of degree one}\label{SSec:OnePointSGWofPnDeg1}
In this section we calculate one point Super Gromov--Witten invariants of \(\ProjectiveSpace{n}\) of degree one.
In this case, \(d_{1,1}=2n-1\) and \(r_{1,1} = n\).
The fixed points of the \(T\)-action on \(\ModuliStackMaps{0,1}{\ProjectiveSpace{n},d=1}\) are labeled by the graphs \(\Gamma^{1,1}_{a,b,A}\) where \(0\leq a<b\leq n\) and \(A\subset\Set{1}\), that is, \(A=\emptyset\) or \(A=\Set{1}\).
Consequently, by~\eqref{eq:SGWviaLocalization}:
\begin{equation}
	\begin{split}
		\MoveEqLeft
		\left<SGW_{0,1}^{\ProjectiveSpace{n},d=1}\right>(\alpha)\\
		&= \sum_{0\leq a< b\leq n} \sum_{A\subset\Set{1}} \int_{M_{\Gamma^{1,1}_{a,b,A}}}\left.\left(
			\frac{i_{\Gamma^{1,1}_{a,b,A}}^*\ev_1^*\alpha}{e^{K\times T}\left(i^*_{\Gamma^{1,1}_{a,b,A}} \overline{N}_{1,1}\right) e^T(N_{\Gamma^{1,1}_{a,b,A}})}
			\right)\right|_{d_p = -(3n-1)+\frac12\deg\alpha}
	\end{split}
\end{equation}
As all moduli spaces \(M_{\Gamma^{1,1}_{a,b,A}}\) are points, taking the integral is a trivial operation.
By linearity, we can assume that \(\alpha=\Lambda^j\) for some \(j=0,1, \dotsc, n\).

Note that \(\ev_1\circ i_{\Gamma^{1,1}_{a,b,A}}=q_a\) if \(1\in A\) and \(\ev_1\circ i_{\Gamma^{1,1}_{a,b,A}}=q_b\) if \(j\not\in A\).
As \(q_a\) and \(q_b\) are the fixed points of the \(T\)-action, we have that the pullback of the hyperplane class \(\Lambda\in H^2(\ProjectiveSpace{n}, \Z)\) is given by \(q_j^*\Lambda=\tau_j\), see~\cite[Equation~(9.5)]{CK-MSAG}.

From Proposition~\ref{prop:EqEulerClassN1d0} and~\ref{prop:EqEulerClassN1d1} we have
\begin{align}
	e^{K\times T}(i_{\Gamma^{1,1}_{a,b,\emptyset}}^*\overline{N}_{1,1})
	&= \prod_{\substack{0\leq m \leq n\\m\neq a}}\left(\kappa - \frac12\tau_a - \frac12 \tau_b + \tau_m\right),\\
	e^{K\times T}(i_{\Gamma^{1,1}_{a,b,\Set{1}}}^*\overline{N}_{1,1})
	&= \prod_{\substack{0\leq m \leq n\\m\neq b}}\left(\kappa - \frac12\tau_a - \frac12 \tau_b + \tau_m\right).
\end{align}
The inverse is calculated as follows:
\begin{equation}
	\begin{split}
		{e^{K\times T}(i_{\Gamma^{1,1}_{a,b,\emptyset}}^*\overline{N}_{1,1})}^{-1}
		&= \frac1{\kappa^{n}}\prod_{\substack{0\leq m \leq n\\m\neq a}}\frac1{1-\frac{\frac12\tau_a + \frac12\tau_b - \tau_m}{\kappa}} \\
		&= \frac1{\kappa^{n}}\prod_{\substack{0\leq m \leq n\\m\neq a}}\sum_{l\geq 0} {\left(\frac{\frac12\tau_a + \frac12\tau_b - \tau_m}{\kappa}\right)}^l \\
		&= \sum_{l\geq 0}\frac1{\kappa^{n+l}}\sum_{\substack{l_0\geq0, \dotsc,l_a=0,\dotsc l_n\geq0\\l_0+\dotsb+l_n=l}}\prod_{m} {\left(\frac12\tau_a + \frac12\tau_b - \tau_m\right)}^{l_m}
	\end{split}
\end{equation}
and analogously
\begin{equation}
	{e^{K\times T}(i_{\Gamma^{1,1}_{a,b,\Set{1}}}^*\overline{N}_{1,1})}^{-1}
	= \sum_{l\geq 0}\frac1{\kappa^{n+l}}\sum_{\substack{l_0\geq0, \dotsc,l_b=0,\dotsc l_n\geq0\\l_0+\dotsb+l_n=l}}\prod_{m} {\left(\frac12\tau_a + \frac12\tau_b - \tau_m\right)}^{l_m}.
\end{equation}

The equivariant Euler class of the normal bundle \(N_{\Gamma^{1,1}_{a,b,A}}\) can be calculated by the methods developed in~\cite{K-ERCVTA}.
We use the reformulation in~\cite[Theorem~9.2.1]{CK-MSAG} where \(e^T(N_{\Gamma^{1,1}_{a,b,A}})\) is given as a product:
\begin{align}
	e^F_{\Gamma^{1,1}_{a,b,\emptyset}}
		&= \frac{1}{\prod_{j\neq a}\left(\tau_a - \tau_j\right)\prod_{j\neq b}\left(\tau_b - \tau_j\right)} \\
	e^v_{\Gamma^{1,1}_{a,b,\emptyset}}
		&= \frac{\prod_{j\neq a}\left(\tau_a - \tau_j\right) \prod_{j\neq b}\left(\tau_b - \tau_j\right)}{\tau_a - \tau_b} \\
	e^e_{\Gamma^{1,1}_{a,b,\emptyset}}
		&= -{\left(\tau_a - \tau_b\right)}^2\prod_{j\neq a,b}\left(\tau_a - \tau_j\right)\left(\tau_b - \tau_j\right) \\
	e^T(N_{\Gamma^{1,1}_{a,b,\emptyset}})
		&= e^F_{\Gamma^{1,1}_{a,b,\emptyset}}e^v_{\Gamma^{1,1}_{a,b,\emptyset}}e^e_{\Gamma^{1,1}_{a,b,\emptyset}}
		= \left(\tau_b - \tau_a\right)\prod_{j\neq a,b}\left(\tau_a - \tau_j\right)\left(\tau_b - \tau_j\right)
\end{align}
and, in the same way
\begin{equation}
	e^T(N_{\Gamma^{1,1}_{a,b,\Set{1}}})
	= \left(\tau_a - \tau_b\right)\prod_{j\neq a,b}\left(\tau_a - \tau_j\right)\left(\tau_b - \tau_j\right).
\end{equation}

\begin{ex}[One point Super Gromov--Witten invariants of \(\ProjectiveSpace{1}\) of degree one]
	Let us first rewrite the above formulas for \(n=1\) and \(a=0\), \(b=1\):
	\begin{align}
		{e^{K\times T}(i_{\Gamma^{1,1}_{0,1,\emptyset}}^*\overline{N}_{1,1})}^{-1}
			&= \sum_{l\geq 0}\frac1{2^l\kappa^{l+1}} {\left(\tau_0 - \tau_1\right)}^l &
		e^T(N_{\Gamma^{1,1}_{a,b,\emptyset}})
			&= \left(\tau_1 - \tau_0\right) \\
		{e^{K\times T}(i_{\Gamma^{1,1}_{0,1,\Set{1}}}^*\overline{N}_{1,1})}^{-1}
			&= \sum_{l\geq 0}\frac1{2^l\kappa^{l+1}}{\left(\tau_1 - \tau_0\right)}^l &
		e^T(N_{\Gamma^{1,1}_{a,b,\Set{1}}})
			&= \left(\tau_0 - \tau_1\right)
	\end{align}
	It follows:
	\begin{equation}
		\begin{split}
			\left<SGW_{0,1}^{\ProjectiveSpace{1},d=1}\right>(\lambda)
			={} &\int_{M_{\Gamma^{1,1}_{0,1,\emptyset}}}\left.\left(\frac{i_{\Gamma^{1,1}_{0,1,\emptyset}}^*\ev_1^*\lambda}{e^{K\times T}(i_{\Gamma^{1,1}_{0,1,\emptyset}}^*\overline{N}_{1,1}) e^T(N_{\Gamma^{1,1}_{0,1,\emptyset}})}\right)\right|_{d_p = -1}\\
					&+ \int_{M_{\Gamma^{1,1}_{0,1,\Set{1}}}}\left.\left(\frac{i_{\Gamma^{1,1}_{0,1,\Set{1}}}^*\ev_1^*\lambda}{e^{K\times T}(i_{\Gamma^{1,1}_{0,1,\Set{1}}}^*\overline{N}_{1,1}) e^T(N_{\Gamma^{1,1}_{0,1,\Set{1}}})}\right)\right|_{d_p = -1} \\
			={} & \frac{\tau_1}{\kappa\left(\tau_1-\tau_0\right)} + \frac{\tau_0}{\kappa\left(\tau_0 - \tau_1\right)} \\
			={} &\frac1\kappa
		\end{split}
	\end{equation}
	By the Extension Axiom~\ref{axiom:Extension}, we have
	\begin{equation}
		\left<SGW_{0,1}^{\ProjectiveSpace{1},d=1}\right>(\lambda)
		= \frac1\kappa \left<GW_{0,1}^{\ProjectiveSpace{1},d=1}\right>(\lambda).
	\end{equation}

	The first super Gromov--Witten invariant which does not arise from a classical Gromov--Witten invariant is:
	\begin{equation}
		\begin{split}
			\left<SGW_{0,1}^{\ProjectiveSpace{1},d=1}\right>(1)
			={} &\int_{M_{\Gamma^{1,1}_{0,1,\emptyset}}}\left.\left(\frac{1}{e^{K\times T}(i_{\Gamma^{1,1}_{0,1,\emptyset}}^*\overline{N}_{1,1}) e^T(N_{\Gamma^{1,1}_{0,1,\emptyset}})}\right)\right|_{d_p = -2}\\
				&+ \int_{M_{\Gamma^{1,1}_{0,1,\Set{1}}}}\left.\left(\frac{1}{e^{K\times T}(i_{\Gamma^{1,1}_{0,1,\Set{1}}}^*\overline{N}_{1,1}) e^T(N_{\Gamma^{1,1}_{0,1,\Set{1}}})}\right)\right|_{d_p = -2} \\
			={} & \frac{\tau_0-\tau_1}{2\kappa^2\left(\tau_1-\tau_0\right)} + \frac{\tau_1-\tau_0}{2\kappa^2\left(\tau_0 - \tau_1\right)} \\
			={} &-\frac1{\kappa^2}
		\end{split}
	\end{equation}
\end{ex}
\begin{ex}[One-point Super Gromov--Witten invariants of \(\ProjectiveSpace{n}\) of degree one for \(n=2,3,4,5\)]
	We have evaluated the formulas from this section using the computer algebra program \enquote{Sage}.
	Note that all of the following super Gromov--Witten invariants are non-classical, that is, are not obtained from classical Gromov--Witten invariants via the Extension Axiom~\ref{axiom:Extension} for degree reasons.
	We obtain for \(\ProjectiveSpace{2}\):
	\begin{align}
		\left<SGW_{0,1}^{\ProjectiveSpace{2},d=1}\right>(\Lambda^2) &= \frac{2}{\kappa^3},&
		\left<SGW_{0,1}^{\ProjectiveSpace{2},d=1}\right>(\Lambda) &= \frac{3}{4\kappa^4},&
		\left<SGW_{0,1}^{\ProjectiveSpace{2},d=1}\right>(1) &= -\frac{3}{2\kappa^5}.
	\end{align}
	For \(\ProjectiveSpace{3}\):
	\begin{align}
		\left<SGW_{0,1}^{\ProjectiveSpace{3},d=1}\right>(\Lambda^3) &= \frac{7}{2\kappa^5},&
		\left<SGW_{0,1}^{\ProjectiveSpace{3},d=1}\right>(\Lambda^2) &= \frac{5}{\kappa^6},\\
		\left<SGW_{0,1}^{\ProjectiveSpace{3},d=1}\right>(\Lambda) &= \frac{15}{8\kappa^7},&
		\left<SGW_{0,1}^{\ProjectiveSpace{3},d=1}\right>(1) &= -\frac{35}{8\kappa^8}.
	\end{align}
	For \(\ProjectiveSpace{4}\):
	\begin{align}
		\left<SGW_{0,1}^{\ProjectiveSpace{4},d=1}\right>(\Lambda^4) &= \frac{25}{4\kappa^7},&
		\left<SGW_{0,1}^{\ProjectiveSpace{4},d=1}\right>(\Lambda^3) &= \frac{245}{16\kappa^8},\\
		\left<SGW_{0,1}^{\ProjectiveSpace{4},d=1}\right>(\Lambda^2) &= \frac{35}{2\kappa^9},&
		\left<SGW_{0,1}^{\ProjectiveSpace{4},d=1}\right>(\Lambda) &= \frac{105}{16\kappa^{10}},\\
		\left<SGW_{0,1}^{\ProjectiveSpace{4},d=1}\right>(1) &= -\frac{525}{32\kappa^{11}}.
	\end{align}
	For \(\ProjectiveSpace{5}\):
	\begin{align}
		\left<SGW_{0,1}^{\ProjectiveSpace{5},d=1}\right>(\Lambda^5) &= \frac{91}{8\kappa^9},&
		\left<SGW_{0,1}^{\ProjectiveSpace{5},d=1}\right>(\Lambda^4) &= \frac{315}{8\kappa^{10}},\\
		\left<SGW_{0,1}^{\ProjectiveSpace{5},d=1}\right>(\Lambda^3) &= \frac{2205}{32\kappa^{11}},&
		\left<SGW_{0,1}^{\ProjectiveSpace{5},d=1}\right>(\Lambda^2) &= \frac{1155}{16\kappa^{12}},\\
		\left<SGW_{0,1}^{\ProjectiveSpace{5},d=1}\right>(\Lambda) &= \frac{3465}{128\kappa^{12}},&
		\left<SGW_{0,1}^{\ProjectiveSpace{5},d=1}\right>(1) &= -\frac{9009}{128\kappa^{12}}.
	\end{align}
\end{ex}

\subsection{Two-point Super Gromov--Witten Invariants of \texorpdfstring{\(\ProjectiveSpace{n}\)}{Pn} of degree one}
In this section we calculate two-point super Gromov--Witten invariants of \(\ProjectiveSpace{n}\) of degree one using the same methods as in the previous Section~\ref{SSec:OnePointSGWofPnDeg1}.
In this case \(d_{2,1} = 2n\) and \(r_{2,1} = n+1\).
Fixed points of the \(T\)-action are labeled by the graphs \(\Gamma^{2,1}_{a,b, A}\) for \(0\leq a,< b\leq n\) and \(A\subset \Set{1,2}\).
Consequently, by~\eqref{eq:SGWviaLocalization}:
\begin{equation}
	\begin{split}
		\MoveEqLeft
		\left<SGW_{0,2}^{\ProjectiveSpace{n},d=1}\right>(\alpha_1, \alpha_2)\\
		&= \sum_{0\leq a< b\leq n} \sum_{A\subset\Set{1,2}} \int_{M_{\Gamma^{2,1}_{a,b,A}}}\left.\left(
			\frac{i_{\Gamma^{2,1}_{a,b,A}}^*\ev_1^*\alpha_1 \cup i_{\Gamma^{2,1}_{a,b,A}}^*\ev_2^*\alpha_2}{e^{K\times T}\left(i^*_{\Gamma^{2,1}_{a,b,A}} \overline{N}_{2,1}\right) e^T(N_{\Gamma^{2,1}_{a,b,A}})}
			\right)\right|_{d_p = -(3n+1)+\frac12\deg\alpha}
	\end{split}
\end{equation}
Again, all moduli spaces \(M_{\Gamma^{2,1}_{a,b,A}}\) are points.
But, for \(A=\emptyset\) and \(A=\Set{1,2}\) the stable maps that are represented by \(M_{\Gamma^{2,1}_{a,b,A}}\) consist of two irreducible components.
When the two marked points are mapped to the same fixed point of the \(T\)-action they must lie on a separate component which is contracted under the stable map and hence
\begin{equation}
	M_{\Gamma^{2,1}_{a,b,\emptyset}}
	= \ModuliStackCurves{0,3}\times M_{\Gamma^{1,1}_{a,b,\emptyset}}
\end{equation}
where the three marked points arise as the two marked points of \(M_{\Gamma^{2,1}_{a,b,\emptyset}}\) and a preimage of a nodal point under gluing.
The one marked point of \(M_{\Gamma^{1,1}_{a,b,\emptyset}}\) is the other preimage of the nodal point.
By the Splitting Axiom~\ref{axiom:SplittingAxiom} together with Proposition~\ref{prop:SGWPoint} and Proposition~\ref{prop:EqEulerClassN1d0}:
\begin{equation}
	\begin{split}
		e^{K\times T}(i_{\Gamma^{2,1}_{a,b,\emptyset}}^*\overline{N}_{2,1})
		&= e^K(\overline{N}_3) e^{K\times T}(i_{\Gamma^{1,1}_{a,b,\emptyset}}^*\overline{N}_{1,1})\\
		&= \kappa \prod_{\substack{0\leq m \leq n\\m\neq a}}\left(\kappa - \frac12\tau_a - \frac12 \tau_b + \tau_m\right)
	\end{split}
\end{equation}
The inverse is
\begin{equation}{e^{K\times T}(i_{\Gamma^{2,1}_{a,b,\emptyset}}^*\overline{N}_{2,1})}^{-1}
	= \sum_{l\geq 0}\frac{1}{\kappa^{n+1+l}}\sum_{\substack{l_0\geq0, \dotsc,l_a=0,\dotsc l_n\geq0\\l_0+\dotsb+l_n=l}}\prod_{m} {\left(\frac12\tau_a + \frac12\tau_b - \tau_m\right)}^{l_m}.
\end{equation}

The equivariant Euler class of the normal bundle \(N_{\Gamma^{2,1}_{a,b,\emptyset}}\) is calculated by~\cite[Theorem~9.2.1]{CK-MSAG} as
\begin{align}
	e^F_{\Gamma^{2,1}_{a,b,\emptyset}} &= \frac{\tau_b - \tau_a}{\prod_{j\neq a}\left(\tau_a - \tau_j\right)\prod_{j\neq b}\left(\tau_b - \tau_j\right)} \\
	e^v_{\Gamma^{2,1}_{a,b, \emptyset}} &= \frac{\prod_{j\neq a}\left(\tau_a - \tau_j\right) \prod_{j\neq b}\left(\tau_b - \tau_j\right)}{\tau_a - \tau_b} \\
	e^e_{\Gamma^{2,1}_{a,b,\emptyset}} &= -{\left(\tau_a - \tau_b\right)}^2\prod_{j\neq a,b}\left(\tau_a - \tau_j\right)\left(\tau_b - \tau_j\right) \\
	e^T(N_{\Gamma^{2,1}_{a,b,\emptyset}})
	&= e^F_{\Gamma^{2,1}_{a,b,\emptyset}}e^v_{\Gamma^{2,1}_{a,b,\emptyset}}e^e_{\Gamma^{2,1}_{a,b,\emptyset}}
	= -\prod_{j\neq a}\left(\tau_a - \tau_j\right) \prod_{j\neq b}\left(\tau_b - \tau_j\right)
\end{align}

The case of the full set \(A=\Set{1,2}\) is handled in the same way and yields:
\begin{align}
	{e^{K\times T}(i_{\Gamma^{2,1}_{a,b,\Set{1,2}}}^*\overline{N}_{2,1})}^{-1}
	&= \sum_{l\geq 0}\frac{1}{\kappa^{n+1+l}}\sum_{\substack{l_0\geq0, \dotsc,l_b=0,\dotsc l_n\geq0\\l_0+\dotsb+l_n=l}}\prod_{m} {\left(\frac12\tau_a + \frac12\tau_b - \tau_m\right)}^{l_m}, \\
	e^T(N_{\Gamma^{2,1}_{a,b,\Set{1,2}}})
	&= e^T(N_{\Gamma^{2,1}_{a,b,\emptyset}})
	= -\prod_{j\neq a}\left(\tau_a - \tau_j\right) \prod_{j\neq b}\left(\tau_b - \tau_j\right).
\end{align}

When the set \(A=\Set{A_1}\) has a single element, the curve of the stable map representing \(\Gamma^{2,1}_{a,b,A}\) consist of a single irreducible component with two marked points.
That is, the equivariant Euler class of the SUSY normal bundle can be obtained from Proposition~\ref{prop:EqEulerClassN2d} as
\begin{equation}
	e^{K\times T}\left(i_{\Gamma^{2,1}_{a,b,\Set{A_1}}}^*\overline{N}_{2,1}\right)\\
	= \prod_{0\leq m \leq n}\left(\kappa - \frac12\tau_a - \frac12 \tau_b + \tau_m\right),
\end{equation}
and hence
\begin{equation}{e^{K\times T}(i_{\Gamma^{2,1}_{a,b,\Set{A_1}}}^*\overline{N}_{2,1})}^{-1}
	= \sum_{l\geq 0}\frac{1}{\kappa^{n+1+l}}\sum_{\substack{l_0\geq0, \dotsc, l_n\geq0\\l_0+\dotsb+l_n=l}}\prod_{m} {\left(\frac12\tau_a + \frac12\tau_b - \tau_m\right)}^{l_m}.
\end{equation}
The equivariant Euler class of the normal bundle \(N_\Gamma\) is:
\begin{align}
	e^F_{\Gamma^{2,1}_{a,b,\Set{A_1}}} &= \frac{1}{\prod_{j\neq a}\left(\tau_a - \tau_j\right)\prod_{j\neq b}\left(\tau_b - \tau_j\right)} \\
	e^v_{\Gamma^{2,1}_{a,b, \Set{A_1}}} &= \prod_{j\neq a}\left(\tau_a - \tau_j\right) \prod_{j\neq b}\left(\tau_b - \tau_j\right) \\
	e^e_{\Gamma^{2,1}_{a,b,\Set{A_1}}} &= -{\left(\tau_a - \tau_b\right)}^2\prod_{j\neq a,b}\left(\tau_a - \tau_j\right)\left(\tau_b - \tau_j\right) \\
	e^T(N_{\Gamma^{2,1}_{a,b,\Set{A_1}}})
	&= e^F_{\Gamma^{2,1}_{a,b,\Set{A_1}}}e^v_{\Gamma^{2,1}_{a,b,\Set{A_1}}}e^e_{\Gamma^{2,1}_{a,b,\Set{A_1}}}
	= \prod_{j\neq a}\left(\tau_a - \tau_j\right)\prod_{j\neq b}\left(\tau_b - \tau_j\right)
\end{align}

\begin{ex}
	Specializing the results of this section to \(n=1\), \(a=0\) and \(b=1\) yields:
	\begin{equation}
		\begin{aligned}
			{e^{K\times T}(i_{\Gamma^{2,1}_{0,1,\emptyset}}^*\overline{N}_{2,1})}^{-1}
				&= \sum_{l\geq 0}\frac{1}{2^l\kappa^{l+2}}{\left(\tau_0 - \tau_1\right)}^l&
			{e^T(N_{\Gamma^{2,1}_{a,b,\emptyset}})}^{-1}
				&= \frac{1}{{\left(\tau_1 - \tau_0\right)}^2}\\
			{e^{K\times T}(i_{\Gamma^{2,1}_{0,1,\Set{A_1}}}^*\overline{N}_{2,1})}^{-1}
				&= \sum_{l\geq 0}\frac{1}{2^{2l}\kappa^{2l+2}}{\left(\tau_1 - \tau_0\right)}^{2l}&
			{e^T(N_{\Gamma^{2,1}_{0,1,\Set{A_1}}})}^{-1}
				&= -\frac1{{\left(\tau_1-\tau_0\right)}^2} \\
			{e^{K\times T}(i_{\Gamma^{2,1}_{0,1,\Set{1,2}}}^*\overline{N}_{2,1})}^{-1}
				&= \sum_{l\geq 0}\frac{1}{2^l\kappa^{l+2}}{\left(\tau_1 - \tau_0\right)}^l&
			{e^T(N_{\Gamma^{2,1}_{a,b,\Set{1,2}}})}^{-1}
				&= \frac{1}{{\left(\tau_1 - \tau_0\right)}^2}
		\end{aligned}
	\end{equation}
	\begin{equation}
		\begin{split}
			\left<SGW_{0,2}^{\ProjectiveSpace{1},d=1}\right>(\lambda,\lambda)
			&= \sum_{A\subset\Set{1,2}} \int_{M_{\Gamma^{2,1}_{0,1,A}}}\left.\left(\frac{i_{\Gamma^{2,1}_{0,1,A}}^*\ev_1^*\lambda \cdot i_{\Gamma^{2,1}_{0,1,A}}^*\ev_2^*\lambda}{e^{K\times T}(i_{\Gamma^{2,1}_{0,1,A}}^*\overline{N}_{2,1}) e^T(N_{\Gamma^{2,1}_{0,1,A}})}\right)\right|_{d_p = -2}\\
			&= \frac1{\kappa^2}\left(\frac{\tau_1^2}{{\left(\tau_1 - \tau_0\right)}^2} - 2\frac{\tau_1\tau_0}{{\left(\tau_1 - \tau_0\right)}^2} + \frac{\tau_0^2}{{\left(\tau_1 - \tau_0\right)}^2}\right)\\
			&= \frac1{\kappa^2}
		\end{split}
	\end{equation}
	\begin{equation}
		\begin{split}
			\left<SGW_{0,2}^{\ProjectiveSpace{1},d=1}\right>(\lambda,1)
			&= \sum_{A\subset\Set{1,2}} \int_{M_{\Gamma^{2,1}_{0,1,A}}}\left.\left(\frac{i_{\Gamma^{2,1}_{0,1,A}}^*\ev_1^*\lambda}{e^{K\times T}(i_{\Gamma^{2,1}_{0,1,A}}^*\overline{N}_{2,1}) e^T(N_{\Gamma^{2,1}_{0,1,A}})}\right)\right|_{d_p = -3}\\
			&= \frac1{2\kappa^3}\left(\frac{\tau_1\left(\tau_0 - \tau_1\right)}{{\left(\tau_1 - \tau_0\right)}^2} + \frac{\tau_0\left(\tau_1 - \tau_0\right)}{{\left(\tau_1 - \tau_0\right)}^2}\right)\\
			&= -\frac1{2\kappa^3}
		\end{split}
	\end{equation}
	\begin{equation}
		\begin{split}
			\left<SGW_{0,2}^{\ProjectiveSpace{1},d=1}\right>(1,1)
			&= \sum_{A\subset\Set{1,2}} \int_{M_{\Gamma^{2,1}_{0,1,A}}}\left.\left(\frac{1}{e^{K\times T}(i_{\Gamma^{2,1}_{0,1,A}}^*\overline{N}_{2,1}) e^T(N_{\Gamma^{2,1}_{0,1,A}})}\right)\right|_{d_p = -4}\\
			&= \frac1{4\kappa^4}\left(\frac{{\left(\tau_0 - \tau_1\right)}^2}{{\left(\tau_1 - \tau_0\right)}^2} - \frac{{\left(\tau_1 - \tau_0\right)}^2}{{\left(\tau_1 - \tau_0\right)}^2} - \frac{{\left(\tau_1 - \tau_0\right)}^2}{{\left(\tau_1 - \tau_0\right)}^2} + \frac{{\left(\tau_1 - \tau_0\right)}^2}{{\left(\tau_1 - \tau_0\right)}^2}\right)\\
			&= 0
		\end{split}
	\end{equation}
\end{ex}
Using the above formulas and computer algebra we have also calculated the following examples:
\begin{ex}[Two-point Super Gromov--Witten invariants of \(\ProjectiveSpace{2}\) of degree one]
	\begin{align}
		\left<SGW_{0,2}^{\ProjectiveSpace{2},d=1}\right>(\Lambda^2, \Lambda^2) &= \frac{1}{\kappa^3},&
		\left<SGW_{0,2}^{\ProjectiveSpace{2},d=1}\right>(\Lambda^2, \Lambda) &= \frac{3}{2\kappa^4},\\
		\left<SGW_{0,2}^{\ProjectiveSpace{2},d=1}\right>(\Lambda, \Lambda) &= \frac{3}{4\kappa^5},&
		\left<SGW_{0,2}^{\ProjectiveSpace{2},d=1}\right>(\Lambda^2, 1) &= -\frac{3}{4\kappa^5},\\
		\left<SGW_{0,2}^{\ProjectiveSpace{2},d=1}\right>(\Lambda, 1) &= -\frac{3}{4\kappa^6},&
		\left<SGW_{0,2}^{\ProjectiveSpace{2},d=1}\right>(1, 1) &= 0,&
	\end{align}
\end{ex}
\begin{ex}[Two-point Super Gromov--Witten invariants of \(\ProjectiveSpace{3}\) of degree one]
	\begin{align}\allowdisplaybreaks[1]
		\left<SGW_{0,2}^{\ProjectiveSpace{3},d=1}\right>(\Lambda^3, \Lambda^3) &= \frac{1}{\kappa^4}\\
		\left<SGW_{0,2}^{\ProjectiveSpace{3},d=1}\right>(\Lambda^3, \Lambda^2) &= \frac{2}{\kappa^5}\\
		\left<SGW_{0,2}^{\ProjectiveSpace{3},d=1}\right>(\Lambda^3, \Lambda) &= \frac{5}{2\kappa^6}&
		\left<SGW_{0,2}^{\ProjectiveSpace{3},d=1}\right>(\Lambda^2, \Lambda^2) &= \frac{5}{\kappa^6}\\
		\left<SGW_{0,2}^{\ProjectiveSpace{3},d=1}\right>(\Lambda^3, 1) &= -\frac{5}{4\kappa^7}&
		\left<SGW_{0,2}^{\ProjectiveSpace{3},d=1}\right>(\Lambda^2, \Lambda) &= -\frac{15}{4\kappa^7}\\
		\left<SGW_{0,2}^{\ProjectiveSpace{3},d=1}\right>(\Lambda^2, 1) &= -\frac{5}{2\kappa^8}&
		\left<SGW_{0,2}^{\ProjectiveSpace{3},d=1}\right>(\Lambda, \Lambda) &= \frac{15}{8\kappa^8}\\
		\left<SGW_{0,2}^{\ProjectiveSpace{3},d=1}\right>(\Lambda, 1) &= -\frac{35}{16\kappa^9}\\
		\left<SGW_{0,2}^{\ProjectiveSpace{3},d=1}\right>(1, 1) &= 0
	\end{align}
\end{ex}
\begin{ex}[Two-point Super Gromov--Witten invariants of \(\ProjectiveSpace{4}\) of degree one]
	\begin{align}\allowdisplaybreaks[1]
		\left<SGW_{0,2}^{\ProjectiveSpace{4},d=1}\right>(\Lambda^4, \Lambda^4) &= \frac{1}{\kappa^5}\\
		\left<SGW_{0,2}^{\ProjectiveSpace{4},d=1}\right>(\Lambda^4, \Lambda^3) &= \frac{5}{2\kappa^6}\\
		\left<SGW_{0,2}^{\ProjectiveSpace{4},d=1}\right>(\Lambda^4, \Lambda^2) &= \frac{15}{4\kappa^7}&
		\left<SGW_{0,2}^{\ProjectiveSpace{4},d=1}\right>(\Lambda^3, \Lambda^3) &= \frac{15}{2\kappa^7}\\
		\left<SGW_{0,2}^{\ProjectiveSpace{4},d=1}\right>(\Lambda^4, \Lambda) &= \frac{35}{8\kappa^8}&
		\left<SGW_{0,2}^{\ProjectiveSpace{4},d=1}\right>(\Lambda^3, \Lambda^2) &= \frac{105}{8\kappa^8}\\
		\left<SGW_{0,2}^{\ProjectiveSpace{4},d=1}\right>(\Lambda^4, 1) &= -\frac{35}{16\kappa^9}&
		\left<SGW_{0,2}^{\ProjectiveSpace{4},d=1}\right>(\Lambda^3, \Lambda) &= \frac{175}{16\kappa^9}\\
		\left<SGW_{0,2}^{\ProjectiveSpace{4},d=1}\right>(\Lambda^2, \Lambda^2) &= \frac{315}{16\kappa^9}\\
		\left<SGW_{0,2}^{\ProjectiveSpace{4},d=1}\right>(\Lambda^3, 1) &= -\frac{105}{16\kappa^{10}}&
		\left<SGW_{0,2}^{\ProjectiveSpace{4},d=1}\right>(\Lambda^2, \Lambda) &= \frac{105}{8\kappa^{10}}\\
		\left<SGW_{0,2}^{\ProjectiveSpace{4},d=1}\right>(\Lambda^2, 1) &= -\frac{315}{32\kappa^{11}}&
		\left<SGW_{0,2}^{\ProjectiveSpace{4},d=1}\right>(\Lambda, \Lambda) &= -\frac{525}{64\kappa^{11}}\\
		\left<SGW_{0,2}^{\ProjectiveSpace{4},d=1}\right>(\Lambda, 1) &= -\frac{35}{16\kappa^{12}}\\
		\left<SGW_{0,2}^{\ProjectiveSpace{4},d=1}\right>(1, 1) &= 0
	\end{align}
\end{ex}
\begin{ex}[Two-point Super Gromov--Witten invariants of \(\ProjectiveSpace{5}\) of degree one]
	\begin{align}\allowdisplaybreaks[1]
		\left<SGW_{0,2}^{\ProjectiveSpace{5},d=1}\right>(\Lambda^5, \Lambda^5) &= \frac{1}{\kappa^6}\\
		\left<SGW_{0,2}^{\ProjectiveSpace{5},d=1}\right>(\Lambda^5, \Lambda^4) &= \frac{3}{\kappa^7}\\
		\left<SGW_{0,2}^{\ProjectiveSpace{5},d=1}\right>(\Lambda^5, \Lambda^3) &= \frac{21}{4\kappa^8}&
		\left<SGW_{0,2}^{\ProjectiveSpace{5},d=1}\right>(\Lambda^4, \Lambda^4) &= \frac{21}{2\kappa^8}\\
		\left<SGW_{0,2}^{\ProjectiveSpace{5},d=1}\right>(\Lambda^5, \Lambda^2) &= \frac{7}{\kappa^9}&
		\left<SGW_{0,2}^{\ProjectiveSpace{5},d=1}\right>(\Lambda^4, \Lambda^3) &= \frac{21}{\kappa^9}\\
		\left<SGW_{0,2}^{\ProjectiveSpace{5},d=1}\right>(\Lambda^5, \Lambda) &= \frac{63}{8\kappa^{10}}&
		\left<SGW_{0,2}^{\ProjectiveSpace{5},d=1}\right>(\Lambda^4, \Lambda^2) &= \frac{63}{2\kappa^{10}}\\
		\left<SGW_{0,2}^{\ProjectiveSpace{5},d=1}\right>(\Lambda^3, \Lambda^3) &= \frac{189}{4\kappa^{10}}\\
		\left<SGW_{0,2}^{\ProjectiveSpace{5},d=1}\right>(\Lambda^5, 1) &= -\frac{63}{16\kappa^{11}}&
		\left<SGW_{0,2}^{\ProjectiveSpace{5},d=1}\right>(\Lambda^4, \Lambda) &= \frac{441}{16\kappa^{11}}\\
		\left<SGW_{0,2}^{\ProjectiveSpace{5},d=1}\right>(\Lambda^3, \Lambda^2) &= \frac{1071}{16\kappa^{11}}\\
		\left<SGW_{0,2}^{\ProjectiveSpace{5},d=1}\right>(\Lambda^4, 1) &= -\frac{63}{4\kappa^{12}}&
		\left<SGW_{0,2}^{\ProjectiveSpace{5},d=1}\right>(\Lambda^3, \Lambda) &= \frac{1575}{316\kappa^{12}}\\
		\left<SGW_{0,2}^{\ProjectiveSpace{5},d=1}\right>(\Lambda^2, \Lambda^2) &= \frac{1365}{16\kappa^{12}}\displaybreak[0]\\
		\left<SGW_{0,2}^{\ProjectiveSpace{5},d=1}\right>(\Lambda^3, 1) &= -\frac{2079}{64\kappa^{13}}&
		\left<SGW_{0,2}^{\ProjectiveSpace{5},d=1}\right>(\Lambda^2, \Lambda) &= \frac{3465}{64\kappa^{13}}\\
		\left<SGW_{0,2}^{\ProjectiveSpace{5},d=1}\right>(\Lambda^2, 1) &= -\frac{693}{16\kappa^{14}}&
		\left<SGW_{0,2}^{\ProjectiveSpace{5},d=1}\right>(\Lambda, \Lambda) &= -\frac{3465}{128\kappa^{14}}\\
		\left<SGW_{0,2}^{\ProjectiveSpace{5},d=1}\right>(\Lambda, 1) &= -\frac{9009}{2566\kappa^{15}}\\
		\left<SGW_{0,2}^{\ProjectiveSpace{5},d=1}\right>(1, 1) &= 0
	\end{align}
\end{ex}

\subsection{Three-point Super Gromov--Witten Invariants of \texorpdfstring{\(\ProjectiveSpace{n}\)}{Pn} of degree one}\label{SSec:ThreePointSGWofPnDeg1}
In this section we calculate three point Super Gromov--Witten invariants of \(\ProjectiveSpace{n}\) of degree one.
In this case \(d_{3,1} = 2n+1\) and \(r_{3,1} = n+2\).
The fixed points of the \(T\) action are labeled by the graphs \(\Gamma^{3,1}_{a,b,A}\) where \(A\subset\Set{1,2,3}\) and \(0\leq a<b\leq n\).
\begin{equation}
	\begin{split}
		\MoveEqLeftR
		\left<SGW_{0,3}^{\ProjectiveSpace{n},d=1}\right>(\alpha_1, \alpha_2, \alpha_3)
		= \sum_{0\leq a< b\leq n} \sum_{A\subset\Set{1,2,3}} \\
		&\int_{M_{\Gamma^{3,1}_{a,b,A}}}\left.\left(
			\frac{i_{\Gamma^{3,1}_{a,b,A}}^*\ev_1^*\alpha_1 \cup i_{\Gamma^{3,1}_{a,b,A}}^*\ev_2^*\alpha_2 \cup i_{\Gamma^{3,1}_{a,b,A}}^*\ev_3^*\alpha_3}{e^{K\times T}\left(i^*_{\Gamma^{3,1}_{a,b,A}} \overline{N}_{3,1}\right) e^T(N_{\Gamma^{3,1}_{a,b,A}})}
			\right)\right|_{d_p = -(3n+3)+\frac12\deg\alpha}
	\end{split}
\end{equation}

If the cardinality \(\#A\) of \(A\) is zero or three the moduli space \(M_{\Gamma^{3,1}_{a,b,A}}\) is isomorphic to \(\ModuliStackCurves{0,4}\), otherwise the moduli space is a point.
The equivariant Euler classes of the SUSY normal bundles also depend on the cardinality of \(A\).

Let us first treat the case \(A=\emptyset\).
In this case, all four three marked points are mapped to the fixed point \(q_b\in \ProjectiveSpace{n}\).
Consequently, the marked points must lie on a different component than the one that is mapped to the coordinate line through \(q_a\) and \(q_b\).
The component that is contracted contains four special points and the non-contracted component one.
Hence
\begin{equation}
	M_{\Gamma^{3,1}_{a,b,\emptyset}}
	= \ModuliStackCurves{0,4}\times M_{\Gamma^{1,1}_{a,b,\emptyset}}
	= \ProjectiveSpace{1}.
\end{equation}

It follows from the splitting principle together with Proposition~\ref{prop:SGWPoint} and Proposition~\ref{prop:EqEulerClassN1d0} that
\begin{equation}
	\begin{split}
		e^{K\times T}(i_{\Gamma^{3,1}_{a,b,\emptyset}}^*\overline{N}_{3,1})
		&= e^K(\overline{N}_4) e^{K\times T}(i_{\Gamma^{1,1}_{a,b,\emptyset}}^*\overline{N}_{1,1})\\
		&= \kappa\left(\kappa-\frac12\lambda\right) \prod_{\substack{0\leq m \leq n\\m\neq a}}\left(\kappa - \frac12\tau_a - \frac12 \tau_b + \tau_m\right)
	\end{split}
\end{equation}
As the inverse of \(\kappa\left(\kappa - \frac12\lambda\right)\) is given by \(\kappa^{-3}\left(\kappa+\frac12\lambda\right)\), the inverse of the equivariant Euler class is given by
\begin{equation}\label{eq:IEqEulerClassSUSYGamma31ab0}
	{e^{K\times T}(i_{\Gamma^{3,1}_{a,b,\emptyset}}^*\overline{N}_{3,1})}^{-1}
	= \sum_{l\geq 0}\frac{\kappa + \frac12\lambda}{\kappa^{n+3+l}}\sum_{\substack{l_0\geq0, \dotsc,l_a=0,\dotsc l_n\geq0\\l_0+\dotsb+l_n=l}}\prod_{m} {\left(\frac12\tau_a + \frac12\tau_b - \tau_m\right)}^{l_m}.
\end{equation}
By~\cite[Theorem~9.2.1]{CK-MSAG}, the equivariant Euler class of the normal bundle \(N_{\Gamma^{3,1}_{a,b,\emptyset}}\) is given by
\begin{align}
	e^F_{\Gamma^{3,1}_{a,b,\emptyset}} &= \frac{\left(\tau_b - \tau_a - \lambda\right)}{\prod_{j\neq a}\left(\tau_a - \tau_j\right)\prod_{j\neq b}\left(\tau_b - \tau_j\right)} \\
	e^v_{\Gamma^{3,1}_{a,b, \emptyset}} &= \frac{\prod_{j\neq a}\left(\tau_a - \tau_j\right) \prod_{j\neq b}\left(\tau_b - \tau_j\right)}{\left(\tau_a-\tau_b\right)} \\
	e^e_{\Gamma^{3,1}_{a,b,\emptyset}} &= -{\left(\tau_a - \tau_b\right)}^2\prod_{j\neq a,b}\left(\tau_a - \tau_j\right)\left(\tau_b - \tau_j\right) \\
	e^T(N_{\Gamma^{3,1}_{a,b,\emptyset}})
	&= e^F_{\Gamma^{3,1}_{a,b,\emptyset}}e^v_{\Gamma^{3,1}_{a,b,\emptyset}}e^e_{\Gamma^{3,1}_{a,b,\emptyset}}
	= \frac{\left(\tau_b - \tau_a - \lambda\right)\prod_{j\neq a}\left(\tau_a - \tau_j\right) \prod_{j\neq b}\left(\tau_b - \tau_j\right)}{\left( \tau_a - \tau_b\right)}
\end{align}
Note that the class \(e_b=\psi_b\) in~\cite[Theorem~9.2.1]{CK-MSAG} is equal to \(\lambda\) in the case at hand \(\ModuliStackCurves{0,4}=\ProjectiveSpace{1}\).
The inverse is calculated as follows:
\begin{equation}\label{eq:IEqEulerClassNGamma31ab0}
	\begin{split}
		{e^T(N_{\Gamma^{3,1}_{a,b,\emptyset}})}^{-1}
		&= - \frac1{1-\frac{\lambda}{\tau_b - \tau_a}}\frac1{\prod_{j\neq a}\left(\tau_a - \tau_j\right) \prod_{j\neq b}\left(\tau_b - \tau_j\right)} \\
		&= - \frac{1+\frac{\lambda}{\tau_b - \tau_a}}{\prod_{j\neq a}\left(\tau_a - \tau_j\right) \prod_{j\neq b}\left(\tau_b - \tau_j\right)}\\
		&= \frac{\tau_a - \tau_b - \lambda}{\left(\tau_b - \tau_a\right)\prod_{j\neq a}\left(\tau_a - \tau_j\right) \prod_{j\neq b}\left(\tau_b - \tau_j\right)}
	\end{split}
\end{equation}

The case \(A=\Set{1,2,3}\) is very similar to the case \(A=\emptyset\) and one obtains
\begin{equation}\label{eq:IEqEulerClassSUSYGamma31ab3}
	{e^{K\times T}(i_{\Gamma^{3,1}_{a,b,\Set{1,2,3}}}^*\overline{N}_{3,1})}^{-1}
	= \sum_{l\geq 0}\frac{\kappa + \frac12\lambda}{\kappa^{n+3+l}}\sum_{\substack{l_0\geq0, \dotsc,l_b=0,\dotsc l_n\geq0\\l_0+\dotsb+l_n=l}}\prod_{m} {\left(\frac12\tau_a + \frac12\tau_b - \tau_m\right)}^{l_m}
\end{equation}
and
\begin{equation}\label{eq:IEqEulerClassNGamma31ab3}
	\begin{split}
		{e^T(N_{\Gamma^{3,1}_{a,b,\Set{1,2,3}}})}^{-1}
		&= \frac{\tau_b - \tau_a - \lambda}{\left(\tau_a - \tau_b\right)\prod_{j\neq a}\left(\tau_a - \tau_j\right) \prod_{j\neq b}\left(\tau_b - \tau_j\right)}.
	\end{split}
\end{equation}

In the case that \(A=\Set{A_1}\), that is \(A\) has cardinality \(\#A=1\), the marked point with label \(A_1\) is mapped to \(q_a\) and the other two marked points with labels \(\Set{1,2,3}\setminus\Set{A_1}\) are mapped to the marked point \(q_b\).
Consequently, the two marked points must lie on a different component of the prestable curve than the one that is mapped to the coordinate line through \(q_a\) and \(q_b\).
The contracted component has three special points and the non-contracted two and the resulting moduli space is a point:
\begin{equation}
	M_{\Gamma^{3,1}_{a,b,\Set{A_1}}}
	= \ModuliStackCurves{0,3}\times M_{\Gamma^{1,1}_{a,b,\Set{A_1}}}
\end{equation}

It follows from the Splitting Axiom~\ref{axiom:SplittingAxiom} together with Proposition~\ref{prop:SGWPoint} and Proposition~\ref{prop:EqEulerClassN2d} that
\begin{equation}
	\begin{split}
		e^{K\times T}\left(i^*_{\Gamma^{3,1}_{a,b,\Set{A_1}}}\overline{N}_{3,1}\right)
		&= e^{K}(\overline{N}_{3})e^{K\times T}\left(i_{\Gamma^{2,1}_{a,b,\Set{A_1}}}^*\overline{N}_{2,1}\right)\\
		&= \kappa  \prod_{0\leq m \leq n}\left(\kappa - \frac12\tau_a - \frac12 \tau_b + \tau_m\right)
	\end{split}
\end{equation}
and hence
\begin{equation}\label{eq:IEqEulerClassSUSYGamma31ab1}
	{e^{K\times T}(i_{\Gamma^{3,1}_{a,b,\Set{A_1}}}^*\overline{N}_{3,1})}^{-1}
	= \sum_{l\geq 0}\frac{1}{\kappa^{n+2+l}}\sum_{\substack{l_0\geq0, \dotsc, l_n\geq0\\l_0+\dotsb+l_n=l}}\prod_{m} {\left(\frac12\tau_a + \frac12\tau_b - \tau_m\right)}^{l_m}.
\end{equation}
The equivariant Euler class of the normal bundle \(N_{\Gamma^{3,1}_{a,b,\Set{A_1}}}\) can be calculated by~\cite[Theorem~9.2.1]{CK-MSAG} as
\begin{align}
	e^F_{\Gamma^{3,1}_{a,b,\Set{A_1}}} &= \frac{\left(\tau_b - \tau_a\right)}{\prod_{j\neq a}\left(\tau_a - \tau_j\right)\prod_{j\neq b}\left(\tau_b - \tau_j\right)} \\
	e^v_{\Gamma^{3,1}_{a,b, \Set{A_1}}} &= \prod_{j\neq a}\left(\tau_a - \tau_j\right) \prod_{j\neq b}\left(\tau_b - \tau_j\right) \\
	e^e_{\Gamma^{3,1}_{a,b,\Set{A_1}}} &= -{\left(\tau_a - \tau_b\right)}^2\prod_{j\neq a,b}\left(\tau_a - \tau_j\right)\left(\tau_b - \tau_j\right) \\
	e^T(N_{\Gamma^{3,1}_{a,b,\Set{A_1}}})
	&= e^F_{\Gamma^{3,1}_{a,b,\Set{A_1}}}e^v_{\Gamma^{3,1}_{a,b,\Set{A_1}}}e^e_{\Gamma^{3,1}_{a,b,\Set{A_1}}}
	= \left(\tau_b - \tau_a\right)\prod_{j\neq a}\left(\tau_a - \tau_j\right)\prod_{j\neq b}\left(\tau_b - \tau_j\right)\label{eq:EqEulerClassNGamma31ab1}
\end{align}

The case \(A=\Set{A_1,A_2}\) is analogous and one obtains
\begin{equation}\label{eq:IEqEulerClassSUSYGamma31ab2}
	{e^{K\times T}(i_{\Gamma^{3,1}_{a,b,\Set{A_1,A_2}}}^*\overline{N}_{3,1})}^{-1}
	= \sum_{l\geq 0}\frac{1}{\kappa^{n+2+l}}\sum_{\substack{l_0\geq0, \dotsc, l_n\geq0\\l_0+\dotsb+l_n=l}}\prod_{m} {\left(\frac12\tau_a + \frac12\tau_b - \tau_m\right)}^{l_m}
\end{equation}
and
\begin{equation}\label{eq:EqEulerClassNGamma31ab2}
	e^T(N_{\Gamma^{3,1}_{a,b,\Set{A_1,A_2}}})
	= \left(\tau_a - \tau_b\right)\prod_{j\neq a}\left(\tau_a - \tau_j\right)\prod_{j\neq b}\left(\tau_b - \tau_j\right).
\end{equation}

In a first step we want to calculate the super Gromov--Witten invariants of \(\ProjectiveSpace{1}\) for stable maps of degree one and three marked points.
In this case, we have \(d_{3,1}=3\) and \(r_{3,1}=3\).
We call the number \(d_{3,1}-2\deg\alpha\) codegree.

\begin{ex}[Codegree zero, \(\deg\alpha = d_{3,1}\)]
	The first example is the case of codegree zero, that is \(\alpha_j=\lambda\) for \(j=1,2,3\).
	\begin{equation}
		\begin{split}
			\MoveEqLeft
			\left<SGW_{0,3}^{\ProjectiveSpace{1},d=1}\right>(\lambda,\lambda,\lambda) \\
			&= \sum_{A\subset\Set{1,2,3}} \int_{M_{\Gamma^{3,1}_{0,1,A}}}\left.\left(\frac{i_{\Gamma^{3,1}_{0,1,A}}^*\ev_1^*\lambda \cdot i_{\Gamma^{3,1}_{0,1,A}}^*\ev_2^*\lambda \cdot i_{\Gamma^{3,1}_{0,1,A}}^*\ev_3^*\lambda}{e^{K\times T}(i_{\Gamma^{3,1}_{0,1,A}}^*\overline{N}_{3,1}) e^T(N_{\Gamma^{3,1}_{0,1,A}})}\right)\right|_{d_p = -3}
		\end{split}
	\end{equation}
	Specializing the results of this section to \(n=1\), \(a=0\) and \(b=0\) yields:
	\begin{equation}\label{eq:IEqEulerClasses31P1}
		\begin{aligned}
			{e^{K\times T}(i_{\Gamma^{3,1}_{0,1,\emptyset}}^*\overline{N}_{3,1})}^{-1}
				&= \sum_{l\geq 0}\frac{2\kappa + \lambda}{2^{l+1}\kappa^{l+4}}{\left(\tau_0 - \tau_1\right)}^l&
			{e^T(N_{\Gamma^{3,1}_{a,b,\emptyset}})}^{-1}
				&= \frac{\lambda+\tau_1-\tau_0}{{\left(\tau_1 - \tau_0\right)}^3}\\
			{e^{K\times T}(i_{\Gamma^{3,1}_{0,1,\Set{A_1}}}^*\overline{N}_{3,1})}^{-1}
				&= \sum_{l\geq 0}\frac{1}{2^{2l}\kappa^{2l+3}}{\left(\tau_1 - \tau_0\right)}^{2l}&
			{e^T(N_{\Gamma^{3,1}_{0,1,\Set{A_1}}})}^{-1}
				&= -\frac1{{\left(\tau_1-\tau_0\right)}^3} \\
			{e^{K\times T}(i_{\Gamma^{3,1}_{0,1,\Set{A_1,A_2}}}^*\overline{N}_{3,1})}^{-1}
				&= \sum_{l\geq 0}\frac{1}{2^{2l}\kappa^{2l+3}}{\left(\tau_1 - \tau_0\right)}^{2l}&
			{e^T(N_{\Gamma^{3,1}_{0,1,\Set{A_1,A_2}}})}^{-1}
				&= \frac1{{\left(\tau_1-\tau_0\right)}^3} \\
			{e^{K\times T}(i_{\Gamma^{3,1}_{0,1,\Set{1,2,3}}}^*\overline{N}_{3,1})}^{-1}
				&= \sum_{l\geq 0}\frac{2\kappa + \lambda}{2^{l+1}\kappa^{l+4}}{\left(\tau_1 - \tau_0\right)}^l&
			{e^T(N_{\Gamma^{3,1}_{a,b,\Set{1,2,3}}})}^{-1}
				&= \frac{\tau_1-\tau_0-\lambda}{{\left(\tau_1 - \tau_0\right)}^3}
		\end{aligned}
	\end{equation}
	In particular, the summand of \(e^{K\times T}(i_{\Gamma_{0,1,A}}^*\overline{N}_{3,1})\) that is of polynomial degree \(d_p=-3\) is given by \(\kappa^{-3}\) for all \(A\).

	We will evaluate the integral for the different sets separately.
	First, in the case where \(A=\emptyset\), we use Equations~\eqref{eq:IEqEulerClasses31P1} to obtain:
	\begin{equation}
		\int_{M_{\Gamma^{3,1}_{0,1,\emptyset}}}\left.\left(\frac{i_{\Gamma^{3,1}_{0,1,\emptyset}}^*\ev_1^*\lambda \cdot i_{\Gamma^{3,1}_{0,1,\emptyset}}^*\ev_2^*\lambda \cdot i_{\Gamma^{3,1}_{0,1,\emptyset}}^*\ev_3^*\lambda}{e^{K\times T}(i_{\Gamma^{3,1}_{0,1,\emptyset}}^*\overline{N}_{3,1}) e^T(N_{\Gamma^{3,1}_{0,1,\emptyset}})}\right)\right|_{d_p = -3} \\
		= \kappa^{-3}\frac{\tau_1^3}{{\left(\tau_1-\tau_0\right)}^3}
	\end{equation}
	Similarly,for \(A=\Set{A_1}\):
	\begin{equation}
		\int_{M_{\Gamma^{3,1}_{0,1,\Set{A_1}}}}\left.\left(\frac{i_{\Gamma^{3,1}_{0,1,\Set{A_1}}}^*\ev_1^*\lambda \cdot i_{\Gamma^{3,1}_{0,1,\Set{A_1}}}^*\ev_2^*\lambda \cdot i_{\Gamma^{3,1}_{0,1,\Set{A_1}}}^*\ev_3^*\lambda}{e^{K\times T}(i_{\Gamma^{3,1}_{0,1,\Set{A_1}}}^*\overline{N}_{3,1}) e^T(N_{\Gamma^{3,1}_{0,1,\Set{A_1}}})}\right)\right|_{d_p = -3} \\
		= -\kappa^{-3}\frac{\tau_0\tau_1^2}{{\left(\tau_1-\tau_0\right)}^3}
	\end{equation}
	and for the two-element set \(A=\Set{A_1, A_2}\):
	\begin{equation}
		\begin{split}
			\MoveEqLeft
			\int_{M_{\Gamma^{3,1}_{0,1,\Set{A_1, A_2}}}}\left.\left(\frac{i_{\Gamma^{3,1}_{0,1,\Set{A_1, A_2}}}^*\ev_1^*\lambda \cdot i_{\Gamma^{3,1}_{0,1,\Set{A_1, A_2}}}^*\ev_2^*\lambda \cdot i_{\Gamma^{3,1}_{0,1,\Set{A_1, A_2}}}^*\ev_3^*\lambda}{e^{K\times T}(i_{\Gamma^{3,1}_{0,1,\Set{A_1, A_2}}}^*\overline{N}_{3,1}) e^T(N_{\Gamma^{3,1}_{0,1,\Set{A_1, A_2}}})}\right)\right|_{d_p = -3} \\
			&= \kappa^{-3}\frac{\tau_0^2\tau_1}{{\left(\tau_1-\tau_0\right)}^3}
		\end{split}
	\end{equation}
	The case of the full set \(A=\Set{1,2,3}\) is:
	\begin{equation}
		\begin{split}
			\MoveEqLeft
			\int_{M_{\Gamma^{3,1}_{0,1,\Set{1,2,3}}}}\left.\left(\frac{i_{\Gamma^{3,1}_{0,1,\Set{1,2,3}}}^*\ev_1^*\lambda \cdot i_{\Gamma^{3,1}_{0,1,\Set{1,2,3}}}^*\ev_2^*\lambda \cdot i_{\Gamma^{3,1}_{0,1,\Set{1,2,3}}}^*\ev_3^*\lambda}{e^{K\times T}(i_{\Gamma^{3,1}_{0,1,\Set{1,2,3}}}^*\overline{N}_{3,1}) e^T(N_{\Gamma^{3,1}_{0,1,\Set{1,2,3}}})}\right)\right|_{d_p = -3} \\
			&= -\kappa^{-3}\frac{\tau_0^3}{{\left(\tau_1-\tau_0\right)}^3}
		\end{split}
	\end{equation}
	Summing up:
	\begin{equation}
		\begin{split}
			\MoveEqLeft
			\left<SGW_{0,3}^{\ProjectiveSpace{1},d=1}\right>(\lambda,\lambda,\lambda) \\
			&= \kappa^{-3}\frac{\tau_1^3}{{\left(\tau_1-\tau_0\right)}^3} - \kappa^{-3}\frac{3\tau_0\tau_1^2}{{\left(\tau_1-\tau_0\right)}^3} + \kappa^{-3}\frac{3\tau_0^2\tau_1}{{\left(\tau_1-\tau_0\right)}^3} - \kappa^{-3}\frac{\tau_0^3}{{\left(\tau_1-\tau_0\right)}^3} \\
			&= \kappa^{-3}
		\end{split}
	\end{equation}
	This result is expected by the Extension Axiom~\ref{axiom:Extension}.
	Because the \(2d_{3,1}=\deg\alpha\)
	\begin{equation}
		\left<SGW_{0,3}^{\ProjectiveSpace{1},d=1}\right>(\lambda,\lambda,\lambda)
		= \kappa^{-3}\left<GW_{0,3}^{\ProjectiveSpace{1},d=1}\right>(\lambda,\lambda,\lambda)
	\end{equation}
	where the only data not present in Gromov--Witten invariants is the power of \(\kappa\) in the prefactor given by \(-r_{3,1}=-3\).
	\(\left<GW_{0,3}^{\ProjectiveSpace{1},d=1}\right>(\lambda,\lambda,\lambda)=1\) because there is a unique automorphism of \(\ProjectiveSpace{1}\) that fixes three points.
\end{ex}

\begin{ex}[Codegree two, \(2d_{3,1}-\deg\alpha=2\)]
	The case where \(2d_{3,1}-\deg \alpha = 2\) involves also characteristic classes of the normal bundle \(\overline{N}_{3,1}\).
	\begin{equation}
		\left<SGW_{0,3}^{\ProjectiveSpace{1},d=1}\right>(\lambda,\lambda,1)
		= \sum_{A\subset\Set{1,2,3}} \int_{M_{\Gamma^{3,1}_{0,1,A}}}\left.\left(\frac{i_{\Gamma^{3,1}_{0,1,A}}^*\ev_1^*\lambda \cdot i_{\Gamma^{3,1}_{0,1,A}}^*\ev_2^*\lambda }{e^{K\times T}(i_{\Gamma^{3,1}_{0,1,A}}^*\overline{N}_{3,1}) e^T(N_{\Gamma^{3,1}_{0,1,A}})}\right)\right|_{d_p = -4}
	\end{equation}
	We again calculate the integral with the help of Equations~\eqref{eq:IEqEulerClasses31P1}.
	For the empty set \(A=\emptyset\) we have:
	\begin{equation}
		\begin{split}
			\MoveEqLeft
		\int_{M_{\Gamma^{3,1}_{0,1,\emptyset}}}\left.\left(\frac{i_{\Gamma^{3,1}_{0,1,\emptyset}}^*\ev_1^*\lambda \cdot i_{\Gamma^{3,1}_{0,1,\emptyset}}^*\ev_2^*\lambda }{e^{K\times T}(i_{\Gamma^{3,1}_{0,1,\emptyset}}^*\overline{N}_{3,1}) e^T(N_{\Gamma^{3,1}_{0,1,\emptyset}})}\right)\right|_{d_p = -4} \\
			={}& \int_{\ProjectiveSpace{1}}\frac12 \left(\lambda + \tau_0 - \tau_1\right)\frac{\lambda + \tau_1 - \tau_0}{{\left(\tau_1 - \tau_0\right)}^3}
			= 0
		\end{split}
	\end{equation}
	The contributions of the one-element and two-element sets vanish because the coefficient of \(-\kappa^4\) of the equivariant Euler class of the SUSY normal bundles vanishes by Equations~\eqref{eq:IEqEulerClasses31P1}.
	The case of the full set \(A=\Set{1,2,3}\) is treated analogously to the case of the empty set:
	\begin{equation}
			\int_{M_{\Gamma^{3,1}_{0,1,\Set{1,2,3}}}}\left.\left(\frac{i_{\Gamma^{3,1}_{0,1,\Set{1,2,3}}}^*\ev_1^*\lambda \cdot i_{\Gamma^{3,1}_{0,1,\Set{1,2,3}}}^*\ev_2^*\lambda }{e^{K\times T}(i_{\Gamma^{3,1}_{0,1,\Set{1,2,3}}}^*\overline{N}_{3,1}) e^T(N_{\Gamma^{3,1}_{0,1,\Set{1,2,3}}})}\right)\right|_{d_p = -4}
				= 0
	\end{equation}
	Summing up:
	\begin{equation}
		\left<SGW_{0,3}^{\ProjectiveSpace{1},d=1}\right>(\lambda,\lambda,1)
		= 0
	\end{equation}
\end{ex}
\begin{ex}[Codegree four, \(2d_{3,1}-\deg\alpha=4\)]
	The case of codegree four is treated in the same way.
	First we calculate the summands using Equations~\eqref{eq:IEqEulerClasses31P1}:
	\begin{equation}
		\begin{split}
			\MoveEqLeft
			\int_{M_{\Gamma^{3,1}_{0,1,\emptyset}}}\left.\left(\frac{i_{\Gamma^{3,1}_{0,1,\emptyset}}^*\ev_1^*\lambda}{e^{K\times T}(i_{\Gamma^{3,1}_{0,1,\emptyset}}^*\overline{N}_{3,1}) e^T(N_{\Gamma^{3,1}_{0,1,\emptyset}})}\right)\right|_{d_p = -5} \\
				={}& \int_{\ProjectiveSpace{1}} \frac14\left(\lambda\left(\tau_0 - \tau_1\right) + {\left(\tau_0 - \tau_1\right)}^2\right)\frac{\lambda + \tau_1 - \tau_0}{{\left(\tau_1 - \tau_0\right)}^3}
				= 0
		\end{split}
	\end{equation}
	\begin{equation}
		\int_{M_{\Gamma^{3,1}_{0,1,\Set{1}}}}\left.\left(\frac{i_{\Gamma^{3,1}_{0,1,\Set{1}}}^*\ev_1^*\lambda}{e^{K\times T}(i_{\Gamma^{3,1}_{0,1,\Set{1}}}^*\overline{N}_{3,1}) e^T(N_{\Gamma^{3,1}_{0,1,\Set{1}}})}\right)\right|_{d_p = -5} \\
		= -\kappa^{-5}\frac{\tau_0}{4\left(\tau_1 - \tau_0\right)}
	\end{equation}
	\begin{equation}
		\begin{split}
			\MoveEqLeft
			\int_{M_{\Gamma^{3,1}_{0,1,\Set{2}}}}\left.\left(\frac{i_{\Gamma^{3,1}_{0,1,\Set{2}}}^*\ev_1^*\lambda}{e^{K\times T}(i_{\Gamma^{3,1}_{0,1,\Set{2}}}^*\overline{N}_{3,1}) e^T(N_{\Gamma^{3,1}_{0,1,\Set{2}}})}\right)\right|_{d_p = -5} \\
			&= \int_{M_{\Gamma^{3,1}_{0,1,\Set{3}}}}\left.\left(\frac{i_{\Gamma^{3,1}_{0,1,\Set{3}}}^*\ev_1^*\lambda}{e^{K\times T}(i_{\Gamma^{3,1}_{0,1,\Set{3}}}^*N_{3,1}) e^T(N_{\Gamma^{3,1}_{0,1,\Set{3}}})}\right)\right|_{d_p = -5} \\
			&= -\kappa^{-5}\frac{\tau_1}{4\left(\tau_1 - \tau_0\right)}
		\end{split}
	\end{equation}
	\begin{equation}
		\int_{M_{\Gamma^{3,1}_{0,1,\Set{2,3}}}}\left.\left(\frac{i_{\Gamma^{3,1}_{0,1,\Set{2,3}}}^*\ev_1^*\lambda}{e^{K\times T}(i_{\Gamma^{3,1}_{0,1,\Set{2,3}}}^*\overline{N}_{3,1}) e^T(N_{\Gamma^{3,1}_{0,1,\Set{2,3}}})}\right)\right|_{d_p = -5} \\
		= \kappa^{-5}\frac{\tau_1}{4\left(\tau_1 - \tau_0\right)}
	\end{equation}
	\begin{equation}
		\begin{split}
			\MoveEqLeft
			\int_{M_{\Gamma^{3,1}_{0,1,\Set{1,3}}}}\left.\left(\frac{i_{\Gamma^{3,1}_{0,1,\Set{1,3}}}^*\ev_1^*\lambda}{e^{K\times T}(i_{\Gamma^{3,1}_{0,1,\Set{1,3}}}^*\overline{N}_{3,1}) e^T(N_{\Gamma^{3,1}_{0,1,\Set{1,3}}})}\right)\right|_{d_p = -5} \\
			&=\int_{M_{\Gamma^{3,1}_{0,1,\Set{1,2}}}}\left.\left(\frac{i_{\Gamma^{3,1}_{0,1,\Set{1,2}}}^*\ev_1^*\lambda}{e^{K\times T}(i_{\Gamma^{3,1}_{0,1,\Set{1,2}}}^*\overline{N}_{3,1}) e^T(N_{\Gamma^{3,1}_{0,1,\Set{1,2}}})}\right)\right|_{d_p = -5} \\
			&= \kappa^{-5}\frac{\tau_0}{4\left(\tau_1 - \tau_0\right)}
		\end{split}
	\end{equation}
	\begin{equation}
		\int_{M_{\Gamma^{3,1}_{0,1,\Set{1,2,3}}}}\left.\left(\frac{i_{\Gamma^{3,1}_{0,1,\Set{1,2,3}}}^*\ev_1^*\lambda}{e^{K\times T}(i_{\Gamma^{3,1}_{0,1,\Set{1,2,3}}}^*\overline{N}_{3,1}) e^T(N_{\Gamma^{3,1}_{0,1,\Set{1,2,3}}})}\right)\right|_{d_p = -5}
		= 0
	\end{equation}
	Summing up yields:
	\begin{equation}
		\begin{split}
			\left<SGW_{0,3}^{\ProjectiveSpace{1},d=1}\right>(\lambda,1,1)
			={}& \sum_{A\subset\Set{1,2,3}} \int_{M_{\Gamma^{3,1}_{0,1,A}}}\left.\left(\frac{i_{\Gamma^{3,1}_{0,1,A}}^*\ev_1^*\lambda}{e^{K\times T}(i_{\Gamma^{3,1}_{0,1,A}}^*\overline{N}_{3,1}) e^T(N_{\Gamma^{3,1}_{0,1,A}})}\right)\right|_{d_p = -5} \\
			={}& - \kappa^{-5}\frac{\tau_0}{4\left(\tau_1 - \tau_0\right)} - \kappa^{-5}\frac{\tau_1}{4\left(\tau_1 - \tau_0\right)} - \kappa^{-5}\frac{\tau_1}{4\left(\tau_1 - \tau_0\right)} \\
				& + \kappa^{-5}\frac{\tau_1}{4\left(\tau_1 - \tau_0\right)} + \kappa^{-5}\frac{\tau_0}{4\left(\tau_1 - \tau_0\right)} + \kappa^{-5}\frac{\tau_0}{4\left(\tau_1 - \tau_0\right)} \\
			={}& -\frac14\kappa^{-5}
		\end{split}
	\end{equation}
\end{ex}

\begin{ex}[Codegree six, \(2d_{3,1}-\deg\alpha=6\)]
	The summands for the fixed point loci \(M_{\Gamma^{3,1}_{0,1,A}}\) are:
	\begin{equation}
		\begin{split}
			\MoveEqLeft
			\int_{M_{\Gamma^{3,1}_{0,1,\emptyset}}}\left.\left(\frac{1}{e^{K\times T}(i_{\Gamma^{3,1}_{0,1,\emptyset}}^*\overline{N}_{3,1}) e^T(N_{\Gamma^{3,1}_{0,1,\emptyset}})}\right)\right|_{d_p = -6} \\
			={} & \int_{\ProjectiveSpace{1}} \frac18\left(\lambda{\left(\tau_0 - \tau_1\right)}^2 + {\left(\tau_0 - \tau_1\right)}^3\right)\frac{\lambda + \tau_1 - \tau_0}{{\left(\tau_1 - \tau_0\right)}^3}
			= 0
		\end{split}
	\end{equation}
	\begin{equation}
	\int_{M_{\Gamma^{3,1}_{0,1,\Set{A_1}}}}\left.\left(\frac{1}{e^{K\times T}(i_{\Gamma^{3,1}_{0,1,\Set{A_1}}}^*\overline{N}_{3,1}) e^T(N_{\Gamma^{3,1}_{0,1,\Set{A_1}}})}\right)\right|_{d_p = -6} \\
		= 0
	\end{equation}
	\begin{equation}
		\int_{M_{\Gamma^{3,1}_{0,1,\Set{A_1,A_2}}}}\left.\left(\frac{1}{e^{K\times T}(i_{\Gamma^{3,1}_{0,1,\Set{A_1,A_2}}}^*\overline{N}_{3,1}) e^T(N_{\Gamma^{3,1}_{0,1,\Set{A_1,A_2}}})}\right)\right|_{d_p = -6} \\
		= 0
	\end{equation}
	\begin{equation}
		\int_{M_{\Gamma^{3,1}_{0,1,\Set{1,2,3}}}}\left.\left(\frac{1}{e^{K\times T}(i_{\Gamma^{3,1}_{0,1,\Set{1,2,3}}}^*\overline{N}_{3,1}) e^T(N_{\Gamma^{3,1}_{0,1,\Set{1,2,3}}})}\right)\right|_{d_p = -6} \\
		= 0
	\end{equation}
	Summing up yields:
	\begin{equation}
		\left<SGW_{0,3}^{\ProjectiveSpace{1},d=1}\right>(1,1,1)
		=0
	\end{equation}
\end{ex}

\begin{ex}[Three-point Super Gromov--Witten invariants of \(\ProjectiveSpace{2}\) of degree one]
	\begin{align}
		\left<SGW_{0,3}^{\ProjectiveSpace{2},d=1}\right>(\Lambda^2, \Lambda^2, \Lambda) &= \frac{1}{\kappa^4}\\
		\left<SGW_{0,3}^{\ProjectiveSpace{2},d=1}\right>(\Lambda^2, \Lambda^2, 1) &= 0&
		\left<SGW_{0,3}^{\ProjectiveSpace{2},d=1}\right>(\Lambda^2, \Lambda, \Lambda) &= \frac{3}{2\kappa^5}\\
		\left<SGW_{0,3}^{\ProjectiveSpace{2},d=1}\right>(\Lambda^2, \Lambda, 1) &= 0&
		\left<SGW_{0,3}^{\ProjectiveSpace{2},d=1}\right>(\Lambda, \Lambda, \Lambda) &= \frac{3}{2\kappa^6}\\
		\left<SGW_{0,3}^{\ProjectiveSpace{2},d=1}\right>(\Lambda^2, 1, 1) &= -\frac{3}{8\kappa^7}&
		\left<SGW_{0,3}^{\ProjectiveSpace{2},d=1}\right>(\Lambda, \Lambda, 1) &= -\frac{3}{8\kappa^7}\\
		\left<SGW_{0,3}^{\ProjectiveSpace{2},d=1}\right>(\Lambda, 1, 1) &= -\frac{3}{8\kappa^7}\\
		\left<SGW_{0,3}^{\ProjectiveSpace{2},d=1}\right>(1, 1, 1) &= 0
	\end{align}
\end{ex}
\begin{ex}[Three-point Super Gromov--Witten invariants of \(\ProjectiveSpace{3}\) of degree one]
	\begin{align}
		\left<SGW_{0,3}^{\ProjectiveSpace{3},d=1}\right>(\Lambda^3, \Lambda^3, \Lambda) &= \frac{1}{\kappa^5}&
		\left<SGW_{0,3}^{\ProjectiveSpace{3},d=1}\right>(\Lambda^3, \Lambda^2, \Lambda^2) &= \frac{1}{\kappa^5}\\
		\left<SGW_{0,3}^{\ProjectiveSpace{3},d=1}\right>(\Lambda^3, \Lambda^3, 1) &= 0&
		\left<SGW_{0,3}^{\ProjectiveSpace{3},d=1}\right>(\Lambda^3, \Lambda^2, \Lambda) &= \frac{2}{\kappa^6}\\
		\left<SGW_{0,3}^{\ProjectiveSpace{3},d=1}\right>(\Lambda^3, \Lambda^2, 1) &= 0&
		\left<SGW_{0,3}^{\ProjectiveSpace{3},d=1}\right>(\Lambda^3, \Lambda, \Lambda) &= \frac{5}{2\kappa^7}\\
		\left<SGW_{0,3}^{\ProjectiveSpace{3},d=1}\right>(\Lambda^2, \Lambda^2, \Lambda) &= \frac{5}{\kappa^7}\\
		\left<SGW_{0,3}^{\ProjectiveSpace{3},d=1}\right>(\Lambda^3, \Lambda, 1) &= 0&
		\left<SGW_{0,3}^{\ProjectiveSpace{3},d=1}\right>(\Lambda^2, \Lambda^2, 1) &= 0\\
		\left<SGW_{0,3}^{\ProjectiveSpace{3},d=1}\right>(\Lambda^2, \Lambda, \Lambda) &= \frac{5}{\kappa^8}\\
		\left<SGW_{0,3}^{\ProjectiveSpace{3},d=1}\right>(\Lambda^2, \Lambda, 1) &= -\frac{5}{8\kappa^9}&
		\left<SGW_{0,3}^{\ProjectiveSpace{3},d=1}\right>(\Lambda, \Lambda, \Lambda) &= \frac{15}{4\kappa^9}\\
		\left<SGW_{0,3}^{\ProjectiveSpace{3},d=1}\right>(\Lambda^2, 1, 1) &= -\frac{5}{4\kappa^{10}}&
		\left<SGW_{0,3}^{\ProjectiveSpace{3},d=1}\right>(\Lambda, \Lambda, 1) &= -\frac{5}{4\kappa^{10}}\\
		\left<SGW_{0,3}^{\ProjectiveSpace{3},d=1}\right>(\Lambda, 1, 1) &= -\frac{35}{32\kappa^{11}}\\
		\left<SGW_{0,3}^{\ProjectiveSpace{3},d=1}\right>(1, 1, 1) &= 0
	\end{align}
\end{ex}
 
\printbibliography

\textsc{Enno Keßler\\
Max-Planck-Institut für Mathematik,
Vivatsgasse 7,
53111 Bonn,
Germany\\[.5em]
Max-Planck-Institut für Mathematik in den Naturwissenschaften,
Inselstraße 22,
04103 Leipzig,
Germany}\\
\texttt{kessler@mis.mpg.de}\\

\textsc{Artan Sheshmani\\
Beijing Institute of Mathematical Sciences and Applications,
A6, Room~205 No.~544,
Hefangkou Village,
Huaibei Town,
Huairou District,
Beijing~101408,
China\\[.5em]
Massachusetts Institute of Technology (MIT),
IAiFi Institute,
182 Memorial Drive,
Cambridge, MA 02139,
USA}\\
\texttt{artan@mit.edu}\\

\textsc{Shing-Tung Yau\\
Yau Mathematical Sciences Center,
Tsinghua University,
Haidian District,
Beijing,
China
}\\
\texttt{styau@tsinghua.edu.cn}
\end{document}